\documentclass[12pt,oneside]{amsart}
\usepackage{amsmath,amssymb,amsthm,amsfonts,amscd,soul,cite, amsxtra,amstext,mathrsfs}
\usepackage{color,enumitem,graphicx}
\usepackage[utf8]{inputenc}
\usepackage[english]{babel}
\usepackage[colorlinks=true,urlcolor=blue,
citecolor=red,linkcolor=blue,linktocpage,pdfpagelabels,
bookmarksnumbered,bookmarksopen]{hyperref}
\usepackage[english]{babel}
\usepackage{lineno}
\usepackage[hyperpageref]{backref}
\def\dis{\displaystyle}

\def\ato0{{\buildrel{\dis\longrightarrow}\over{a\to0}}}

\newtheorem{theorem}{Theorem}[section]
\newtheorem{proposition}{Proposition}[section]

\newtheorem{lemma}[theorem]{Lemma}

\newtheorem{remark}{Remark}[section]

\numberwithin{equation}{section}

\DeclareMathSymbol{\C}{\mathalpha}{AMSb}{"43}

\newtheorem*{theorem-a}{Theorem A}
\newtheorem*{theorem-b}{Theorem B}

\newcommand{\bsub}{\begin{subequations}}
\newcommand{\esub}{\end{subequations}$\!$}

\title[On a sharp inequality of Adimurthi-Druet type]{On a sharp inequality of Adimurthi-Druet type and extremal functions}
\author[J.F. de Oliveira]{Jos\'{e} Francisco de Oliveira}\thanks{First author was supported by   CNPq/Brazil grant 309491/2021-5 }
\author[J.M. do \'{O}]{Jo\~{a}o Marcos do \'{O}}\thanks{Second author was supported by  INCTmat/MCT/Brazil and CNPq grant  312340/2021-4}

\address[J.F. de Oliveira]{\newline\indent Department of Mathematics
\newline\indent 
Federal University of Piau\'{\i}
\newline\indent
 64049-550 Teresina, PI, Brazil}
\email{\href{mailto:jfoliveira@ufpi.edu.br}{jfoliveira@ufpi.edu.br}}

\address[J.M. do \'{O}]{\newline\indent Department of Mathematics
\newline\indent 
Federal University of Para\'{\i}ba
\newline\indent
58051-900 Jo\~{a}o Pessoa, PB, Brazil}
\email{\href{mailto:jmbo@pq.cnpq.br}{jmbo@pq.cnpq.br}}

\begin{document}
\maketitle
\thispagestyle{empty}
\begin{abstract}
Our main purpose in this paper is to establish the existence and nonexistence of extremal functions for sharp inequality of Adimurthi-Druet type for fractional dimensions on the entire space. Precisely, we extend the sharp Trudinger-Moser type inequality in (Calc.Var.Partial Differential Equations, \textbf{52}  (2015) 125-163) for the entire space. In addition, we perform the two-step strategy of Carleson-Chang together blow up analysis method to ensure the existence of maximizers for the associated extremal problems for both subcritical and critical  regimes. We also present a nonexistence result under subcritical regime for some special cases.
\end{abstract}
\section{Introduction}
The classical Trudinger-Moser inequality \cite{Moser, Trudinger67, Pohozaev65, Yudo} states that
\begin{equation}\label{classica-TM}
   \sup_{u\in C^{1}_{0}(\Omega),\, \int_{\Omega}\vert \nabla u\vert ^{N}dx\le 1}\int_{\Omega}e^{\mu \vert u \vert^{\frac{N}{N-1}}}dx \;\; \begin{cases}
  <\infty,\;\; \mbox{if}\;\;  \mu\le \mu_{N}\\
      =\infty,\;\; \mbox{if}\;\;  \mu>\mu_{N}
    \end{cases}
\end{equation}
where $\Omega\subset\mathbb{R}^{N}$  is a bounded domain,  $\mu_{N}=N\omega^{1/(N-1)}_{N-1}$, and $\omega_{N-1}$ is the measure of the unit sphere in $\mathbb{R}^{N}$. By using symmetrization techniques J. Moser \cite{Moser} was able to reduce \eqref{classica-TM}  to the following
\begin{equation}\label{r-classica-TM}
    \sup_{u\in C^{1}_{0,rad}(B_R),\, \int_{B_R}\vert \nabla u \vert^{N}dx\le 1}\int_{B_R}e^{\mu \vert u\vert^{\frac{N}{N-1}}}dx\;\; \left\{\begin{aligned}
    & <\infty,\;\; \mbox{if}\;\;  \mu\le \mu_{N}\\
     & =\infty,\;\; \mbox{if}\;\;  \mu>\mu_{N}
    \end{aligned}\right.
\end{equation}
where $\vert B_R\vert=\vert\Omega\vert$ and $C^{1}_{0,rad}(B_R)$ represents the set of the radially symmetric functions in $C^{1}_{0}(B_R)$. 

On the other hand,  according to the formalism in \cite{Still77,Zubair12}, the integration of radially symmetric function $f(r)$ on a $\theta$-dimensional fractional space is given by
\begin{equation}\label{fractional integral}
    \int f(r(x_0,x_1))dx_0=\omega_{\theta}\int_{0}^{\infty}r^{\theta}f(r)dr,
\end{equation}
where $r(x_0,x_1)$ is the distance between two points $x_0$ and $x_1$, and $\omega_{\theta}$
given by
\begin{equation}\label{fractional volume}
\omega_{\theta}=\frac{2\pi^{\frac{\theta}{2}}}{\Gamma(\frac{\theta}{2})},\;\;\;\mbox{with}\;\;\;\Gamma(x)=\int_0^{\infty} t^{x-1} e^{-t} \,
\mathrm{d} t.
\end{equation}

In the case that $\theta$ is a positive integer number $\omega_{\theta}$ agrees precisely with the measure of the unit sphere in for Euclidean space $\mathbb{R}^{\theta+1}$. Integration over fractional dimensional spaces is often used in the dimensional regularization method as a powerful tool to obtain results in statistical mechanics and quantum field theory \cite{Collins, Palmer, Zubair11,Zubair12}.
 Motivated by the above discussion, in \cite{JJ2012} the authors were able to establish a sharp Trudinger-Moser type inequality which extends the classical \eqref{classica-TM} for fractional dimensions. Indeed, for  $0<R<\infty$,  $\alpha\ge 1$, $\sigma\ge 0$ and $\alpha-p+1=0$,
 it is proven that
\begin{equation}\label{TM1}
\sup_{u\in AC_{loc}(0,R],\, u(R)=0,\,\int_{0}^{R}\vert u^{\prime}\vert^{p}\mathrm{d}\lambda_{\alpha}\le1}\int_{0}^{R}
e^{\mu\vert u\vert^{\frac{p}{p-1}}}\, \mathrm{d}\lambda_{\sigma}\;\;
\left\{
  \begin{array}{lll}
    <\infty & \mbox{if} & \mu\le\mu_{\alpha,\sigma}\\
    =\infty & \mbox{if} & \mu>\mu_{\alpha,\sigma}
  \end{array}
\right.
\end{equation}
where $\mu_{\alpha,\sigma}=(\sigma+1)\omega_{\alpha}^{1/\alpha}$,  $AC_{loc}(0,R]$ denotes set of all locally absolutely continuous functions on the interval $(0,R]$ and we are denoting the $\theta$-fractional measure (cf. \eqref{fractional integral}) by
\begin{equation}\label{fractional-measure}
\begin{aligned}
\int_{0}^{R}f(r)\mathrm{d}\lambda_{\theta}=\omega_{\theta}\int_{0}^{R} r^{\theta}f(r)\mathrm{d}r,\quad \theta\ge0,\quad \mbox{and}\quad 0<R\le \infty.
\end{aligned}
\end{equation} 

Despite its simple form, the inequality \eqref{TM1} hides surprises and some interesting points have been drawing attention. Firstly, in the particular case that $\alpha=\sigma=N-1$, the fractional inequality \eqref{TM1} implies that the Moser's reduction \eqref{r-classica-TM} and then \eqref{classica-TM} holds. In this sense we say that \eqref{TM1} extends \eqref{classica-TM} to weighted Sobolev spaces including fractional dimensions. Secondly, for $\alpha=N-k$ and $\sigma=N-1$ we can recover the Trudinger-Moser type inequality for the $k$-Hessian equation obtained by Tian and Wang \cite{Tian}. Further, based on \eqref{TM1}, in \cite{RUDOL}  the authors were able to investigate the existence of maximizers for that inequality obtained in \cite{Tian} and the existence of radially symmetric solutions for the $k$-Hessian equation  was obtained in \cite{OLUBrmi}. Also, for arbitrary real choices of $\alpha$ and $\sigma$,  the estimate \eqref{TM1} can be employed  to study a general class of quasi-linear elliptic operators  \cite{JJ2014,Clement-deFigueiredo-Mitidieri}. Thirdly, in \cite{JJ2014} was proved that the exponential growth is optimal in \eqref{TM1} which leads to loss of compactness in the sense of the embedding into the Orlicz-type spaces and makes the extremal problem associated to \eqref{TM1} interesting. Inspired by the results due to Adimurthi and O.~Druet \cite{ADDruet2004}, in \cite{JJ2013} the authors have obtained a sharp form of \eqref{TM1} and investigated the associated extremal problem. 

In this paper we are mainly interested in extends the results in \cite{JJ2013} for the unbounded case, when $R=\infty$. In order to state our results, let us present briefly the related weighted Sobolev space introduced by  P. Cl\'{e}ment et al. \cite{Clement-deFigueiredo-Mitidieri}. For  $0<R\le\infty$, $\theta\ge0$ and $q\ge 1$, set $L^q_{\theta}=L^q_{\theta}(0,R)$ the Lebesgue space associated with the $\theta$-fractional measure \eqref{fractional-measure} on the interval $(0,R)$. Then, we denote by $X_R=X^{1,p}_R\left(\alpha,\theta\right)$ the completion of the set  of all  functions  $u\in AC_{loc}(0,R)$ such that  $\lim_{r\rightarrow R}u(r)=0$, $u\in L^{p}_{\theta}$ and $u^{\prime}\in L^{p}_{\alpha}$ with the norm
\begin{equation}\label{Xnorm-full}
\|u\|=(\|u\|^{p}_{L^{p}_{\theta}}+\|u^{\prime}\|^{p}_{L^{p}_{\alpha}})^{\frac{1}{p}}.
\end{equation}
If $\alpha-p+1=0$ and $R=\infty$, we have the continuous embedding (See Section~\ref{Improvement-section} below)
\begin{equation}
\label{PTM-compactembeddings}
 X^{1,p}_{R}(\alpha,\theta) \hookrightarrow L^q_{\theta}\quad\mbox{for all}\quad q\in [p,\infty).
\end{equation}
Set 
\begin{equation}\label{exp-frac}
\varphi_p(t)=e^{t}-
\sum_{k=0}^{k_0-1}\frac{t^{k}}{k!}=\sum_{j\in\mathbb{N}\;:\; j\ge p-1}\frac{t^j}{j!},\; t\ge 0,
\end{equation}
with $k_0=\min\left\{j\in\mathbb{N}\;:\; p-1\le j\right\}$. In view of the \eqref{PTM-compactembeddings}, for each term of the series expansion of $\varphi_{p}(\vert u \vert^{p/(p-1)})$ belongs to $L^{1}_{\theta}$ for all $u\in X^{1,p}_{\infty}(\alpha,\theta)$. It motivates us to investigate the supremum
\begin{equation}\label{SUP}
   AD(\eta, \mu, \alpha,\theta)=\sup_{u\in X^{1,p}_{\infty}(\alpha,\theta),\;\|u\|= 1}\int_{0}^{\infty}\varphi_{p}\big(\mu(1+\eta\|u\|^{p}_{L^{p}_{\theta}})^{\frac{1}{p-1}}\vert u\vert^{\frac{p}{p-1}}\big)\mathrm{d}\lambda_{\theta}.
 \end{equation}
  
 Actually,  we are able to establish  the following sharp result:
 \begin{theorem}\label{TSUP}
 Let  $p\ge 2$ and $\alpha=p-1$ and $\theta\ge 0$. Then, 
 \begin{itemize}
 \item [$(1)$] $ AD(\eta, \mu, \alpha,\theta)<\infty$, for any $ \mu\le \mu_{\alpha,\theta}$ and $0\le \eta<1$
 \item [$(2)$] $  AD(1, \mu_{\alpha,\theta}, \alpha,\theta)=\infty.$ 
 \end{itemize}
 \end{theorem}
If $\eta=0$, Theorem~\ref{TSUP} recovers  the result in \cite[Theorem~1.1]{AbreFern} and part of  the \cite[Theorem~1.2]{JJ2021}.  In addition, Theorem~\ref{TSUP}  extends \cite[Theorem~2]{JJ2013}  for the entire space $R=\infty$.

 Maximizers for the Trudinger-Moser type inequality \eqref{TM1} and its extensions were investigated in \cite{JJ2021, AbreFern,JJ2012, JJ2013}. In this paper, we will also investigate the existence and nonexistence of extremal function for the supremum \eqref{SUP}. Our first existence results reads as follow:

 \begin{theorem}[Subcritical case] \label{thm-extremal}  Let  $p\ge 2$  and $\alpha=p-1$ and $\theta\ge \alpha$. Then the supremum $AD(\eta, \mu, \alpha,\theta)$ is attained  in the following cases:
 \begin{itemize}
 \item [$(1)$]  $p>2$,   $0\le \eta<1$ and $ 0<\mu<\mu_{\alpha,\theta}$,
 \item [$(2)$]  $p=2$, $0\le \eta<1$ and   $\frac{2(1+2\eta)}{(1+\eta)^2B_{2,\theta}}<\mu<\mu_{1,\theta}$, where
 \begin{equation}\label{B2}
\frac{1}{B_{2,\theta}}=\inf_{0\not\equiv u\in X^{1,2}_{\infty}(1,\theta)}\frac{\|u^{\prime}\|^{2}_{L^{2}_{1}}\|u\|^{2}_{L^{2}_{\theta}}}{\|u\|^{4}_{L^{4}_{\theta}}}.
\end{equation}
  \end{itemize} 
 \end{theorem}
 
If we choose $\eta =0$, Theorem~\ref{thm-extremal} recovers precisely \cite[Theorem~1.2]{AbreFern}. In addition, for the critical case $\mu=\mu_{\alpha,\theta}$ we are able to obtain the following:
 \begin{theorem}[Critical case]\label{thm-cmax} Assume $p,\alpha$ and $\theta$ under the assumption of Theorem~\ref{thm-extremal}. Then, there exists $\eta_{0}\in (0, 1)$ such that $AD(\eta, \mu_{\alpha,
 \theta}, \alpha,\theta)$ is attained for any $0\le \eta<\eta_0$.
 \end{theorem}
 We note that Theorem~\ref{thm-cmax} is new even  for $\eta=0$. Indeed, the existence of maximizers for $AD(0, \mu, \alpha,\theta)$ was recently ensured  in \cite{AbreFern} and \cite{JJ2021} only for the strict case $\mu<\mu_{\alpha,\theta}$.

On the non-existence we provide the following result which extends \cite[Theorem~1.3]{AbreFern} for $\eta>0$ and complements the classical non-existence results  \cite{Ishi,Nguyen} to include non-integer dimensions.
 \begin{theorem}\label{thm non-existence}
 Let $p=2$,  $\alpha=1$ and $\theta\ge 0$. Then there exists $\mu_{0}>0$ such that  $ AD(\eta, \mu, 1,\theta)$ is not attained for any $0\le \eta<1$ and $0<\mu<\mu_0$.
 \end{theorem}
This paper is organized as follows. In the  Section~\ref{Improvement-section} we present some preliminary results on the weighted Sobolev space $X^{1,p}_{R}(\alpha,
\theta)$.  The Section~\ref{sec-AD} is devoted to prove of the sharp estimate given by Theorem~\ref{TSUP}. In the Section~\ref{sec=extremal} we prove the attainability of  $AD(\eta, \mu, \alpha,\theta)$ such as stated in Theorem~\ref{thm-extremal} and Theorem~\ref{thm-cmax} . Finally, in the Section~\ref{sec-non-existence} we prove Theorem~\ref{thm non-existence}.

\section{Notations and preliminary results}
\label{Improvement-section}
In this section we present briefly some notations and preliminary results on  $X^{1,p}_{R}(\alpha,\theta)$. For a deeper discussion on this subject we recommend \cite{JF2017,JJ2021,Clement-deFigueiredo-Mitidieri,DOMADE} and the references therein. 

According to the relation between the parameters $\alpha$ and $p$,  we can distinguish two cases for $X^{1,p}_{R}(\alpha,\theta)$: the \textit{Sobolev case} when $\alpha-p+1>0$ and the  \textit{Trudinger-Moser case} for $\alpha-p+1=0$. Supposing $\alpha-p+1>0$, we have the following continuous embedding  
\begin{equation}\label{Ebeddingswhole}
    X^{1,p}_{R}(\alpha,\theta) \hookrightarrow L^q_{\theta} \quad \mbox{if}\quad q \in \left.\left[p, p^{\ast}\right.\right]\quad\mbox{and}\quad\theta\ge\alpha-p
\end{equation}
where the critical exponent $p^{\ast}$ is given by
$$
p^{\ast}=p^{\ast}(\alpha,\theta,p)=\frac{(\theta+1)p}{\alpha-p+1}.
$$
Also, the embeddings \eqref{Ebeddingswhole} are compact for the strict conditions  $\theta>\alpha-p$ and $p<q < p^{\ast}$. On the other hand, for the \textit{Trudinger-Moser case} we have
the continuous embeddings
\begin{equation}
\label{TM-compactembeddings}
 X^{1,p}_{R}(\alpha,\theta) \hookrightarrow L^q_{\theta}\quad\mbox{for all}\quad q\in [p,\infty)
\end{equation}
which are compact for the  strict case $q>p$. Of course, if $0<R\not =\infty$ the embedding \eqref{Ebeddingswhole} and \eqref{TM-compactembeddings} can be extend to  $1\le q\le p^{*}$ and $1\le q<\infty$, respectively.
\begin{remark}\label{remarkW} Let us denote $W^{1,p}_{R}(\alpha,
\theta)$ the set  of all  functions  $u\in AC_{loc}(0,R)$ such that  $u\in L^{p}_{\theta}$ and $u^{\prime}\in L^{p}_{\alpha}$. We recall  $W^{1,p}_{R}(\alpha,
\theta)$  is  a Banach space  endowed with the norm \eqref{Xnorm-full}. In addition,  according to \cite[Lemma~2.2]{JF2017},  the embeddings \eqref{Ebeddingswhole} and \eqref{TM-compactembeddings} also hold for $W^{1,p}_{R}(\alpha,\theta)$.
\end{remark}
\noindent From \cite[Lemma~4.1]{JJ2021},  for each $u\in X^{1,p}_{\infty}(\alpha,\theta)$, $p\ge2$, we have the point-wise estimate
\begin{equation}\label{radial lemma}
    \vert u(r)\vert^{p}\le \frac{C}{r^{\frac{\alpha+\theta(p-1)}{p}}}\|u\|^{p},
\,\;\forall\,\; r>0 
\end{equation}
where $C>0$ depends only on $\alpha$, $p$ and
$\theta$. For any $q\ge 1$, the following  elementary inequality holds
 \begin{equation}\label{elementary-ineq}
 \left(x+y\right)^{q}\le (1+\epsilon)^{\frac{q-1}{q}}x^{q}+\big(1-(1+\epsilon)^{-\frac{1}{q}}\big)^{1-q}y^{q},\;\; x,y\ge 0
 \end{equation}
for all 
$\epsilon>0$.

Henceforth we suppose $\alpha, \theta$ and $p$ such as in Theorem~\ref{TSUP} and $\varphi_{p}$ as defined in \eqref{exp-frac}. 
\begin{lemma}\label{prop1} Let $R>0$ and $\mu, \sigma\ge 0$ be arbitrary real numbers.
\begin{itemize}
\item [$(i)$] For any $u\in X^{1,p}_{\infty}(\alpha,\theta)$ we have $\exp(\mu \vert u\vert^{\frac{p}{p-1}})\in L^{1}_{\sigma}(0,R)$. In addition, if $\mu<\mu_{\alpha,\sigma}$ then
\begin{equation}\nonumber
\sup_{\|u^{\prime}\|_{L^{p}_{\alpha}}\le 1,\, \|u\|_{L^{p}_{\theta}}\le M }\int_{0}^{R}e^{\mu \vert u\vert^{\frac{p}{p-1}}}\mathrm{d}\lambda_{\sigma}\le c
\end{equation}
 for some  constant $c=c( \alpha,\sigma,\mu,M, R)>0$.\\
 \item[$(ii)$] For any $u\in X^{1,p}_{\infty}(\alpha,\theta)$ we have $\varphi_{p}(\mu \vert u\vert^{\frac{p}{p-1}})\in L^{1}_{\theta}(R,\infty)$. Also, 
 \begin{equation}\label{uniform-rabo}
\sup_{\|u^{\prime}\|_{L^{p}_{\alpha}}\le 1,\, \|u\|_{L^{p}_{\theta}}\le M }\int_{R}^{\infty}\varphi_{p}(\mu \vert u\vert^{\frac{p}{p-1}})\mathrm{d}\lambda_{\theta}\le c
\end{equation}
 for some  constant $c=c( \alpha,\theta,\mu,M, R)>0$.
\end{itemize}
\end{lemma}
\begin{proof}
 For each $u\in X^{1,p}_{\infty}(\alpha,\theta)$, by setting $v=u-u(R)$ on $(0,R)$  we have $v\in X^{1,p}_{R}(\alpha,\theta)$. Thus, from \eqref{radial lemma} and \eqref{elementary-ineq} we have
\begin{equation}\nonumber
\begin{aligned}
\vert u\vert^{\frac{p}{p-1}}
&\le (1+\epsilon)^{\frac{1}{p}}\vert v\vert^{\frac{p}{p-1}}+\frac{c_1}{R^{\theta+1}}\|u\|^{\frac{p}{p-1}},\\
\end{aligned}
\end{equation}
where $c_1$ depends only on $\alpha,\theta$ and $\epsilon$. Hence,
\begin{equation}\label{bound-ball}
\begin{aligned}
 \int_{0}^{R}e^{\mu \vert u\vert^{\frac{p}{p-1}}}\mathrm{d}\lambda_{\sigma}
 & \le e^{\frac{c_1\mu }{R^{\theta+1}}\|u\|^{\frac{p}{p-1}}} \int_{0}^{R}e^{(1+\epsilon)^{\frac{1}{p}}\mu \vert v\vert^{\frac{p}{p-1}}}\mathrm{d}\lambda_{\sigma}.
\end{aligned}
\end{equation}
 Hence, by choosing $\epsilon>0$ such that $(1+\epsilon)^{1/p}\mu<\mu_{\alpha,\sigma}$  the above inequality and \eqref{TM1} imply  $(i)$. 
 
 Further, the continuous embedding  \eqref{TM-compactembeddings} and the monotone convergence theorem yield
 \begin{equation}
 \label{UE}
\begin{aligned}
& \int_{R}^{\infty}\varphi_{p}(\mu \vert u\vert^{\frac{p}{p-1}})\mathrm{d}\lambda
_{\theta}=\sum_{j \in\mathbb{N}\;:\;j\ge k_0}\int _{R}^{\infty}\frac{\mu^{j}\vert u\vert^{\frac{jp}{p-1}}}{j!}\mathrm{d}\lambda_{\theta}\\
\quad\quad&\le \frac{\mu^{k_0}}{k_0!}\left(c_1\|u\|\right)^{\frac{k_0p}{p-1}}+ \frac{\mu^{k_0+1}}{(k_0+1)!}\left(c_1\|u\|\right)^{\frac{(k_0+1)p}{p-1}}+ \sum_{j\in\mathbb{N}\;:\;j\ge k_0+2}\frac{\mu^{j}}{j!}\int _{R}^{\infty}\vert u\vert^{\frac{jp}{p-1}}\mathrm{d}\lambda_{\theta}
\end{aligned}
\end{equation}
for some $c_1>0$ depending only on $\alpha$ and $\theta$. Also, $j\ge k_0+2$, from \eqref{radial lemma} we have
\begin{equation}\nonumber
\int_{R}^{\infty}\vert u(r)\vert^{\frac{jp}{p-1}}\mathrm{d}\lambda_{\theta}\le \left(C\|u\|^{p}\right)^{\frac{j}{p-1}}\int_{R}^{\infty}\frac{1}{r^{(\theta+1)\frac{j}{p}}}\mathrm{d}\lambda_{\theta}\le\left(C\|u\|^{p}\right)^{\frac{j}{p-1}} \frac{c_2}{R^{\frac{(\theta+1)j}{p}}},
\end{equation}
where $c_2$ depends only on $p, \theta
 $ and $R$.  Using \eqref{UE}
 \begin{equation}\label{LE1}
\begin{aligned}
\int_{R}^{\infty}\varphi_{p}(\mu \vert u\vert^{\frac{p}{p-1}})\mathrm{d}\lambda
_{\theta} & \le \max\{1, c_2\} e^{\max\big\{\mu (c_1\|u\|)^{\frac{p}{p-1}},\mu(C\|u\|^{p})^{\frac{1}{p-1}}R^{-\frac{\theta+1}{p}}\big\}}
\end{aligned}
\end{equation}
which proves $(ii)$.
\end{proof}
\begin{remark}\label{remark-imp}
By setting $\sigma=\theta$ in Lemma~\ref{prop1}-(i) and combining this with $(ii)$, we can see that $\varphi_{p}(\mu \vert u\vert^{\frac{p}{p-1}})\in L^{1}_{\theta}(0,\infty)$, for any $\mu>0$ and $u\in X^{1,p}_{\infty}(\alpha,\theta)$. Moreover,  in the uniform estimate \eqref{uniform-rabo}  we are not supposing $\mu\le \mu_{\alpha,\theta}$.
\end{remark}
We finish this section with some elementary properties of the fractional integral in \eqref{fractional integral}. 
First, the change of variables $s=\tau r$ yields
\begin{equation}\label{rts-fractional}
\begin{aligned}
\int_{0}^{\infty}f(\tau r)\mathrm{d}\lambda_{\theta}=\frac{1}{\tau^{\theta+1}}\int_{0}^{\infty} f(s)\mathrm{d}\lambda_{\theta},\quad \tau>0.
\end{aligned}
\end{equation}
Thus,   by setting  $u_{\tau}(r)=\zeta u(\tau r),$ with $\zeta, \tau>0$  and $u\in X^{1,p}_{\infty}(\alpha,\theta)$ we can write
\begin{equation}\label{rts-LpLq}
\begin{aligned}
&\|u^{\prime}_{\tau}\|^{p}_{L^{p}_{\alpha}}=\frac{(\zeta\tau)^{p}}{\tau^{\alpha+1}}\|u^{\prime}\|^{p}_{L^{p}_{\alpha}}\\
&\|u_{\tau}\|^{q}_{L^{q}_{\theta}}=\frac{\zeta^{q}}{\tau^{\theta+1}}\|u\|^{q}_{L^{q}_{\theta}}, \;\; q\ge p.
\end{aligned}
\end{equation}
\section{Sharp Trudinger-Moser inequality of Adimurthi-Druet type}
\label{sec-AD}
This  section is devoted to prove Theorem~\ref{TSUP}.   We split the proof into two steps.
\paragraph*{\textbf{Step~1:} Boundedness}
We follow the argument of V.H. Nguyen  \cite{Nguyen}.  Let  $u\in X^{1,p}_{\infty}$ be arbitrary. Set 
$$u_{\tau}(r)=u(\tau^{\frac{1}{\theta+1}}r),\,\tau>0.$$
Then, \eqref{rts-LpLq} yields
\begin{equation}\nonumber
\begin{aligned}
 \|u\|^{p}_{L^{p}_{\theta}}=\tau \|u_{\tau}\|^{p}_{L^{p}_{\theta}}\quad\mbox{and}\quad
 \|u^{\prime}\|^{p}_{L^{p}_{\alpha}}=\|u^{\prime}_{\tau}\|^{p}_{L^{p}_{\alpha}}.
\end{aligned}
\end{equation}
Consequently
\begin{equation}\nonumber
\begin{aligned}
\sup_{\|u^{\prime}\|^{p}_{L^{p}_{\alpha}}+\tau\|u\|^{p}_{L^{p}_{\theta}}= 1}\int_{0}^{\infty}\varphi_{p}\big(\mu_{\alpha,\theta}\vert u\vert^{\frac{p}{p-1}}\big)\mathrm{d}\lambda_{\theta}&=\frac{1}{\tau}\sup_{\|u\|= 1}\int_{0}^{\infty}\varphi_{p}\big(\mu_{\alpha,\theta}\vert u\vert^{\frac{p}{p-1}}\big)\mathrm{d}\lambda_{\theta}
\end{aligned}
\end{equation}
which is finite due to \cite[Theorem~1.1]{AbreFern} (see also \cite[Theorem~1.2]{JJ2021}). Now, for $\tau=1-\eta$ and $u\in X^{1,p}_{\infty}$ with $\|u\|=1$, we  set
$$
w=\frac{u}{(\|u^{\prime}\|^{p}_{L^{p}_{\alpha}}+\tau\|u\|^{p}_{L^{p}_{\theta}})^{\frac{1}{p}}}.
$$
Then one has $\|w^{\prime}\|^{p}_{L^{p}_{\alpha}}+\tau\|w\|^{p}_{L^{p}_{\theta}}=1$,  and from  above observation it follows
\begin{equation}\label{tau-sup}
\int_{0}^{\infty}\varphi_{p}\big(\mu_{\alpha,\theta}\vert w\vert^{\frac{p}{p-1}}\big)\mathrm{d}\lambda_{\theta}\le \frac{1}{1-\eta}\sup_{\|u\|= 1}\int_{0}^{\infty}\varphi_{p}\big(\mu_{\alpha,\theta}\vert u\vert^{\frac{p}{p-1}}\big)\mathrm{d}\lambda_{\theta}.
\end{equation}
Since $\vert u\vert^{\frac{p}{p-1}}\le (1-\eta\|u\|^{p}_{L^{p}_{\theta}})^{\frac{1}{p-1}}\vert w\vert^{\frac{p}{p-1}}$, it follows that
$
 (1+\eta\|u\|^{p}_{L^{p}_{\theta}})^{\frac{1}{p-1}}\vert u\vert ^{\frac{p}{p-1}}\le \vert w\vert^{\frac{p}{p-1}}
$
which together with 
\eqref{tau-sup} and \cite[Theorem~1.1]{AbreFern}  completes the proof of the item $(1)$.\\

\paragraph*{\textbf{Step~2:} Sharpness} We will employ the Moser type sequence $(v_n)$ given by
\begin{equation}\label{Msequences}
v_n(r)=\frac{1}{\omega^{\frac{1}{p}}_{\alpha}}\left\{\begin{aligned}
&\Big(\frac{n}{\theta+1}\Big)^{\frac{p-1}{p}}, & \mbox{if}&\quad 0\le r \le e^{-\frac{n}{\theta+1}},\\
&\Big(\frac{n}{\theta+1}\Big)^{-\frac{1}{p}}\ln\frac{1}{r},&\mbox{if}&\quad e^{-\frac{n}{\theta+1}}<r<1,\\
& 0, &\mbox{if}& \quad r\ge 1.
 \end{aligned}\right.
\end{equation}
It follows that 
 \begin{equation}\nonumber
 \begin{aligned}
 \|v^{\prime}_n\|^{p}_{L^{p}_{\alpha}}=1\quad\mbox{and}\quad
 \|v_n\|^{p}_{L^{p}_{\theta}} =\frac{\omega_{\theta}}{n(\theta+1)^{p}\omega_{\alpha}}\int_{0}^{n}s^{p}e^{-s}\mathrm{d} s+O(n^{p-1}e^{-n}).
 \end{aligned}
 \end{equation}
 Since 
 \begin{equation}\nonumber
\Gamma(p+1)=\int_{0}^{\infty}s^{p}e^{-s}\mathrm{d}s\quad \mbox{and}\quad \lim_{x\rightarrow\infty}\frac{e^{x}}{x^p}\int_{x}^{\infty}s^pe^{-s}\mathrm{d}s=1
\end{equation}
 we can also write
 \begin{equation}\nonumber
 \begin{aligned}
\|v_n\|^{p}_{L^{p}_{\theta}} =\frac{\omega_{\theta}}{\omega_{\alpha}}\frac{\Gamma(p+1)}{n(\theta+1)^{p}}+O(n^{p-1}e^{-n}).
 \end{aligned}
 \end{equation}
  For $\rho>0$, define $v_{n,\rho}(r)=v_{n}(r/\rho)$ and $w_{n,\rho}=v_{n,\rho}/\|v_{n,\rho}\|$. From \eqref{rts-LpLq} with $\alpha=p-1$, we get
  \begin{equation}\label{vrho1}
\|v^{\prime}_{n,\rho}\|^{p}_{L^{p}_{\alpha}}=1\quad\mbox{and}\quad
\|v_{n,\rho}\|^{p}_{L^{p}_{\theta}}=\frac{\rho^{\theta+1}}{n}\Big[\frac{\omega_{\theta}}{\omega_{\alpha}}\frac{\Gamma(p+1)}{(\theta+1)^{p}}+O(n^{p}e^{-n})\Big].
\end{equation}
 Then
 \begin{equation}\nonumber
 \frac{1+\|w_{n,\rho}\|^{p}_{L^{p}_{\theta}}}{\|v_{n,\rho}\|^{p}}=\frac{1+\frac{\|v_{n,\rho}\|^{p}_{L^{p}_{\theta}}}{\|v_{n,\rho}\|^{p}}}{\|v_{n,\rho}\|^{p}}=\frac{1+2\|v_{n,\rho}\|^{p}_{L^{p}_{\theta}}}{(1+\|v_{n,\rho}\|^{p}_{L^{p}_{\theta}})^{2}}=1-\mathcal{R}_n,
 \end{equation}
 where
 $
 \mathcal{R}_n={\|v_{n,\rho}\|^{2p}_{L^{p}_{\theta}}}/{(1+\|v_{n,\rho}\|^{p}_{L^{p}_{\theta}})^{2}}.
 $
 Hence, since $\|w_{n,\rho}\|=1$ for $n$ large enough we have
 \begin{equation}\nonumber
 \begin{aligned}
   AD(1, \mu_{\alpha,\theta}, \alpha,\theta) &\ge \int_{0}^{\infty}\varphi_{p}\big(\mu_{\alpha,\theta}(1+\|w_{n,\rho}\|^{p}_{L^{p}_{\theta}})^{\frac{1}{p-1}}\vert w_{n,\rho}\vert^{\frac{p}{p-1}}\big)\mathrm{d}\lambda_{\theta}\\
   &\ge \rho^{\theta+1}\int_{0}^{e^{-\frac{n}{\theta+1}}}\varphi_{p}\big(\left(1-\mathcal{R}_n\right)\mu_{\alpha,\theta}\vert v_{n}\vert^{\frac{p}{p-1}}\big)\mathrm{d}\lambda_{\theta}\\
   &=\frac{\omega_{\theta}\rho^{\theta+1}}{\theta+1}\Big[e^{n-n\mathcal{R}_n}-\sum_{j\in\mathbb{N\; :\;}j<p-1}\frac{\left(1-\mathcal{R}_n\right)^{j}n^{j}}{j!}\Big]e^{-n}.\\
   \end{aligned}
 \end{equation}
From \eqref{vrho1}, we obtain $n\mathcal{R}_n\rightarrow 0$  as   $n\rightarrow\infty$. Then, letting  $n\rightarrow\infty$ we obtain $AD(1, \mu_{\alpha,\theta}, \alpha,\theta)\ge
   {\omega_{\theta}\rho^{\theta+1}}/{(\theta+1)}$
 for any $\rho>0$. Hence $AD(1, \mu_{\alpha,\theta}, \alpha,\theta)=\infty$.

\section{Extremal function for Adimurthi-Druet type inequality}
\label{sec=extremal}
In this section we will show Theorem~\ref{thm-extremal}.  Here,  we are assuming  the assumptions
$\alpha=p-1\ge 1$ and $\theta\ge \alpha $. We first prove a lower bound for $ AD(\eta, \mu, \alpha,\theta)$.
\begin{proposition}\label{vanishing-integer} Let $p\ge 2$ be an integer number and  $\eta\in [0,1)$. Then
\begin{equation}\nonumber
 AD(\eta, \mu, \alpha,\theta)>\left\{\begin{aligned}
& \frac{\mu^{p-1}}{(p-1)!} (\eta+1), \;\; \mbox{if}\;\; p>2\;\;\mbox{and}\;\; \mu\in (0,\mu_{\alpha,\theta}]\\
& \frac{\mu^{p-1}}{(p-1)!} (\eta+1), \;\; \mbox{if}\;\; p=2\;\;\mbox{and}\;\; \mu \in \Big(\frac{2(1+2\eta)}{(1+\eta)^2B_{2,\theta}},\mu_{1,\theta}\Big],
 \end{aligned}\right.
\end{equation}
where
\begin{equation}
\frac{1}{B_{2,\theta}}=\inf_{0\not\equiv u\in X^{1,2}_{\infty}(1,\theta)}\frac{\|u^{\prime}\|^{2}_{L^{2}_{1}}\|u\|^{2}_{L^{2}_{\theta}}}{\|u\|^{4}_{L^{4}_{\theta}}}.
\end{equation}
\end{proposition}
\begin{proof}
We follow the argument of Ishiwata \cite{Ishi}.
Let $u\in X^{1,p}_{\infty}(\alpha,\theta)$ such that $\|u\|=1$, and  set $$u_t(r)=t^{1/p}u(t^{\frac{1}{\theta+1}}r).$$
From \eqref{rts-LpLq}, we can easily to show that
\begin{equation}\nonumber
\begin{aligned}
&\|u^{\prime}_t\|^{p}_{L^{p}_{\alpha}}=t\|u^{\prime}\|_{L^{p}_{\alpha}}^{p}\\
 &\|u_t\|^{q}_{L^{q}_{\theta}}=t^{\frac{q-p}{p}}\|u\|_{L^{q}_{\theta}}^{q},\;\; \forall\, q\ge p.
\end{aligned}
\end{equation}
In particular, if  $v_t=\xi_{t}u_t$ with $\xi_{t}=(t+(1-t)\|u\|^{p}_{L^{p}_{\theta}})^{-1/p}$  with $t>0$ small enough, we have
\begin{equation}\nonumber
\|v_t\|=1\;\;\mbox{and}\;\; \|v_t\|^{q}_{L^{q}_{\theta}}=\xi^{q}_{t}t^{\frac{q-p}{p}}\|u\|^{q}_{L^{q}_{\theta}},\;\; q\ge p.
\end{equation} 
Noticing that
$\varphi_{p}(s)\ge \frac{s^{k_0}}{k_0!}+\frac{s^{k_0+1}}{(k_0+1)!}$, $s\ge 0$
we obtain 
\begin{equation}\nonumber
\begin{aligned}
AD(\eta, \mu, \alpha,\theta)&\ge \frac{\mu^{k_0}}{k_0!}\big(1+\eta \xi^{p}_{t} \|u\|^{p}_{L^{p}_{\theta}}\big)^{\frac{k_0}{p-1}}\|u\|^{\frac{pk_0}{p-1}}_{L^{\frac{k_0p}{p-1}}_{\theta}}t^{\frac{\frac{k_0p}{p-1}-p}{p}}\xi^{\frac{k_0p}{p-1}}_{t}\\
&+ \frac{\mu^{k_0+1}}{(k_0+1)!}\big(1+\eta \xi^{p}_{t} \|u\|^{p}_{L^{p}_{\theta}}\big)^{\frac{k_0+1}{p-1}}\|u\|^{\frac{p(k_0+1)}{p-1}}_{L^{\frac{(k_0+1)p}{p-1}}_{\theta}}t^{\frac{\frac{(k_0+1)p}{p-1}-p}{p}}\xi^{\frac{(k_0+1)p}{p-1}}_{t}.
\end{aligned}
\end{equation}
Since we are supposing that $p\ge 2$ is an integer number we have $k_0=p-1$. Thus, 
\begin{equation}\nonumber
\begin{aligned}
AD(\eta, \mu, \alpha,\theta)&\ge \frac{\mu^{p-1}}{(p-1)!}h(t),
\end{aligned}
\end{equation}
where
\begin{equation}\nonumber
h(t)=\big(\xi^{p}_{t}+\eta \xi^{2p}_{t} \|u\|^{p}_{L^{p}_{\theta}}\big)\|u\|^{p}_{L^{p}_{\theta}}
+ \frac{\mu}{p}\big(\xi^{p}_t+\eta \xi^{2p}_{t} \|u\|^{p}_{L^{p}_{\theta}}\big)^{\frac{p}{p-1}}\|u\|^{\frac{p^2}{p-1}}_{L^{\frac{p^2}{p-1}}_{\theta}}t^{\frac{1}{p-1}}.
\end{equation}
Since  $\xi_{t}\rightarrow1/\|u\|_{L^{p}_{\theta}}$ as $t\rightarrow 0^{+}$, we have $ h(0)=1+\eta$. Thus, it remains to show that $h^{\prime}(t)>0$, if $0<t\ll
1$.  In order to see this, note that for any $q\ge p$  (recall $\|u^{\prime}\|^{p}_{L^{p}_{\alpha}}+\|u\|^{p}_{L^{p}_{\theta}}=1)$
\begin{equation}\label{xit}
\begin{aligned}
\frac{d}{dt}\left(\xi^{q}_{t}\right)
&=-\frac{q}{p}\big(t+(1-t)\|u\|^{p}_{L^{p}_{\theta}}\big)^{-\frac{q+p}{p}}\big(1-\|u\|^{p}_{L^{p}_{\theta}}\big)
\rightarrow-\frac{q}{p} \|u\|^{-(p+q)}_{L^{p}_{\theta}}\|u^{\prime}\|^{p}_{L^{p}_{\alpha}}\;\;\mbox{as}\;\; t\rightarrow 0.
\end{aligned}
\end{equation}
 Therefore
\begin{equation}\nonumber
\begin{aligned}
h^{\prime}(t)&=\frac{\mu}{p(p-1)}\big(\xi^{p}_t+\eta \xi^{2p}_{t} \|u\|^{p}_{L^{p}_{\theta}}\big)^{\frac{p}{p-1}}\|u\|^{\frac{p^2}{p-1}}_{L^{\frac{p^2}{p-1}}_{\theta}}t^{\frac{2-p}{p-1}}+f(t),
\end{aligned}
\end{equation}
where
\begin{equation}\label{xit1}
\begin{aligned}
f(t)&=\|u\|^{p}_{L^{p}_{\theta}}\frac{d}{dt}\big(\xi^{p}_{t}+\eta \xi^{2p}_{t} \|u\|^{p}_{L^{p}_{\theta}}\big)+t^{\frac{1}{p-1}}\frac{\mu}{p}\|u\|^{\frac{p^2}{p-1}}_{L^{\frac{p^2}{p-1}}_{\theta}}\frac{d}{dt}\big(\xi^{p}_t+\eta \xi^{2p}_{t} \|u\|^{p}_{L^{p}_{\theta}}\big)^{\frac{p}{p-1}}\\
&\rightarrow -(1+2\eta)\frac{\|u^{\prime}\|^{p}_{L^{p}_{\alpha}}}{\|u\|^{p}_{L^{p}_{\theta}}}\;\;\mbox{as}\;\; t\rightarrow 0.
\end{aligned}
\end{equation}
Hence, if $p>2$ we have  $$\lim_{t\rightarrow 0}h^{\prime}(t)=\infty$$
and, thus $h^{\prime}(t)>0$, if $t>0$ is small enough.  If $p=2$, also from \eqref{xit1}, we have
\begin{equation}\nonumber
\begin{aligned}
\lim_{t\rightarrow 0}h^{\prime}(t)&=\frac{\mu}{2}\frac{(1+\eta)^2}{\|u\|^{4}_{L^{2}_{\theta}}}\|u\|^{4}_{L^{4}_{\theta}}-(1+2\eta)\frac{\|u^{\prime}\|^{2}_{L^{2}_{1}}}{\|u\|^{2}_{L^{2}_{\theta}}}.
\end{aligned}
\end{equation}
Hence, 
\begin{equation}\nonumber
\begin{aligned}
\lim_{t\rightarrow 0}h^{\prime}(t)>0,\;\;\mbox{if}\;\; \mu>\frac{2(1+2\eta)}{(1+\eta)^2}\frac{\|u^{\prime}\|^{2}_{L^{2}_{1}}\|u\|^{2}_{L^{2}_{\theta}}}{\|u\|^{4}_{L^{4}_{\theta}}}.
\end{aligned}
\end{equation}
Since $1/B_{2,\theta}$ is attained (cf. \cite[Proposition~7.1]{AbreFern}) for some $u\in X^{1,2}_{\infty}$, with $\|u\|=1$ our result is proved.
\end{proof} 

Let $(u_n)\subset X^{1,p}_{\infty}(\alpha,\theta)$ be a maximizing sequence for the supremmum $AD(\eta, \mu, \alpha,\theta)$. From the  P\'{o}lya-Szeg\"{o} type principle in \cite{AbreFern} (see also \cite{Alvino2017}),  we can assume that each $u_n$ is a non-increasing function. Moreover,   from \cite[Lemma~2.2]{JF2017}) (see Remark~\ref{remarkW}), for any $ R>0$ and $q\in (1,\infty)$ we can assume that 
\begin{equation}\label{loc-convergence}
u_n\rightharpoonup u_0\;\;\mbox{in}\;\; X^{1,p}_{\infty},\;\; u_n\rightarrow u_0\;\;\mbox{in}\;\;L^{q}_{\theta}(0,R)\;\;\mbox{and}\;\; u_n(r)\rightarrow u_0(r)\;\;\mbox{a.e in}\;\; (0,\infty).
\end{equation} 
In addition, since $0<\|u_n\|^{p}_{L^{p}_{\theta}}\le 1$, up to a subsequence, we can take $a\in [0,1]$ such that 
\begin{equation}\label{ab-limite}
\|u_n\|^{p}_{L^{p}_{\theta}}\rightarrow a.
\end{equation}
\subsection{Proof of Theorem~\ref{thm-extremal}:  Maximizers  for the subcritical case}
\label{Max-sub}
For $R>1$, we set $v_n(r)=u_n(r)-u_n(R)$, with $r\in (0,R]$. Then
 $v_n\in X^{1,p}_{R}(\alpha,\theta)$ and since $\|u_n\|=1$
\begin{equation}\label{deltauv}
\|v^{\prime}_n\|^{p}_{L^{p}_{\alpha}}\le 1-\|u_n\|^{p}_{L^{p}_{\theta}}.
\end{equation}
For any $\epsilon>0$, combining \eqref{radial lemma} with \eqref{elementary-ineq}  we have
\begin{equation}\label{uvR}
\begin{aligned}
\vert u_n\vert ^{\frac{p}{p-1}}
&\le (1+\epsilon)^{\frac{1}{p}}\vert v_n\vert ^{\frac{p}{p-1}}+\frac{c_{\epsilon}}{R^{\frac{\theta+1}{p}}},
\end{aligned}
\end{equation}
for some $c_{\epsilon}>0$ depending only on $\epsilon, p$ and $\theta$.  If $w_n=v_n/\|v^{\prime}_n\|_{L^{p}_{\alpha}}$,  then \eqref{deltauv} and  \eqref{uvR}  imply
\begin{equation}\nonumber
\begin{aligned}
\varphi_{p}\big(\mu(1+\eta\|u_n\|^{p}_{L^{p}_{\theta}})^{\frac{1}{p-1}}\vert u_n\vert^{\frac{p}{p-1}}\big)
&\le e^{\mu(1+\eta)c_{\epsilon}R^{-\frac{\theta+1}{p}}}e^{\mu(1+\epsilon)^{\frac{1}{p}}\vert w_n \vert^{\frac{p}{p-1}}}.
\end{aligned}
\end{equation}
Hence, if $\epsilon>0$ is small enough and $q>1$ is close to $1$  the inequality \eqref{TM1} ensures 
\begin{equation}\nonumber
\begin{aligned}
\sup_{n}\int_{0}^{R}\big[\varphi_{p}\big(\mu(1+\eta\|u_n\|^{p}_{L^{p}_{\theta}})^{\frac{1}{p-1}}\vert u_n\vert ^{\frac{p}{p-1}}\big)\big]^{q}\mathrm{d}\lambda_{\theta} <\infty.
\end{aligned}
\end{equation}
From Vitali's convergence theorem
\begin{equation}\label{sub-loc}
\begin{aligned}
&\lim_{n\rightarrow\infty}\int^{R}_{0}\varphi_{p}\big(\mu(1+\eta\|u_n\|^{p}_{L^{p}_{\theta}})^{\frac{1}{p-1}}\vert u_n\vert^{\frac{p}{p-1}}\big)\mathrm{d}\lambda_{\theta}\\
&\quad\quad=\int^{R}_{0}\varphi_{p}\big(\mu(1+\eta a)^{\frac{1}{p-1}}\vert u_0\vert^{\frac{p}{p-1}}\big)\mathrm{d}\lambda_{\theta}.
\end{aligned}
\end{equation}
Also, form \eqref{loc-convergence}, for any $R>0$ we get
\begin{equation}\label{lq-convergence}
\lim_{n\rightarrow\infty}\int^{R}_{0}(1+\eta\|u_n\|^{p}_{L^{p}_{\theta}})\vert u_n\vert^{p}\mathrm{d}\lambda_{\theta}=(1+\eta a)\int_{0}^{R}\vert u_0\vert^{p}\mathrm{d} \lambda_{\theta}.
\end{equation}
Now, we claim that 
\begin{equation}\label{NnotN}
AD(\eta, \mu, \alpha,\theta)=\left\{\begin{aligned}
& \int_{0}^{\infty}\varphi_{p}\big(\mu(1+\eta a)^{\frac{1}{p-1}}\vert u_0\vert^{\frac{p}{p-1}}\big)\mathrm{d}\lambda_{\theta}, & \;\;\mbox{if}\;\; p\not\in \mathbb{N},\\
& \left.\begin{aligned}
\int_{0}^{\infty}\varphi_{p}\big(\mu(1+\eta a)^{\frac{1}{p-1}}\vert u_0\vert^{\frac{p}{p-1}}\big)\mathrm{d}\lambda_{\theta}\\
+\frac{\mu^{p-1}}{(p-1)!}(1+\eta a)\big(a-\|u_0\|^{p}_{L^{p}_{\theta}}\big)
\end{aligned}\right\} & \mbox{if}\;\; p\in \mathbb{N}.\\
\end{aligned}\right.
\end{equation}
Indeed, for any real number $p\ge 2$, $\alpha=p-1$ and $r\ge R\ge 1$, \eqref{radial lemma} yields
\begin{equation}\label{tail-1}
\begin{aligned}
& \varphi_{p}\big(\mu(1+\eta\|u_n\|^{p}_{L^{p}_{\theta}})^{\frac{1}{p-1}}\vert u_n\vert^{\frac{p}{p-1}}\big)-\frac{\mu^{k_0}}{k_0!}(1+\eta\|u_n\|^{p}_{L^{p}_{\theta}})^{\frac{k_0}{p-1}}\vert u_n\vert ^{\frac{pk_0}{p-1}}\\
&\le \frac{\vert u_n\vert^{p}}{R^{\frac{\theta+1}{p}}} \sum_{j=k_0+1}^{\infty}\frac{\mu^{j}}{j!}(1+\eta)^{\frac{j}{p-1}}C^{\frac{j-p}{p-1}}\\
&\le C^{\prime}\frac{\vert u_n\vert^{p}}{R^{\frac{\theta+1}{p}}},
\end{aligned}
\end{equation}
where $C^{\prime}$  does not dependent of $n$ and $R$. If $p\in\mathbb{N}$ the above estimate yields
\begin{equation}\nonumber
\begin{aligned}
\lim_{R\rightarrow\infty}\lim_{n\rightarrow\infty}\int_{R}^{\infty}\big[(\varphi_{p}\big(\mu(1+\eta\|u_n\|^{p}_{L^{p}_{\theta}})^{\frac{1}{p-1}}\vert u_n\vert ^{\frac{p}{p-1}}\big)\\
\;\;\;\;\;\;-\frac{\mu^{p-1}}{(p-1)!}(1+\eta\|u_n\|^{p}_{L^{p}_{\theta}})\vert u_n\vert^{p}\big]\mathrm{d}\lambda_{\theta}=0.
\end{aligned}
\end{equation}
Hence, by splitting the integral on $(0,R)$ and $(R,\infty)$ and letting $n\rightarrow\infty$ and then $R\rightarrow\infty$, from \eqref{sub-loc} and \eqref{lq-convergence} we obtain  \eqref{NnotN}. If $p\notin\mathbb{N}$, we must have $k_0>p-1$. Since $u_n$ is a non-increasing function
$$
\|u_n\|^{p}_{L^{p}_{\theta}}\ge \int_{0}^{r}\vert u_n\vert^{p}\mathrm{d}\lambda_{\theta}\ge \vert u_n(r)\vert^{p}\int_{0}^{r}\mathrm{d}\lambda_{\theta},\quad r>0.
$$
Thus, there is $c>0$ (independent of $n$) such that
\begin{equation}\nonumber
\vert u_n\vert^{\frac{k_0p}{p-1}}\le \frac{c}{r^{(\theta+1)\frac{k_0}{p-1}}},\quad r>0.
\end{equation}
Hence, arguing as in \eqref{tail-1}, for $r>R$ we can write
\begin{equation}\label{tail-1null}
\begin{aligned}
 \varphi_{p}\big(\mu(1+\eta\|u_n\|^{p}_{L^{p}_{\theta}})^{\frac{1}{p-1}}\vert u_n\vert^{\frac{p}{p-1}}\big)
& \le \frac{c_{2}}{r^{(\theta+1)\frac{k_0}{p-1}}}+\frac{ c_1\vert u_n\vert^{p}}{R^{\frac{\theta+1}{p}}},
\end{aligned}
\end{equation}
where  $c_1$ and $c_{2}$ are  independent of $n$ and $R$. By 
using $k_0>p-1$, from \eqref{tail-1null}
\begin{equation}\nonumber
\begin{aligned}
\lim_{R\rightarrow\infty}\lim_{n\rightarrow\infty}\int_{R}^{\infty}\varphi_{p}\big(\mu(1+\eta\|u_n\|^{p}_{L^{p}_{\theta}})^{\frac{1}{p-1}}\vert u_n\vert^{\frac{p}{p-1}}\big)\mathrm{d}\lambda_{\theta}=0.
\end{aligned}
\end{equation}
Hence,  by splitting the integral on $(0,R)$ and $(R,\infty)$ and letting $n\rightarrow\infty$ and then $R\rightarrow\infty$, from \eqref{sub-loc} we complete the proof of \eqref{NnotN}.

\noindent Now,  if $u_0\equiv 0$ then \eqref{NnotN} yields
\begin{equation}\nonumber
0<AD(\eta, \mu, \alpha,\theta)=\left\{\begin{aligned}
& 0, & \mbox{if}\;\; p\not\in \mathbb{N}\\
& \frac{\mu^{p-1}}{(p-1)!}(1+\eta a)a\le \frac{\mu^{p-1}}{(p-1)!}(1+\eta),
 & \mbox{if}\;\; p\in \mathbb{N},
\end{aligned}\right.
\end{equation}
which contradicts the Proposition~\ref{vanishing-integer}. Hence $u_0\not\equiv0$. If $\tau=(a/\|u_0\|^{p}_{L^{p}_{\theta}})^{1/(\theta+1)}\ge 1$ and $v_0(r)=u_0(r/\tau)$ then
\begin{equation}
\|v_0\|^{p}_{L^{p}_{\theta}}=\|u_0\|^{p}_{L^{p}_{\theta}}\tau^{\theta+1}=a,\;\; \mbox{and}\;\;\|v^{\prime}_0\|_{L^{p}_{\alpha}}=\|u^{\prime}_0\|^{p}_{L^{p}_{\alpha}}.
\end{equation}
It follows that
$
1\ge \liminf_{n}\|u_n\|^{p}\ge \|v_0\|^{p}.
$
Therefore, if $p\in\mathbb{N}$, by using  \eqref{NnotN}
\begin{equation}\nonumber
\begin{aligned}
AD(\eta, \mu, \alpha,\theta) & \ge  \tau^{\theta+1}\int_{0}^{\infty}\varphi_{p}\big(\mu(1+\eta a)^{\frac{1}{p-1}}\vert u_0\vert^{\frac{p}{p-1}}\big)\mathrm{d}\lambda_{\theta}\\
&= AD(\eta, \mu, \alpha,\theta) + (\tau^{\theta+1}-1)\int_{0}^{\infty}\big[\varphi_{p}\big(\mu(1+\eta a)^{\frac{1}{p-1}}\vert u_0\vert^{\frac{p}{p-1}}\big)\\
&\quad\quad\quad-\frac{\mu^{p-1}}{(p-1)!}(1+\eta a)\vert u_0\vert^{p}\big]\mathrm{d}\lambda_{\theta}.
\end{aligned}
\end{equation}
Since $u_0\not\equiv0$, we obtain $\tau=1$ which gives $a=\|u_0\|^{p}_{L^{p}_{\theta}}.$ On the other hand, for $p\not\in\mathbb{N}$ \eqref{NnotN} yields
\begin{equation}\nonumber
\begin{aligned}
AD(\eta, \mu, \alpha,\theta) & \ge \int_{0}^{\infty}\varphi_{p}\big(\mu(1+\eta \|v_0\|^{p}_{L^{p}_{\theta}})^{\frac{1}{p-1}}\vert v_0\vert^{\frac{p}{p-1}}\big)\mathrm{d}\lambda_{\theta}\\
&= \tau^{\theta+1}\int_{0}^{\infty}\varphi_{p}\big(\mu(1+\eta a)^{\frac{1}{p-1}}\vert u_0\vert^{\frac{p}{p-1}}\big)\mathrm{d}\lambda_{\theta}\\
&=\tau^{\theta+1}AD(\eta, \mu, \alpha,\theta),
\end{aligned}
\end{equation}
which gives $\tau=1$ once again. Hence, for any real number $p\ge 2$ (if $p=2$, according to  Proposition~\ref{vanishing-integer}, we must assume $\mu>2(1+2\eta)/(1+\eta)^2B_{2,\theta}$), we have the following
\begin{equation}\nonumber
\begin{aligned}
AD(\eta, \mu, \alpha,\theta)=\int_{0}^{\infty}\varphi_{p}\big(\mu(1+\eta \|u_0\|^{p}_{L^{p}_{\theta}})^{\frac{1}{p-1}}\vert u_0\vert ^{\frac{p}{p-1}}\big)\mathrm{d}\lambda_{\theta}.
\end{aligned}
\end{equation}
This together with the fact that $\|u_0\|\le 1$, implies $\|u_0\|=1$ and completes the proof. 
\subsection{Proof of Theorem~\ref{thm-cmax}:  Maximizers for the critical case}
From Theorem~\ref{thm-extremal}, for each $\epsilon>0$ 
(small) there is a non-increasing function $u_{\epsilon}\in X^{1,p}_{\infty}$ with $\|u_{\epsilon}\|=1$ such that
\begin{equation}\label{sub-extremals}
\begin{aligned}
AD(\eta, \mu_{\alpha,\theta}-\epsilon, \alpha,\theta)=\int_{0}^{\infty}\varphi_{p}\big((\mu_{\alpha,\theta}-\epsilon)(1+\eta \|u_{\epsilon}\|^{p}_{L^{p}_{\theta}})^{\frac{1}{p-1}}\vert u_{\epsilon}\vert^{\frac{p}{p-1}}\big)\mathrm{d}\lambda_{\theta}.
\end{aligned}
\end{equation}
It is easy to see that the Lagrange multipliers theorem implies
\begin{equation}\label{Euler-Lagrange}
\begin{aligned}
 \int_{0}^{\infty}\vert u_{\epsilon}^{\prime}\vert^{p-2}u_{\epsilon}^{\prime}v^{\prime}\,\mathrm{d}\lambda_{\alpha} & =\frac{b_{\epsilon}}{d_{\epsilon}}\int_{0}^{\infty}\varphi_{p}^{\prime}\big(\eta_{\epsilon}\vert u_{\epsilon}\vert^{\frac{p}{p-1}}\big)\vert u_{\epsilon}\vert^{\frac{1}{p-1}}v\,\mathrm{d}\lambda_{\theta}\\
 & +(c_{\epsilon}-1)\int_{0}^{\infty}\vert u_{\epsilon}\vert^{p-1}v\,\mathrm{d}\lambda_{\theta}
\end{aligned}
\end{equation}
for all $v\in X^{1,p}_{\infty}$,  where
\begin{equation}\label{L-constants}
\left\{\begin{aligned}
\eta_{\epsilon}&=\mu_{\epsilon}(1+\eta\|u_{\epsilon}\|^{p}_{L^{p}_{\theta}})^{\frac{1}{p-1}}, \;\;\mbox{with}\;\; \mu_{\epsilon}=\mu_{\alpha,\theta}-\epsilon\\
b_{\epsilon}&=(1+\eta\|u_{\epsilon}\|^{p}_{L^{p}_{\theta}})/(1+2\eta\|u_{\epsilon}\|^{p}_{L^{p}_{\theta}})\\
c_{\epsilon}&=\eta/(1+2\eta\|u_{\epsilon}\|^{p}_{L^{p}_{\theta}})\\
d_{\epsilon}&=\int_{0}^{\infty}\vert u_{\epsilon}\vert^{\frac{p}{p-1}}\varphi_{p}^{\prime}\big(\eta_{\epsilon}\vert u_{\epsilon}\vert^{\frac{p}{p-1}}\big)\,\mathrm{d}\lambda_{\theta}.
\end{aligned}\right.
\end{equation}
In order to perform the blow up analysis method, we need to show that each maximizers $u_{\epsilon}$  belongs to $C^{1}[0,\infty)$.  Indeed, we have the following:
\begin{lemma}\label{reg-lemma}  For each $\epsilon>0$, we have $u_{\epsilon}\in C^{1}[0,\infty)\cap C^{2}(0,\infty).
$\end{lemma} 
\begin{proof}
The proof of this result  proceeds along the same lines of  \cite[Lemma~6]{RUDOL}, and we limit
ourselves to sketching a few differences. Firstly, given $\sigma>0$ we will prove that 
\begin{equation}\label{limite-exzero}
 \lim_{r\rightarrow 0^{+}}r^{\sigma}\big[\varphi_{p}^{\prime}\big(\eta_{\epsilon}\vert u_{\epsilon}\vert ^{\frac{p}{p-1}}\big)\vert u_{\epsilon}\vert^{\frac{1}{p-1}}+(c_{\epsilon}-1)\vert u_{\epsilon}\vert^{p-1}\big]=0.
\end{equation}
Since $0\le \varphi^{\prime}_{p}(t)\le e^{t}$, $t\ge 0$ it is sufficient to show that 
\begin{equation}\label{varphi-zerolimit}
 \lim_{r\rightarrow 0^{+}}r^{\sigma}\varphi(r)=0, \;\; \mbox{with}\;\; \varphi =:\vert u_{\epsilon}\vert^{\frac{1}{p-1}}e^{\eta_{\epsilon}\vert u_{\epsilon}\vert ^{\frac{p}{p-1}}}+(c_{\epsilon}-1)\vert u_{\epsilon}\vert^{p-1}.
\end{equation}
Let  $1<q<p$ such that $(\sigma+1)q>p$ and $m={p}/{q}$. Then the H\"{o}lder inequality and $\alpha=p-1$ yield
\begin{equation}\label{grad-Lm}
 \int_{0}^{1}r^{\sigma}\vert u^{\prime}_{\epsilon}\vert^{m}dr\le C\|u_{\epsilon}\|^{m},  
\end{equation}
for some $C>0$ depending only on $\alpha,\sigma$ and $m$. Note that 
\begin{equation}\nonumber
\begin{aligned}
    \varphi^{\prime}&=\frac{1}{p-1}\vert u_{\epsilon}\vert^{\frac{2-p}{p-1}}u^{\prime}_{\epsilon}e^{\eta_{\epsilon}\vert u_{\epsilon}\vert^{\frac{p}{p-1}}}+\frac{\eta_{\epsilon}p}{p-1}\vert u_{\epsilon}\vert^{\frac{2}{p-1}}u^{\prime}_{\epsilon}e^{\eta_{\epsilon}\vert u_{\epsilon}\vert^{\frac{p}{p-1}}}\\
    & +(c_{\epsilon}-1)(p-1)\vert u_{\epsilon}\vert^{p-2}u^{\prime}_{\epsilon}.
    \end{aligned}
\end{equation}
In addition, without loss of generality we can assume $u_{\epsilon}>0$ in $(0, 1)$ and $\lim_{r\rightarrow 0}u_{\epsilon}(r)=+\infty.$ Hence,  there exists $C>0$ such that 
\begin{equation}\nonumber
\begin{aligned}
    \vert \varphi^{\prime}\vert &\le C\left[\vert u^{\prime}_{\epsilon}\vert e^{\eta_{\epsilon}\vert u_{\epsilon}\vert^{\frac{p}{p-1}}}+\vert u_{\epsilon}\vert^{\frac{2}{p-1}}\vert u^{\prime}_{\epsilon}\vert e^{\eta_{\epsilon}\vert u_{\epsilon}\vert^{\frac{p}{p-1}}}+\vert u_{\epsilon}\vert ^{p-2}\vert u^{\prime}_{\epsilon}\vert\right],
    \end{aligned}
\end{equation}
on $ (0,1/2)$. Now, for  $m>1$ given by \eqref{grad-Lm} and $q_1,q_2>1$ such that
\begin{equation}\nonumber
    \frac{2}{q_1(p-1)}+\frac{1}{m}+\frac{1}{q_2}=1
\end{equation}
the H\"{o}lder inequality, Remark~\ref{remarkW} (recall $u_{\epsilon}\in L^{q_1}_{\sigma}(0,1)$), \eqref{grad-Lm} and  Lemma~\ref{prop1}-$(i)$ imply
\begin{equation}\label{graphi-1}
    \begin{aligned}
    & \int_{0}^{1/2}r^{\sigma}\vert u_{\epsilon}\vert^{\frac{2}{p-1}}\vert u^{\prime}_{\epsilon}\vert e^{\eta_{\epsilon}\vert u_{\epsilon}\vert^{\frac{p}{p-1}}}dr \\
    &\le\Big(\int_{0}^{1}r^{\sigma}\vert u_{\epsilon}\vert^{q_1}dr\Big)^{\frac{2}{q_1(p-1)}}\Big(\int_{0}^{1}r^{\sigma}\vert u^{\prime}_{\epsilon}\vert^{m}dr\Big)^{\frac{1}{m}}\Big(\int_{0}^{1}r^{\sigma}e^{q_2\eta_{\epsilon}\vert u_{\epsilon}\vert^{\frac{p}{p-1}}}dr\Big)^{\frac{1}{q_2}}\le c,
    \end{aligned}
\end{equation}
for some $c>0$ depending only on $\alpha,\sigma$ and $m$. Analogously, 
 \begin{equation}\label{graphi-2}
    \begin{aligned}
     \int_{0}^{1/2}r^{\sigma}\vert u^{\prime}_{\epsilon}\vert e^{\eta_{\epsilon}\vert u_{\epsilon}\vert^{\frac{p}{p-1}}}dr \le C \|u^{\prime}_{\epsilon}\|_{L^{m}_{\sigma}(0,1)}\Big(\int_{0}^{1}r^{\sigma}e^{\frac{m\eta_{\epsilon}}{m-1}\vert u_{\epsilon}\vert^{\frac{p}{p-1}}}dr\Big)^{\frac{m-1}{m}}\le c_1
    \end{aligned}
\end{equation}
and 
\begin{equation}\label{graphi-3}
    \begin{aligned}
     \int_{0}^{1/2}r^{\sigma}\vert u_{\epsilon}\vert^{p-2}\vert u^{\prime}_{\epsilon}\vert dr \le C \|u^{\prime}_{\epsilon}\|_{L^{m}_{\sigma}(0,1)}\|u_{\epsilon}\|^{p-2}_{L^{\frac{m(p-2)}{m-1}}_{\sigma}(0,1)}\le c_2.
    \end{aligned}
\end{equation}
Thus, from \eqref{graphi-1}, \eqref{graphi-2} and \eqref{graphi-3} it follows that 
\begin{equation}\label{claim2}
  \varphi^{\prime}\in L^{1}_{\sigma}(0,1/2).
\end{equation}
Hence, in the same line of \cite[Lemma~6]{RUDOL}, combining \eqref{grad-Lm} and \eqref{claim2} we conclude that \eqref{varphi-zerolimit} holds.

\noindent Following  \cite{Clement-deFigueiredo-Mitidieri}, for each $r, \delta>0$ we  consider the test function $v_{\delta}\in X^{1,p}_{\infty}$ given by
 \begin{equation}\label{test-magic}
     v_{\delta}(s)=\left\{\begin{aligned}
     &1\;\;&\mbox{if}&\;\; 0<s\le r,\\
     &1+\frac{1}{\delta}(r-s)\;\;&\mbox{if}&\;\; r\le s\le r+\delta,\\
     & 0\;\;&\mbox{if}&\;\; s\ge r+\delta.
     \end{aligned}
     \right.
 \end{equation}
 By using $v_{\delta}$ in \eqref{Euler-Lagrange} and letting $\delta\rightarrow 0$, we get the integral equation
 \begin{equation}\label{eq.integral}
     (-u_{\epsilon}^{\prime}(r))^{p-1}=\frac{1}{\omega_{\alpha}r^{\alpha}}\int_{0}^{r}\big[\frac{b_{\epsilon}}{d_{\epsilon}}\varphi_{p}^{\prime}\big(\eta_{\epsilon}\vert u_{\epsilon}(s)\vert ^{\frac{p}{p-1}}\big)\vert u_{\epsilon}\vert^{\frac{1}{p-1}}+(c_{\epsilon}-1)\vert u_{\epsilon}(s)\vert^{p-1}\big]\,\mathrm{d}\lambda_{\theta}.
 \end{equation}
 It follows that $u_{\epsilon}\in C^{2}(0,\infty)$. In addition, since we are supposing $\theta\ge \alpha$, the  L'Hospital rule and \eqref{limite-exzero} imply  $u^{\prime}_{\epsilon}(0)=0$, and thus $u_{\epsilon}\in C^{1}[0,\infty)$.
\end{proof}
\begin{lemma} \label{de-maxcri}
We have $AD(\eta, \mu_{\epsilon}, \alpha,\theta) \rightarrow AD(\eta, \mu_{\alpha,\theta}, \alpha,\theta)$, as $\epsilon\rightarrow 0$.  In particular, $\displaystyle\liminf_{\epsilon\rightarrow 0}d_{\epsilon}>0$. 
\end{lemma}
\begin{proof}
For any $u\in X^{1,p}_{\infty}$ with $\|u\|\le 1$, the Fatou's Lemma yields
\begin{equation}\nonumber
    \begin{aligned}
 \liminf_{\epsilon\rightarrow 0} AD(\eta, \mu_{\epsilon}, \alpha,\theta) 
    &\ge \int_{0}^{\infty}\varphi_{p}\big(\mu_{\alpha,\theta}(1+\eta \|u\|^{p}_{L^{p}_{\theta}})^{\frac{1}{p-1}}\vert u\vert ^{\frac{p}{p-1}}\big)\mathrm{d}\lambda_{\theta}.
    \end{aligned}
\end{equation}
Since $u$ is taken arbitrary, we get
\begin{equation}\nonumber
    \begin{aligned}
    \liminf_{\epsilon\rightarrow 0}  AD(\eta, \mu_{\epsilon}, \alpha,\theta)  \ge AD(\eta, \mu_{\alpha,\theta}, \alpha,\theta).
    \end{aligned}
\end{equation}
Also, we clearly have 
\begin{equation}\nonumber
    \begin{aligned}
    \limsup _{\epsilon\rightarrow 0}  AD(\eta, \mu_{\epsilon}, \alpha,\theta)\le  AD(\eta, \mu_{\alpha,\theta}, \alpha,\theta).
    \end{aligned}
\end{equation}
Then we can write 
\begin{equation}\label{MTe-limit}
    \begin{aligned}
    \lim_{\epsilon\rightarrow 0}  AD(\eta, \mu_{\epsilon}, \alpha,\theta)  = AD(\eta, \mu_{\alpha,\theta}, \alpha,\theta).
    \end{aligned}
\end{equation}
Noticing that 
\begin{equation}\nonumber
    t\varphi^{\prime}_{p}(t)=\sum_{j=k_0}^{\infty}\frac{t^{j}}{(j-1)!}\ge \sum_{j=k_0}^{\infty}\frac{t^{j}}{j!}=\varphi_{p}(t), \;\; t\ge0
\end{equation}
we can see that
\begin{equation}\nonumber 
    \begin{aligned}
    d_{\epsilon} & \ge \frac{1}{\eta_{\epsilon}}AD(\eta, \mu_{\epsilon}, \alpha,\theta)
   \ge \frac{1}{\mu_{\epsilon}(1+\eta)^{\frac{1}{p-1}}}AD(\eta, \mu_{\epsilon}, \alpha,\theta).
    \end{aligned}
\end{equation}
Thus, using \eqref{MTe-limit}, we get 
\begin{equation}
    \begin{aligned}
  \liminf_{\epsilon\rightarrow 0} d_{\epsilon}\ge \frac{1}{\mu_{\alpha,\theta}(1+\eta)^{\frac{1}{p-1}}}AD(\eta, \mu_{\alpha,\theta}, \alpha,\theta)>0.
    \end{aligned}
\end{equation}
\end{proof}

In the sequel, we do not distinguish sequence and subsequence. As well as in \eqref{loc-convergence} and \eqref{ab-limite} we have
\begin{equation}\label{eploc-convergence}
u_{\epsilon}\rightharpoonup u_0\;\mbox{in}\;X^{1,p}_{\infty}(\alpha,\theta),\;\; u_{\epsilon}\rightarrow u_0\;\mbox{in}\;L^{q}_{\theta}(0,R)\;\mbox{and}\; u_{\epsilon}(r)\rightarrow u_0(r)\;\mbox{a.e in}\; (0,\infty),
\end{equation} 
for any $ R>0$ and $q\in (1,\infty)$. In addition, we can pick $a\in [0,1]$ such that 
\begin{equation}\label{epab-limite}
\|u_{\epsilon}\|^{p}_{L^{p}_{\theta}}\rightarrow a.
\end{equation}
Moreover, since $u_{\epsilon}\in C^{1}[0,\infty)$ is a  non-increasing function  we can define
\begin{equation}\label{a-blow}
    a_{\epsilon}:=u_{\epsilon}(0)=\max_{r\in [0,\infty)}u_{\epsilon}(r).
\end{equation}
\begin{lemma}\label{lemma-abounded}
If $(a_{\epsilon})$ is bounded, then $ AD(\eta, \mu_{\alpha,\theta}, \alpha,\theta)$ is attained.
\end{lemma}
\begin{proof}
Assume that there is $c>0$ such that $a_{\epsilon}\le c$, for any $\epsilon>0$. Hence,  $u_{\epsilon}(r)\le c$, for $r\in (0,\infty)$ and  uniformly on $\epsilon$. Consequently, for any $R>0$ there is $q>1$ such that 
\begin{equation}\nonumber
\begin{aligned}
\sup_{\epsilon>0}\int_{0}^{R}\big[\varphi_{p}\big(\eta_{\epsilon}\vert u_{\epsilon}\vert^{\frac{p}{p-1}}\big)\big]^{q}\mathrm{d}\lambda_{\theta} <\infty.
\end{aligned}
\end{equation}
From \eqref{eploc-convergence} and the Vitali's convergence theorem
\begin{equation}\label{blow-sub-loc}
\lim_{\epsilon\rightarrow0}\int^{R}_{0}\varphi_{p}\big(\eta_{\epsilon}\vert u_{\epsilon}\vert^{\frac{p}{p-1}}\big)\mathrm{d}\lambda_{\theta}=\int^{R}_{0}\varphi_{p}\big(\mu_{\alpha,\theta}(1+\eta a)^{\frac{1}{p-1}}\vert u_0\vert^{\frac{p}{p-1}}\big)\mathrm{d}\lambda_{\theta}.
\end{equation}
Also, 
\begin{equation}\label{blow-lq-convergence}
\lim_{\epsilon\rightarrow 0}\int^{R}_{0}(1+\eta\|u_{\epsilon}\|^{p}_{L^{p}_{\theta}})\vert u_{\epsilon}\vert^{p}\mathrm{d}\lambda_{\theta}=(1+\eta a)\int_{0}^{R}\vert u_0\vert^{p}\mathrm{d} \lambda_{\theta}.
\end{equation}
For $u_0$ and $a$ are given by \eqref{eploc-convergence} and \eqref{epab-limite},  as in \eqref{NnotN} we will prove that
\begin{equation}\label{Blow-NnotN}
AD(\eta, \mu_{\alpha,\theta}, \alpha,\theta)=\left\{\begin{aligned}
& \int_{0}^{\infty}\varphi_{p}\big(\mu_{\alpha,\theta}(1+\eta a)^{\frac{1}{p-1}}\vert u_0\vert^{\frac{p}{p-1}}\big)\mathrm{d}\lambda_{\theta}, & \;\;\mbox{if}\;\; p\not\in \mathbb{N}\\
& \left. \begin{aligned}
&\int_{0}^{\infty}\varphi_{p}\big(\mu_{\alpha,\theta}(1+\eta a)^{\frac{1}{p-1}}\vert u_0\vert^{\frac{p}{p-1}}\big)\mathrm{d}\lambda_{\theta}\\
&+\frac{\mu^{p-1}_{\alpha,\theta}}{(p-1)!}(1+\eta a)\big(a-\|u_0\|^{p}_{L^{p}_{\theta}}\big)
\end{aligned}\right\} & \mbox{if}\;\; p\in \mathbb{N}.\\
\end{aligned}\right.
\end{equation}
Firstly,  \eqref{radial lemma} yields
\begin{equation}\label{blow-tail-1}
\begin{aligned}
& \varphi_{p}\big(\eta_{\epsilon}\vert u_{\epsilon}\vert^{\frac{p}{p-1}}\big)-\frac{\eta^{k_0}_{\epsilon}}{k_0!}\vert u_{\epsilon}\vert^{\frac{pk_0}{p-1}} \le C^{\prime}\frac{\vert u_{\epsilon}\vert^{p}}{R^{\frac{\theta+1}{p}}},\quad r\ge R\ge 1
\end{aligned}
\end{equation}
where $C^{\prime}$  does not dependent of $\epsilon$ and $R$. Hence, if $p\in\mathbb{N}$, i.e.  $k_0=p-1$ we obtain
\begin{equation}\label{rabophi}
\begin{aligned}
\lim_{R\rightarrow\infty}\lim_{\epsilon\rightarrow 0}\int_{R}^{\infty}\big[\varphi_{p}\big(\eta_{\epsilon}\vert u_{\epsilon}\vert^{\frac{p}{p-1}}\big)-\frac{\eta^{p-1}_{\epsilon}}{(p-1)!}\vert u_{\epsilon}\vert^{p}\big]\mathrm{d}\lambda_{\theta}=0.
\end{aligned}
\end{equation}
Combining \eqref{MTe-limit}, \eqref{blow-sub-loc}, \eqref{blow-lq-convergence} and \eqref{rabophi} we get \eqref{Blow-NnotN} if $p\in\mathbb{N}$. Now, suppose that $p>2$ is not an integer number, i.e.  $k_0>p-1$. Then, arguing as in \eqref{tail-1null}, for $r>R$ we can write
\begin{equation}\label{blow-tail-1null}
\begin{aligned}
 \varphi_{p}\big(\eta_{\epsilon}\vert u_{\epsilon}\vert^{\frac{p}{p-1}}\big)
& \le \frac{C_{2}}{r^{(\theta+1)\frac{k_0}{p-1}}}+\frac{ C_1\vert u_{\epsilon}\vert^{p}}{R^{\frac{\theta+1}{p}}}.
\end{aligned}
\end{equation}
Hence,
\begin{equation}\label{rabophi-notZ}
\begin{aligned}
\lim_{R\rightarrow\infty}\lim_{\epsilon\rightarrow 0}\int_{R}^{\infty}\varphi_{p}\big(\eta_{\epsilon}\vert u_{\epsilon}\vert^{\frac{p}{p-1}}\big)\mathrm{d}\lambda_{\theta}=0.
\end{aligned}
\end{equation}
Combining \eqref{MTe-limit}, \eqref{blow-sub-loc}  and \eqref{rabophi-notZ} we 
get \eqref{Blow-NnotN} for the case $p\not\in\mathbb{N}$. If $u_0\equiv0$ then \eqref{Blow-NnotN} implies 
\begin{equation}\nonumber
0<AD(\eta, \mu_{\alpha,\theta}, \alpha,\theta)=\left\{\begin{aligned}
& 0, & \mbox{if}\;\; p\not\in \mathbb{N}\\
& \frac{\mu^{p-1}_{\alpha,\theta}}{(p-1)!}(1+\eta a)a\le \frac{\mu^{p-1}_{\alpha,\theta}}{(p-1)!}(1+\eta),
 & \mbox{if}\;\; p\in \mathbb{N},
\end{aligned}\right.
\end{equation}
which contradicts the Proposition~\ref{vanishing-integer}. Thus, $u_0\not\equiv0$ and by setting $v_0(r)=u_0(r/\tau)$, with $\tau=(a/\|u_0\|^{p}_{L^{p}_{\theta}})^{1/(\theta+1)}$ we can argue as in the proof of Theorem~\ref{thm-extremal} to conclude the result.
\end{proof}

\noindent In view of Lemma~\ref{lemma-abounded}, without loss of generality,  in the sequel we are supposing  the condition
\begin{equation}\label{a-blowup}
\lim_{\epsilon\rightarrow
\infty}a_{\epsilon}=+\infty.
\end{equation}
To complete the proof of Theorem~\ref{thm-cmax},  we shall apply the two-step strategy of Carleson-Chang \cite{CC}, namely: Supposing the condition \eqref{a-blowup}, we exhibit an explicit constant $\mathcal{U}(\eta,\alpha,\theta)$ so that
\paragraph{\textbf{Step~1}:}  (Lemma~\ref{abaixo})
\begin{equation}\label{S<=}
AD(\eta, \mu_{\alpha,\theta}, \alpha,\theta)\le \mathcal{U}(\eta,\alpha,\theta), \quad\mbox{for any}\quad 0\le \eta<1.
\end{equation}

\paragraph{\textbf{Step~2}:} (Lemma~\ref{acima}) For each $\eta$ small enough, there is  $v_{\eta}\in X^{1,p}_{\infty}$ with $\|v_{\eta}\|=1$ so that 
\begin{equation}\label{S>}
AD(\eta, \mu_{\alpha,\theta}, \alpha,\theta)\ge \int_{0}^{\infty}\varphi_{p}\big(\mu_{\alpha,\theta}(1+\eta\|v_{\eta}\|^{p}_{L^{p}_{\theta}})^{\frac{1}{p-1}}\vert v_{\eta}\vert^{\frac{p}{p-1}}\big)\mathrm{d}\lambda_{\theta}> \mathcal{U}(\eta,\alpha,\theta).
\end{equation}
The contradiction given by \eqref{S<=} and \eqref{S>} excludes \eqref{a-blowup} and Theorem~\ref{thm-cmax} follows from Lemma~\ref{lemma-abounded}. The proof of both Step~1 and Step~2 will be divided into a series of lemmas.
\begin{lemma}\label{we-concentra} The sequence $(u_{\epsilon})$ is  concentrating at the origin, i.e,  
\begin{equation}\label{concentrada}
\|u^{\prime}_{\epsilon}\|_{L^{p}_{\alpha}}\le 1,\;\; u_{\epsilon}\rightharpoonup 0\;\;\mbox{and}\;\; \lim_{\epsilon\rightarrow 0}\int_{r_0}^{\infty}\vert u^{\prime}_{\epsilon}\vert^{p}\mathrm{d}\lambda_{\alpha}=0,\;\;\mbox{for any}\;\; r_0>0.
\end{equation} 
In particular, $b_{\epsilon}\rightarrow 1 $,  $c_{\epsilon}\rightarrow \eta $ and $\eta_{\epsilon}\rightarrow \mu_{\alpha,\theta}$, as $\epsilon\rightarrow 0$.
\end{lemma}
\begin{proof}
Fix $r_0>0$. We claim that 
\begin{equation}\label{no-concentrada-claim}
\lim_{\epsilon\rightarrow 0}\int_{0}^{r_0}\vert u^{\prime}_{\epsilon}\vert^{p}\mathrm{d}\lambda_{\alpha}=1.
\end{equation}
Of course we have $\|u^{\prime}_{\epsilon}\|_{L^{p}_{\alpha}(0,r_0)}\le 1$. By contradiction, suppose that there is $0\le \delta<1$  such that  $\|u^{\prime}_{\epsilon}\|^{p}_{L^{p}_{\alpha}(0,r_0)}\rightarrow \delta$, as $\epsilon\rightarrow0$. By setting $v_{\epsilon}=u_{\epsilon}-u_{\epsilon}(r_0)$ on $(0,r_0]$, we obtain $v_{\epsilon}\in X^{1,p}_{r_0}(\alpha,\theta)$ with
\begin{equation}\label{grad-vdeltar0}
   \lim_{\epsilon\rightarrow 0} \int_{0}^{r_0}\vert v^{\prime}_{\epsilon}\vert^{p}\mathrm{d}\lambda_{\alpha}=\lim_{\epsilon\rightarrow 0}\int_{0}^{r_0}\vert u^{\prime}_{\epsilon}\vert^{p}\mathrm{d}\lambda_{\alpha}=\delta.
\end{equation}
For any $\sigma>0$, from \eqref{radial lemma} and  \eqref{elementary-ineq}
\begin{equation}\nonumber
\begin{aligned}
\vert u_{\epsilon}\vert ^{\frac{p}{p-1}} &\le (1+\sigma)^{\frac{1}{p}}\vert v_{\epsilon}\vert^{\frac{p}{p-1}}+c_{\sigma}r_{0}^{-\frac{\theta+1}{p}},\;\; \mbox{in}\;\; (0, r_0]
\end{aligned}
\end{equation}
for some $c_{\sigma}>0$ depending only on $\sigma, p$ and $\theta$. Define  $w_{\epsilon}=v_{\epsilon}/\|v^{\prime}_{\epsilon}\|_{L^{p}_{\alpha}}$. Then
\begin{equation}\nonumber
\begin{aligned}
\varphi^{\prime}_{p}\big(\eta_{\epsilon}\vert u_{\epsilon}\vert^{\frac{p}{p-1}}\big) 
&\le c e^{\mu_{\epsilon}(1+\sigma)^{\frac{1}{p}}[(1+\eta\|u_{\epsilon}\|^{p}_{L^{p}_{\theta}})\|v^{\prime}_{\epsilon}\|^{p}_{L^{p}_{\alpha}}]^{\frac{1}{p-1}}\vert w_{\epsilon}\vert^{\frac{p}{p-1}}}
\end{aligned}
\end{equation}
for some
$
c=c(\alpha,\theta, r_0,\sigma)>0.
$
Note that $\|u_{\epsilon}\|=1$ yields
\begin{equation}\nonumber
    \begin{aligned}
    (1+\eta\|u_{\epsilon}\|^{p}_{L^{p}_{\theta}})\|v^{\prime}_{\epsilon}\|^{p}_{L^{p}_{\alpha}}
    & \le  (1+\eta-\eta\|v^{\prime}_{\epsilon}\|^{p}_{L^{p}_{\alpha}})\|v^{\prime}_{\epsilon}\|^{p}_{L^{p}_{\alpha}}.
    \end{aligned}
\end{equation}
Hence, from \eqref{grad-vdeltar0}
\begin{equation}\nonumber
\begin{aligned}
   & \lim_{\epsilon\rightarrow 0}\mu_{\epsilon}(1+\sigma)^{\frac{1}{p}}[(1+\eta\|u_{\epsilon}\|^{p}_{L^{p}_{\theta}})\|v^{\prime}_{\epsilon}\|^{p}_{L^{p}_{\alpha}}]^{\frac{1}{p-1}}\\
    &\quad\quad \le \mu_{\alpha,\theta}(1+\sigma)^{\frac{1}{p}}[(1+\eta-\eta\delta)\delta]^{\frac{1}{p-1}}<\mu_{\alpha,\theta}
    \end{aligned}
\end{equation}
for $\sigma>0$ small enough (since $\eta,\delta<1$). Hence, we can choose $q>1$ such that 
\begin{equation}\label{q-perto1}
q^2\mu_{\epsilon}(1+\sigma)^{\frac{1}{p}}[(1+\eta\|u_{\epsilon}\|^{p}_{L^{p}_{\theta}})\|v^{\prime}_{\epsilon}\|^{p}_{L^{p}_{\alpha}}]^{\frac{1}{p-1}}<\mu_{\alpha,\theta}
\end{equation}
for all $\epsilon>0$ small enough. Thus, the above estimates and \eqref{TM1} imply 
\begin{equation}\label{TM-wr0}
\begin{aligned}
\int_{0}^{r_0}\big[\varphi^{\prime}_{p}\big(\eta_{\epsilon}\vert u_{\epsilon}\vert^{\frac{p}{p-1}}\big)\big]^{q^2}\mathrm{d}\lambda_{\theta} 
&\le c^{q^2} \int_{0}^{r_0}e^{\mu_{\alpha,\theta}\vert w_{\epsilon}\vert^{\frac{p}{p-1}}}\mathrm{d}\lambda_{\theta}<c_{1}
\end{aligned}
\end{equation}
where $c_{1}>0$ does not depend on $\epsilon$. Set 
\begin{equation}\nonumber 
   f_{\epsilon}(r)=\frac{1}{\omega_{\alpha}}\frac{b_{\epsilon}}{d_{\epsilon}}\varphi_{p}^{\prime}\big(\eta_{\epsilon}\vert u_{\epsilon}(r)\vert^{\frac{p}{p-1}}\big)\vert u_{\epsilon}\vert^{\frac{1}{p-1}}+(c_{\epsilon}-1)\vert u_{\epsilon}(r)\vert^{p-1}.
\end{equation}
From Lemma~\ref{de-maxcri}, there is $c_2>0$ such that
\begin{equation}\nonumber 
\begin{aligned}
  & \frac{1}{c_2}\int_{0}^{r_0}\vert f_{\epsilon}\vert^{q}\mathrm{d}\lambda_{\theta}\le \int_{0}^{r_0}\big[\varphi_{p}^{\prime}\big(\eta_{\epsilon}\vert u_{\epsilon}\vert^{\frac{p}{p-1}}\big)\big]^{q}\vert u_{\epsilon}\vert^{\frac{q}{p-1}}\mathrm{d}\lambda_{\theta}+\int_{0}^{r_0}\vert u_{\epsilon}\vert^{(p-1)q}\mathrm{d}\lambda_{\theta}.
   \end{aligned}
\end{equation}
Therefore, \eqref{TM-compactembeddings} (see Remark~\ref{remarkW}), \eqref{TM-wr0} and the H\"{o}lder inequality  yield
\begin{equation}\nonumber 
\begin{aligned}
   \frac{1}{c_2}\int_{0}^{r_0}\vert f_{\epsilon}\vert^{q}\mathrm{d}\lambda_{\theta}& \le \Big(\int_{0}^{r_0}\big[\varphi^{\prime}_{p}\big(\eta_{\epsilon}\vert u_{\epsilon}\vert^{\frac{p}{p-1}}\big)\big]^{q^2}\mathrm{d}\lambda_{\theta}\Big)^{\frac{1}{q}}\Big(\int_{0}^{r_0}\vert u_{\epsilon}\vert^{\frac{q^2}{(q-1)(p-1)}}\mathrm{d}\lambda_{\theta}\Big)^{\frac{q-1}{q}}\\
  &+ \Big( \int_{0}^{r_0}\vert u_{\epsilon}\vert^{(p-1)\frac{q^2}{q-1}}\mathrm{d}\lambda_{\theta}\Big)^{\frac{q-1}{q}}\Big(\int_{0}^{r_0}\mathrm{d}\lambda_{\theta}\Big)^{\frac{1}{q}}\le c,
   \end{aligned}
\end{equation}
 where we have chosen $q>1$ (close $1$) such that $q^2/(q-1)>(p-1)p$ and \eqref{q-perto1} hold. In particular, 
 \begin{equation}\label{f-limitado}
\int_{0}^{r_0}\vert f_{\epsilon}\vert^{q}\mathrm{d}\lambda_{\theta}\le c, 
 \end{equation}
 for all $\epsilon>0$ small enough.  Now, from \eqref{eq.integral}, for any $0\le r<r_0$ we can write 
 \begin{equation}\nonumber 
     \begin{aligned}
     u_{\epsilon}(r) & \le u_{\epsilon}(r_0)+\Big(\frac{\omega_{\theta}}{\theta+1}\Big)^{\frac{1}{p-1}}\Big(\int_{0}^{r_0}\vert f_{\epsilon}\vert^{q}\mathrm{d}\lambda_{\theta}\Big)^{\frac{1}{q(p-1)}}\int_{0}^{r_0}s^{(\theta+1)\frac{q-1}{q(p-1)}-1}ds.
     \end{aligned}
 \end{equation}
 Hence, from \eqref{radial lemma} and  \eqref{f-limitado} we obtain 
$ a_{\epsilon}=u_{\epsilon}(0)\le c$. This, contradicts \eqref{a-blowup} and proves \eqref{no-concentrada-claim}.
 
 Next we will prove \eqref{concentrada}. By contradiction,  suppose that there are $0<A<1$ and $r_0>0$ such that 
\begin{equation}\nonumber
\lim_{\epsilon\rightarrow 0}\int_{r_0}^{\infty}\vert u^{\prime}_{\epsilon}\vert^{p}\mathrm{d}\lambda_{\alpha}>A.
\end{equation} 
Thus,
\begin{equation}\nonumber
   1=\|u_{\epsilon}\|^{p}_{L^{p}_{\theta}}+\|u^{\prime}_{\epsilon}\|^{p}_{L^{p}_{\alpha}}>\|u_{\epsilon}\|^{p}_{L^{p}_{\theta}}+\int_{0}^{r_0}\vert u^{\prime}_{\epsilon}\vert^{p}\mathrm{d}\lambda_{\alpha}+A
\end{equation}
and consequently
\begin{equation}\nonumber
   \int_{0}^{r_0}\vert u^{\prime}_{\epsilon}\vert^{p}\mathrm{d}\lambda_{\alpha}<1-A.
\end{equation}
Hence, we get 
\begin{equation}\nonumber
   \lim_{\epsilon\rightarrow 0}\int_{0}^{r_0}\vert u^{\prime}_{\epsilon}\vert^{p}\mathrm{d}\lambda_{\alpha}=\delta\le 1-A<1,
\end{equation}
which contradicts \eqref{no-concentrada-claim}. So, \eqref{concentrada} holds. Finally, from  \eqref{concentrada} and \eqref{no-concentrada-claim}
\begin{equation}\nonumber
\begin{aligned}
   1&=\lim_{\epsilon\rightarrow 0}\Big[\|u_{\epsilon}\|^{p}_{L^{p}_{\theta}}+\int_{0}^{1}\vert u^{\prime}_{\epsilon}\vert^{p}\mathrm{d}\lambda_{\alpha}+\int_{1}^{\infty}\vert u^{\prime}_{\epsilon}\vert^{p}\mathrm{d}\lambda_{\alpha}\Big]
 \ge \lim_{\epsilon\rightarrow 0}\|u_{\epsilon}\|^{p}_{L^{p}_{\theta}}+1.
 \end{aligned}
\end{equation}
Consequently $u_{\epsilon}\rightarrow 0$ in $L^{p}_{\theta}$ and, from \eqref{eploc-convergence},  we get  $u_{\epsilon}\rightharpoonup u_{0}\equiv 0$.
\end{proof}

In order to investigate the behavior of the sequence $(u_{\epsilon})$ around of blowing up point $r=0$, we
consider the following auxiliary functions
\begin{equation}\label{specialfunctions}
    \left\{\begin{aligned}
   & v_{\epsilon}(r)=\frac{u_{\epsilon}(r_{\epsilon}r)}{a_{\epsilon}}\\
 &w_{\epsilon}(r)=a_{\epsilon}^{\frac{1}{p-1}}\left(u_{\epsilon}(r_{\epsilon}r)-a_{\epsilon}\right)
    \end{aligned}
    \right.,\quad r\in (0,\infty)
\end{equation}
where
\begin{equation}\label{repsilon}
    r_{\epsilon}^{\theta+1}=\frac{d_{\epsilon}}{b_{\epsilon}}a^{-\frac{p}{p-1}}_{\epsilon}e^{-\eta_{\epsilon}a^{\frac{p}{p-1}}_{\epsilon}}.
\end{equation}
\begin{lemma} For any $0<\mu<\mu_{\alpha,\theta}$, we have
$$
\lim_{\epsilon\rightarrow 0}r^{\theta+1}_{\epsilon}a^{\frac{p}{p-1}}_{\epsilon}e^{\mu a^{\frac{p}{p-1}}_{\epsilon}}=0.
$$
In particular, $\lim_{\epsilon\rightarrow0} r^{\theta+1}_{\epsilon}=0.$
\end{lemma}
\begin{proof}
For any $R>0$
\begin{equation}\label{re-ptudo}
    \begin{aligned}
    r^{\theta+1}_{\epsilon}a^{\frac{p}{p-1}}_{\epsilon}e^{\mu a^{\frac{p}{p-1}}_{\epsilon}}& =\frac{e^{(\mu-\eta_{\epsilon})a^{\frac{p}{p-1}}_{\epsilon}}}{b_{\epsilon}}\int_{0}^{R}\vert u_{\epsilon}\vert ^{\frac{p}{p-1}}\varphi_{p}^{\prime}\big(\eta_{\epsilon}\vert u_{\epsilon}\vert^{\frac{p}{p-1}}\big)\,\mathrm{d}\lambda_{\theta}\\
    & +\frac{e^{(\mu-\eta_{\epsilon})a^{\frac{p}{p-1}}_{\epsilon}}}{b_{\epsilon}}\int_{R}^{\infty}\vert u_{\epsilon}\vert^{\frac{p}{p-1}}\varphi_{p}^{\prime}\big(\eta_{\epsilon}\vert u_{\epsilon}\vert ^{\frac{p}{p-1}}\big)\,\mathrm{d}\lambda_{\theta}\\
    & = I_1+I_2.
    \end{aligned}
\end{equation}
Since   $\eta_{\epsilon}\rightarrow\mu_{\alpha,\theta}>\mu$, for $\epsilon$ sufficiently small we obtain
\begin{equation}\nonumber
    (\eta_{\epsilon}-\mu)(\vert u_{\epsilon}(r)\vert^{\frac{p}{p-1}}-a^{\frac{p}{p-1}}_{\epsilon})\le 0, \;\; r>0.
\end{equation}
Then, using that $\varphi^{\prime }_{p}(t)\le e^{t},\, t\ge 0$, we can write
\begin{equation}\nonumber 
    \begin{aligned}
   I_1 & \le \frac{1}{b_{\epsilon}}\int_{0}^{R}\vert u_{\epsilon}\vert^{\frac{p}{p-1}}e^{\mu\vert u_{\epsilon}\vert^{\frac{p}{p-1}}}\,\mathrm{d}\lambda_{\theta}.\\
    \end{aligned}
\end{equation}
Now, by choosing $q>1$ such that $q\mu<\mu_{\alpha,\theta}$ and since $b_{\epsilon}\rightarrow 1$,  from Lemma~\ref{prop1}-$(i)$ and the convergence in \eqref{eploc-convergence} with $u_0\equiv 0$, we obtain 
\begin{equation}\label{reI1}
    \begin{aligned}
   I_1  & \le \frac{1}{b_{\epsilon}}\Big(\int_{0}^{R}\vert u_{\epsilon}\vert^{\frac{pq}{(p-1)(q-1)}}\,\mathrm{d}\lambda_{\theta}\Big)^{\frac{q-1}{q}}\Big(\int_{0}^{R}e^{q\mu \vert u_{\epsilon}\vert^{\frac{p}{p-1}}}\,\mathrm{d}\lambda_{\theta}\Big)^{\frac{1}{q}}\rightarrow 0, \;\; \mbox{as}\;\; \epsilon\rightarrow 0.
    \end{aligned}
\end{equation}
From \eqref{radial lemma},  for any $r\ge R$ we have $\vert u_{\epsilon}(r)\vert^{\frac{p}{p-1}}\le 
{c}/{R^{\frac{\theta+1}{p}}}$ for some $c>0$ which is independent of $\epsilon$ and $R$. Hence, since $\eta\|u_{\epsilon}\|^{p}_{L^{p}_{\theta}}\le 1$ and $\mu_{\epsilon}\le \mu_{\alpha,\theta}$ we can write 
\begin{equation}\nonumber
   \eta_{\epsilon}\vert u_{\epsilon}\vert^{\frac{p}{p-1}}\le \mu_{\alpha,\theta}2^{\frac{1}{p-1}}CR^{-\frac{\theta+1}{p}},\;\; \forall r\ge R.
\end{equation}
Hence, noticing that for any $T>0$ we have 
$
  t\varphi^{\prime}_{p}(t)\le t^{k_0}e^{t}\le t^{k_0}e^{T},
$
for $0\le t\le T$
we can write (use $k_0\ge p-1$)
\begin{equation}\label{Fe-later}
   \begin{aligned}
    \vert u_{\epsilon}\vert^{\frac{p}{p-1}}\varphi_{p}^{\prime}\big(\eta_{\epsilon}\vert u_{\epsilon}\vert^{\frac{p}{p-1}}\big)
   & \le \frac{e^{\mu_{\alpha,\theta}2^{\frac{1}{p-1}}CR^{-\frac{\theta+1}{p}}}}{\eta_{\epsilon}}\eta_{\epsilon}^{k_0}\vert u_{\epsilon}\vert^{\frac{k_0p}{p-1}}
    \le  c(\alpha, \eta, \theta, R)\vert u_{\epsilon}\vert^{p}, \;\; \forall r\ge R
   \end{aligned} 
\end{equation}
where $R>0$ is large such that  $\vert u_{\epsilon}(r)\vert\le 1$ for $r\ge R$.  Hence, since  $u_{\epsilon}\rightarrow 0$ in $L^{p}_{\theta}$ we obtain
\begin{equation}\label{reI2}
  I_2\le   c(\alpha, \eta, \theta, R)\frac{e^{(\mu-\eta_{\epsilon})a^{\frac{p}{p-1}}_{\epsilon}}}{b_{\epsilon}}\int_{R}^{\infty}\vert u_{\epsilon}\vert^{p}\mathrm{d}\lambda_{\theta} \rightarrow 0,\;\; \mbox{as}\;\; \epsilon\rightarrow 0.
\end{equation}
Combining \eqref{re-ptudo}, \eqref{reI1} and \eqref{reI2} we get the result.
\end{proof}
\begin{lemma}\label{vto1wtolog} Let $v_{\epsilon}$ and $w_{\epsilon}$ given by \eqref{specialfunctions}. Set
\begin{equation}\label{w-express}
    w(r)=-\frac{p-1}{\mu_{\alpha,\theta}}\ln\big(1+c_{\alpha,\theta}r^{\frac{\theta+1}{p-1}}\big),\;\;\mbox{with}\;\; c_{\alpha,\theta}=\Big(\frac{\omega_{\theta}}{\theta+1}\Big)^{\frac{1}{p-1}}.
\end{equation}
Then $v_{\epsilon}\rightarrow 1$ in $C^{1}_{loc}[0,\infty)$ and $w_{\epsilon}\rightarrow w$ in $C^{0}_{loc}[0,\infty)$. In addition,
\begin{equation}\label{w-normal}
    \int_{0}^{\infty}e^{\frac{p}{p-1}\mu_{\alpha,\theta}w(r)}\mathrm{d}\lambda_{\theta}=1.
\end{equation}
\end{lemma}
\begin{proof}
By using  integral equation \eqref{eq.integral}, from the definition in \eqref{specialfunctions} and \eqref{repsilon},   it is easy to see that 
\begin{equation}\nonumber
\begin{aligned}
\omega_{\alpha}\vert v^{\prime}_{\epsilon}(r)\vert^{p-1}
&=\frac{1}{r^{\alpha}}\frac{e^{-\eta_{\epsilon}a^{\frac{p}{p-1}}_{\epsilon}}}{a^{p}_{\epsilon}}\int_{0}^{r}\big[\varphi_{p}^{\prime}\big(\eta_{\epsilon}\vert v_{\epsilon}(s)\vert^{\frac{p}{p-1}}a^{\frac{p}{p-1}}_{\epsilon}\big)\vert v_{\epsilon}(s)\vert^{\frac{1}{p-1}}\big]\,\mathrm{d}\lambda_{\theta}\\
&+\frac{(c_{\epsilon}-1)r^{\theta+1}_{\epsilon}}{r^{\alpha}}\int_{0}^{r}\vert v_{\epsilon}(s)\vert^{p-1}\,\mathrm{d}\lambda_{\theta}.
\end{aligned}
\end{equation}
Note that
\begin{equation}\label{splitphiprime}
   \varphi^{\prime}_{p}(t)=\left\{ \begin{aligned}
    & e^{t},\;\;&\mbox{if}&\;\; p=2\\
    & e^{t}-\sum^{k_0-2}_{j=0}\frac{t^{j}}{j!},\;\;&\mbox{if}&\;\; p>2\\
    \end{aligned}\right.,\;\; (t\ge 0).
\end{equation}
Thus, we can write
\begin{equation}\label{veeq.integral}
\begin{aligned}
\omega_{\alpha}\vert v^{\prime}_{\epsilon}(r)\vert^{p-1} &=\frac{1}{a^{p}_{\epsilon}}\frac{1}{r^{\alpha}}\int_{0}^{r}e^{\eta_{\epsilon}a^{\frac{p}{p-1}}_{\epsilon}(\vert v_{\epsilon}\vert^{\frac{p}{p-1}}-1)}\vert v_{\epsilon}\vert^{\frac{1}{p-1}}\,\mathrm{d}\lambda_{\theta}\\
&-\frac{e^{-\eta_{\epsilon}a^{\frac{p}{p-1}}_{\epsilon}}}{a^{p}_{\epsilon}}\sum^{k_0-2}_{j=0}\frac{\eta^{j}_{\epsilon}a^{\frac{jp}{p-1}}_{\epsilon}}{j!}\frac{1}{r^{\alpha}}\int_{0}^{r}\vert v_{\epsilon}\vert^{\frac{jp+1}{p-1}}\mathrm{d}\lambda_{\theta}\\
&+\frac{(c_{\epsilon}-1)r^{\theta+1}_{\epsilon}}{r^{\alpha}}\int_{0}^{r}\vert v_{\epsilon}\vert^{p-1}\,\mathrm{d}\lambda_{\theta},
\end{aligned}
\end{equation}
where the second term on the right hand side does not appear if $p=2$. Fix $r_0>0$. Since $\theta\ge \alpha$,  $\vert v_{\epsilon}\vert\le 1$, $a_{\epsilon}\rightarrow\infty$, $c_{\epsilon}\rightarrow\eta$ and $r^{\theta+1}_{\epsilon}\rightarrow 0$, we conclude from \eqref{veeq.integral} that $v^{\prime}_{\epsilon}\rightarrow 0$ uniformly on $[0,r_0]$. Since $v_{\epsilon}(0)=1$, we obtain $v_{\epsilon}\rightarrow 1$ in $C^{1}_{loc}[0,\infty)$. In addition, since $w^{\prime}_{\epsilon}=a^{\frac{p}{p-1}}_{\epsilon}v^{\prime}_{\epsilon}$, from \eqref{veeq.integral} we also have 
\begin{equation}\label{weeq.integral}
\begin{aligned}
\omega_{\alpha}\vert w^{\prime}_{\epsilon}(r)\vert^{p-1} &=\frac{1}{r^{\alpha}}\int_{0}^{r}e^{\eta_{\epsilon}a^{\frac{p}{p-1}}_{\epsilon}(\vert v_{\epsilon}\vert^{\frac{p}{p-1}}-1)}\vert v_{\epsilon}\vert^{\frac{1}{p-1}}\,\mathrm{d}\lambda_{\theta}\\
&-e^{-\eta_{\epsilon}a^{\frac{p}{p-1}}_{\epsilon}}\sum^{k_0-2}_{j=0}\frac{\eta^{j}_{\epsilon}a^{\frac{jp}{p-1}}_{\epsilon}}{j!}\frac{1}{r^{\alpha}}\int_{0}^{r}\vert v_{\epsilon}\vert^{\frac{jp+1}{p-1}}\mathrm{d}\lambda_{\theta}\\
&+\frac{(c_{\epsilon}-1)a^{p}_{\epsilon}r^{\theta+1}_{\epsilon}}{r^{\alpha}}\int_{0}^{r}\vert v_{\epsilon}\vert^{p-1}\,\mathrm{d}\lambda_{\theta}.
\end{aligned}
\end{equation}
Analogously, since $\theta\ge \alpha$,  $\vert v_{\epsilon}\vert\le 1$ and  $(c_{\epsilon}-1)a^{p}_{\epsilon}r^{\theta+1}_{\epsilon}\rightarrow 0$ we have from \eqref{weeq.integral} that $w^{\prime}_{\epsilon}$ is bounded in $C^{0}[0, r_0]$. Since $w_{\epsilon}(0) = 0$, we have that $(w_{\epsilon})$ is uniformly bounded equicontinuous sequence in $C[0, r_0]$. Thus, the Ascoli-Arzel\`{a} theorem
gives $w_{\epsilon}\rightarrow w$ uniformly for some  $w\in C[0, r_0]$. Next, we will show that $w$ has the expression given by \eqref{w-express}. To get this, we observe
\begin{equation}\label{welimit}
    a^{\frac{p}{p-1}}_{\epsilon}(\vert v_{\epsilon}(s)\vert^{\frac{p}{p-1}}-1)=\frac{p}{p-1}w_{\epsilon}(s)\big[1+O_{\epsilon}\big(\vert v_{\epsilon}\vert-1\big)\big].
\end{equation}
By integrating in \eqref{weeq.integral} on $(0,r)$ we obtain
\begin{equation}\nonumber 
\begin{aligned}
w_{\epsilon}(r)=-\int_{0}^{r}\Big(\frac{1}{\omega_{\alpha}t^{\alpha}}\int_{0}^{t}g_{\epsilon}(s)\mathrm{d}\lambda_{\theta}\Big)^{\frac{1}{p-1}}dt,
\end{aligned}
\end{equation}
where
\begin{equation}\nonumber 
\begin{aligned}
g_{\epsilon}(s)& =e^{\eta_{\epsilon}a^{\frac{p}{p-1}}_{\epsilon}(\vert v_{\epsilon}\vert^{\frac{p}{p-1}}-1)}\vert v_{\epsilon}\vert^{\frac{1}{p-1}}+(c_{\epsilon}-1)a^{p}_{\epsilon}r^{\theta+1}_{\epsilon}\vert v_{\epsilon}\vert^{p-1}\\
&-e^{-\eta_{\epsilon}a^{\frac{p}{p-1}}_{\epsilon}}\sum^{k_0-2}_{j=0}\frac{\eta^{j}_{\epsilon}a^{\frac{jp}{p-1}}_{\epsilon}}{j!}\vert v_{\epsilon}\vert^{\frac{jp+1}{p-1}}.
\end{aligned}
\end{equation}
Now, fixed $t>0$ arbitrary,  from \eqref{welimit} we obtain 
\begin{equation}\nonumber 
\begin{aligned}
\lim_{\epsilon\rightarrow 0} g_{\epsilon}(s) =e^{\mu_{\alpha,\theta}\frac{p}{p-1}w(s)},\;\; \forall\, s\in (0,t).
\end{aligned}
\end{equation}
Hence, 
\begin{equation}\nonumber 
\begin{aligned}
w(r)=-\int_{0}^{r}\Big(\frac{1}{\omega_{\alpha}t^{\alpha}}\int_{0}^{t}e^{\mu_{\alpha,\theta}\frac{p}{p-1}w(s)}\mathrm{d}\lambda_{\theta}\Big)^{\frac{1}{p-1}} dt.
\end{aligned}
\end{equation}
It is easy to show that $w$ must satisfy the equation
\begin{equation}\nonumber
    \left\{\begin{aligned}
    &-\omega_{\alpha}(r^{\alpha}\vert w^{\prime}\vert^{p-2}w^{\prime})^{\prime}=\omega_{\theta}r^{\theta}e^{\frac{p}{p-1}\mu_{\alpha,\theta}w(r)}\;\;\mbox{on}\;\; [0,\infty)\\
   &  w(0)=w^{\prime}(0)=0.
    \end{aligned}
    \right.
\end{equation}
Noticing that the  unique solution for the above ODE is the function given in \eqref{w-express} we get the desired expression for $w$. Performing the change of variable $s=c_{\alpha,\theta}r^{(\theta+1)/(p-1)}$ and using the  identities (see \cite{special})
\begin{equation}\nonumber
\Gamma(1)=1, \;\;\Gamma(x+1)=x\Gamma(x)\;\; \mbox{and}\;\;  
\int_{0}^{\infty}\frac{s^{x-1}}{(1+s)^{x+y}}ds=\frac{\Gamma(x)\Gamma(y)}{\Gamma(x+y)},\;\; x, y>0
\end{equation}
we obtain \eqref{w-normal}.
\end{proof}
\begin{lemma}\label{lemmauec} For each $c>1$, set $u_{\epsilon,c}=\min\left\{u_{\epsilon}, a_{\epsilon}/c\right\}$. Then
\begin{equation}\nonumber
    \lim_{\epsilon\rightarrow 0}\int_{0}^{\infty}\vert u^{\prime}_{\epsilon,c}\vert^{p}\mathrm{d}\lambda_{\alpha}=\frac{1}{c}.
\end{equation}
\end{lemma}
\begin{proof}
We have  $v_{\epsilon}\rightarrow 1$ in $C^{1}_{loc}[0,\infty)$ and then $u_{\epsilon}=a_{\epsilon}(1+o_{\epsilon}(R))$ uniformly on $(0, r_{\epsilon}R)$. Then
\begin{equation}\label{qovere}
  \frac{\int_{0}^{r_{\epsilon}R}\vert u_{\epsilon}\vert^{q}\mathrm{d}\lambda_{\theta}}{\int_{0}^{r_{\epsilon}R}e^{\eta_{\epsilon}\vert u_{\epsilon}\vert^{\frac{p}{p-1}}}\mathrm{d}\lambda_{\theta}}=o_{\epsilon}(R),\quad \forall\,q>1
\end{equation}
where $o_{\epsilon}(R)$ means that $\lim_{\epsilon\rightarrow 0}o_{\epsilon}(R)=0$   if $R$ is fixed. Since $\|u_{\epsilon}\|_{L^ {p}_{\theta}}=o_{\epsilon}(1)$, from \eqref{Euler-Lagrange}, we can write 
\begin{equation}\nonumber
\begin{aligned}
\int_{0}^{\infty}\vert u_{\epsilon,c}^{\prime}\vert^{p}\,\mathrm{d}\lambda_{\alpha}
 =\frac{b_{\epsilon}}{d_{\epsilon}}\int_{0}^{\infty}\vert u_{\epsilon}\vert^{\frac{1}{p-1}}\varphi_{p}^{\prime}\big(\eta_{\epsilon}\vert u_{\epsilon}\vert^{\frac{p}{p-1}}\big)u_{\epsilon,c}\,\mathrm{d}\lambda_{\theta}+o_{\epsilon}(1).
\end{aligned}
\end{equation}
In view of \eqref{splitphiprime} and \eqref{qovere} , for $\epsilon>0$ small enough such that $(0,r_{\epsilon}R)\subset\left\{u_{\epsilon}\ge a_{\epsilon}/c\right\}$, we have 
\begin{equation}\nonumber
\begin{aligned}
\int_{0}^{\infty}\vert u_{\epsilon,c}^{\prime}\vert^{p}\,\mathrm{d}\lambda_{\alpha}
 & \ge \frac{b_{\epsilon}}{d_{\epsilon}}\frac{a^{\frac{p}{p-1}}_{\epsilon}}{c}(1+o_{\epsilon}(R))\int_{0}^{r_{\epsilon}R}e^{\eta_{\epsilon}\vert u_{\epsilon}\vert^{\frac{p}{p-1}}}\,\mathrm{d}\lambda_{\theta}+o_{\epsilon}(1).
\end{aligned}
\end{equation}
By using the change of variable $r=r_{\epsilon}s$, from \eqref{repsilon}  we get
\begin{equation}\nonumber
\begin{aligned}
\int_{0}^{\infty}\vert u_{\epsilon,c}^{\prime}\vert^{p}\,\mathrm{d}\lambda_{\alpha}
& \ge  \frac{1}{c}(1+o_{\epsilon}(R))\int_{0}^{R}e^{\eta_{\epsilon}a^{\frac{p}{p-1}}_{\epsilon}(\vert v_{\epsilon}\vert^{\frac{p}{p-1}}-1)}\,\mathrm{d}\lambda_{\theta}+o_{\epsilon}(1).\\
\end{aligned}
\end{equation}
Using \eqref{w-normal} and \eqref{welimit},  setting $\epsilon\rightarrow 0$ and then $R\rightarrow\infty$ we obtain 
\begin{equation}\label{graduec1}
\begin{aligned}
\int_{0}^{\infty}\vert u_{\epsilon,c}^{\prime}\vert^{p}\,\mathrm{d}\lambda_{\alpha}\ge 
 \frac{1}{c}.
\end{aligned}
\end{equation}
Similarly, we obtain
\begin{equation}\label{graduec2}
\begin{aligned}
\int_{0}^{\infty}\vert((u_{\epsilon}-\frac{a_{\epsilon}}{c})^{+})^{\prime}\vert^{p}\,\mathrm{d}\lambda_{\alpha}\ge 
 \frac{c-1}{c},
\end{aligned}
\end{equation}
where $u^{+}=\max\left\{u,0\right\}$. Also, since $\|u_{\epsilon}\|=1$ and  $\|u_{\epsilon}\|^{p}_{L^{p}_{\theta}}=o_{\epsilon}(1)$ we have
\begin{equation}\label{graduec3}
\begin{aligned}
\int_{0}^{\infty}\vert u_{\epsilon,c}^{\prime}\vert^{p}\,\mathrm{d}\lambda_{\alpha}+\int_{0}^{\infty}\vert((u_{\epsilon}-\frac{a_{\epsilon}}{c})^{+})^{\prime}\vert^{p}\,\mathrm{d}\lambda_{\alpha}=\int_{0}^{\infty}\vert u_{\epsilon}^{\prime}\vert^{p}\,\mathrm{d}\lambda_{\alpha}=1+o_{\epsilon}(1).
\end{aligned}
\end{equation}
Combining \eqref{graduec1} \eqref{graduec2} and \eqref{graduec3}, we conclude the proof.
\end{proof}
\begin{lemma}\label{a/d} It holds
\begin{equation}\nonumber
    AD(\eta,\mu_{\alpha,\theta}, \alpha,\theta)=\lim_{\epsilon\rightarrow 0}\frac{d_{\epsilon}}{a^{\frac{p}{p-1}}_{\epsilon}}.
\end{equation}
In particular, $d_{\epsilon}/a^{\sigma}_{\epsilon}\rightarrow\infty$, for any $\sigma<p/(p-1)$. Also, ${a^{p/(p-1)}_{\epsilon}}/{d_{\epsilon}}$ is bounded.
\end{lemma}
\begin{proof}
Let $u_{\epsilon,c}=\min\left\{u_{\epsilon}, a_{\epsilon}/c\right\},$ $c>1$ be given by Lemma~\ref{lemmauec}. Since  $\varphi^{\prime}_{p}(t)\ge\varphi_{p}(t),t\ge 0 $ we obtain
\begin{equation}\nonumber
\begin{aligned}
 AD(\eta,\mu_{\epsilon}, \alpha,\theta)  &\le \int_{0}^{\infty}\varphi_{p}\big(\eta_{\epsilon}\vert u_{\epsilon,c}\vert^{\frac{p}{p-1}}\big)\mathrm{d}\lambda_{\theta}
+\frac{c^{\frac{p}{p-1}}}{a^{\frac{p}{p-1}}_{\epsilon}} \int_{0}^{\infty}\vert u_{\epsilon}\vert^{\frac{p}{p-1}}\varphi_{p}\big(\eta_{\epsilon}\vert u_{\epsilon}\vert^{\frac{p}{p-1}}\big)\mathrm{d}\lambda_{\theta}\\
& \le \int_{0}^{\infty}\varphi_{p}\big(\eta_{\epsilon}\vert u_{\epsilon,c}\vert^{\frac{p}{p-1}}\big)\mathrm{d}\lambda_{\theta}+\frac{c^{\frac{p}{p-1}}d_{\epsilon}}{a^{\frac{p}{p-1}}_{\epsilon}}.
\end{aligned}
\end{equation}
From Lemma~\ref{lemmauec}, $\|u_{\epsilon,c}\|^{p}\rightarrow 1/c<1$ as $\epsilon\rightarrow0$. Hence, using  \eqref{TM1} and arguing as in  \eqref{sub-loc}, \eqref{lq-convergence} we get 
\begin{equation}\label{ucsubc}
\lim_{\epsilon\rightarrow 0}\int_{0}^{\infty}\varphi_{p}\big(\eta_{\epsilon}\vert u_{\epsilon,c}\vert^{\frac{p}{p-1}}\big)\mathrm{d}\lambda_{\theta}=0.
\end{equation}
Letting $\epsilon\rightarrow 0$ and then $c \searrow 1$, we obtain
\begin{equation}\label{MT<frac}
\begin{aligned}
AD(\eta,\mu_{\alpha,\theta}, \alpha,\theta) \le \liminf_{\epsilon\rightarrow 0} AD(\eta,\mu_{\epsilon}, \alpha,\theta)\le  \liminf_{\epsilon\rightarrow 0}\frac{d_{\epsilon}}{a^{\frac{p}{p-1}}_{\epsilon}}.
\end{aligned}
\end{equation}
Noticing that $\varphi^{\prime}_{p}(t)=\varphi_{p}(t)+t^{k_0-1}/((k_0-1)!)$ we can see that
\begin{equation}\nonumber 
\begin{aligned}
d_{\epsilon}\le a^{\frac{p}{p-1}}_{\epsilon}AD(\eta,\mu_{\epsilon}, \alpha,\theta)+\frac{\eta^{k_0-1}_{\epsilon}}{(k_0-1)!}\|u_{\epsilon}\|^{\frac{k_0p}{p-1}}_{L^{\frac{k_0p}{p-1}}_{\theta}}.
\end{aligned}
\end{equation}
Hence, using  that $\|u_{\epsilon}\|_{L^{q}_{\theta}}\rightarrow 0$,  $q\ge p$ and \eqref{MT<frac} and the above inequality we complete the proof. 
\end{proof}

\begin{lemma} \label{lemmaFe}For any $v\in C^{0}_{c}[0,\infty)$, we  have
\begin{equation}\label{v-dirac}
\lim_{\epsilon\rightarrow 0}\int_{0}^{\infty}\frac{a_{\epsilon}b_{\epsilon}}{d_{\epsilon}}\vert u_{\epsilon}\vert^{\frac{1}{p-1}}\varphi^{\prime}_{p}\big(\eta_{\epsilon}\vert u_{\epsilon}\vert^{\frac{p}{p-1}}\big)v\mathrm{d}\lambda_{\theta}=v(0).
\end{equation} 
In particular, the sequence $(F_{\epsilon})$ given by
\begin{equation}\label{Fe}
F_{\epsilon}=\frac{a_{\epsilon}b_{\epsilon}}{d_{\epsilon}}\vert u_{\epsilon}\vert^{\frac{1}{p-1}}\varphi^{\prime}_{p}\big(\eta_{\epsilon}\vert u_{\epsilon}\vert^{\frac{p}{p-1}}\big), \;\; \epsilon>0
\end{equation} 
 satisfies
\begin{equation}\label{1-dirac}
\lim_{\epsilon\rightarrow 0}\int_{0}^{\rho}F_{\epsilon}\mathrm{d}\lambda_{\theta}=1,\quad \rho>0.
\end{equation} 
\end{lemma}
\begin{proof}
For  $R>0$ and $c>1$ we can divide $[0,\infty)=A_1\cup A_2\cup A_3$ into the three disjoints sets
\begin{equation}\nonumber
\begin{aligned}
A_1=\left\{u_{\epsilon}> a_{\epsilon}/c\right\}\cap[r_{\epsilon}R,\infty), \quad A_2= \left\{u_{\epsilon}\le a_{\epsilon}/c\right\}\cap[r_{\epsilon}R,\infty)\;\;\mbox{and}\;\; A_3= (0,r_{\epsilon}R).
\end{aligned}
\end{equation}
We split the integral in \eqref{v-dirac}  into three integrals over these sets and denote by $I^{1}_{\epsilon}$, $I^{2}_{\epsilon}$ and $I^{2}_{\epsilon}$,
respectively. Setting $m=\max_{r\in [0,\infty)}\vert v(r)\vert$ and using \eqref{splitphiprime} we can write
\begin{align*}
I^{1}_{\epsilon} 
&\le  mb_{\epsilon}c\Big[1-\frac{1}{b_{\epsilon}}\int_{0}^{R}\vert v_{\epsilon}\vert^{\frac{p}{p-1}}e^{\eta_{\epsilon}a^{\frac{p}{p-1}}_{\epsilon}(\vert v_{\epsilon}\vert^{\frac{p}{p-1}}-1)}\mathrm{d}\lambda_{\theta}+o_{\epsilon}(R)\Big],
\end{align*}
where we use  $a_{\epsilon}\rightarrow\infty$.  From Lemma~\ref{vto1wtolog}, letting $\epsilon\rightarrow 0$ and then $R\rightarrow\infty$, we obtain $I^{1}_{\epsilon}\rightarrow 0$ as $\epsilon\rightarrow 0$. Also, from Lemma~\ref{a/d} we have $a_{\epsilon}/d_{\epsilon}\rightarrow 0$, then  if $L>0$ is chosen such that $\mathrm{supp}\,v\subset (0,L)$,  we obtain
\begin{align*}
I^{2}_{\epsilon} &\le \frac{a_{\epsilon}b_{\epsilon}m}{d_{\epsilon}}\Big[\int_{A_2}\vert u_{\epsilon}\vert^{\frac{1}{p-1}}\varphi_{p}\big(\eta_{\epsilon}\vert u_{\epsilon}\vert^{\frac{p}{p-1}}\big)\mathrm{d}\lambda_{\theta}+\frac{\eta^{k_0-1}_{\epsilon}}{(k_0-1)!}\int_{0}^{L}\vert u_{\epsilon}\vert^{\frac{(k_0-1)p+1}{p-1}}\mathrm{d}\lambda_{\theta}\Big]\\
& \le \frac{a^{\frac{p}{p-1}}_{\epsilon}}{d_{\epsilon}} \frac{b_{\epsilon}m}{c^{\frac{1}{p-1}}}\int_{0}^{\infty}\varphi_{p}\big(\eta_{\epsilon}\vert u_{\epsilon,c}\vert^{\frac{p}{p-1}}\big)\mathrm{d}\lambda_{\theta}+o_{\epsilon}(1).
\end{align*}
From Lemma~\ref{a/d},  $a^{p/(p-1)}_{\epsilon}/d_{\epsilon}$ is bounded, then using \eqref{ucsubc} we get $I^{2}_{\epsilon}\rightarrow 0$. Finally, for some $\tau\in [0,R]$
\begin{align*}
I^{3}_{\epsilon}
&= \frac{a^{\frac{p}{p-1}}_{\epsilon}b_{\epsilon}r^{\theta+1}_{\epsilon}}{d_{\epsilon}}v(r_{\epsilon}\tau)\int_{0}^{R}\vert v_{\epsilon}\vert^{\frac{1}{p-1}}\varphi^{\prime}_{p}\big(\eta_{\epsilon}a^{\frac{p}{p-1}}_{\epsilon}\vert v_{\epsilon}\vert^{\frac{p}{p-1}}\big)\mathrm{d}\lambda_{\theta}\\
&= v(r_{\epsilon}\tau)\Big[\int_{0}^{R}\vert v_{\epsilon}\vert ^{\frac{1}{p-1}}e^{\eta_{\epsilon}a^{\frac{p}{p-1}}_{\epsilon}(\vert v_{\epsilon}\vert^{\frac{p}{p-1}}-1)}\mathrm{d}\lambda_{\theta}\\
&-e^{-\eta_{\epsilon}a^{\frac{p}{p-1}}_{\epsilon}}\sum_{j=0}^{k_0-2}\frac{\eta^{j}_{\epsilon}a^{\frac{jp}{p-1}}_{\epsilon}}{j!}\int_{0}^{R}\vert v_{\epsilon}\vert^{\frac{pj+1}{p-1}}\mathrm{d}\lambda_{\theta}\Big].
\end{align*}
By \eqref{w-normal} and \eqref{welimit}, letting $\epsilon\rightarrow 0$ and then $R\rightarrow\infty$ we obtain $I^{3}_{\epsilon}\rightarrow v(0)$ which proves \eqref{v-dirac}. To get \eqref{1-dirac}, we choose $v\in C^{0}_{c}[0,\infty)$ such that $v\ge 0$ and $v\equiv 1$ in $[0,\rho)$. For $\epsilon>0$ small enough  we have $r_{\epsilon}R<\rho$ and thus
$(\rho, \infty)\subset A_1\cup A_2.$ So, \eqref{1-dirac}  follows from  $I^{1}_{\epsilon}\rightarrow 0$ and $I^{2}_{\epsilon}\rightarrow 0$ and \eqref{v-dirac}.
\end{proof}

Next, we employ \cite[Lemma~9]{JJ2013} to ensure the existence of the Green-type functions to equation \eqref{Euler-Lagrange}.
\begin{lemma} \label{lemma-greenduro} Let $p\ge 2$, $1<q<p$, $0\le \eta<1$ and $R>0$. Set $g_{\epsilon}=a^{\frac{1}{p-1}}_{\epsilon}u_{\epsilon}$, $\epsilon>0$. Then there is $g_{\eta}$ such that
$(g_{\epsilon})$  converges weakly to $g_{\eta}$  in  $W^{1,q}_{R}(\alpha,\theta)$ and $g_{\eta}$ satisfies the equation
\begin{equation}\label{g-trueeq}
\omega_{\alpha}r^{\alpha}\vert g_{\eta}^{\prime}(r)\vert^{p-1}+\int_{0}^{r}\vert g_{\eta}\vert^{p-1}\mathrm{d}\lambda_{\theta}=1+\eta\int_{0}^{r}\vert g_{\eta}\vert^{p-1}\mathrm{d}\lambda_{\theta},\quad r>0.
\end{equation}
In addition, $g_{\epsilon}\rightarrow g_{\eta}$ in $C^{0}_{loc}(0,\infty)$ and $g^{\prime}_{\epsilon}\rightarrow g^{\prime}_{\eta}$ in $L^{p}_{\alpha}(r_1,\infty)$, for all $r_1>0$.
\end{lemma}
\begin{proof}
By using \eqref{Euler-Lagrange}, for any $v\in X^{1,p}_{\infty}$ we obtain
\begin{equation}\label{EulerG}
\begin{aligned}
 \int_{0}^{\infty}\vert g_{\epsilon}^{\prime}\vert ^{p-2}g_{\epsilon}^{\prime}v^{\prime}\,\mathrm{d}\lambda_{\alpha}
=\int_{0}^{\infty}F_{\epsilon}v\,\mathrm{d}\lambda_{\theta}+(c_{\epsilon}-1)\int_{0}^{\infty}\vert g_{\epsilon}\vert^{p-1}v\,\mathrm{d}\lambda_{\theta},
\end{aligned}
\end{equation}
where $F_{\epsilon}$ is given by \eqref{Fe}. Firstly, for any $R>0$  we claim that 
\begin{equation}\label{gep-1bounded}
\sup_{\epsilon>0}\int_{0}^{R}\vert g_\epsilon\vert^{p-1}\mathrm{d}\lambda_{\theta}\le c.
\end{equation} 
By contradiction,  assume that there exists $R>0$ such that $\lim_{\epsilon\rightarrow 0}\|g_{\epsilon}\|_{L^{p-1}_{\theta}(0,R)}=\infty$. Setting $h_{\epsilon}=g_{\epsilon}/\|g_{\epsilon}\|_{L^{p-1}_{\theta}(0,R)}$ on $(0,R)$,  we have $\|h_{\epsilon}\|_{L^{p-1}_{\theta}(0,R)}=1$ and $h^{\prime}_{\epsilon}=g^{\prime}_{\epsilon}/\|g_{\epsilon}\|_{L^{p-1}_{\theta}(0,R)}$. Also, from \eqref{EulerG}
\begin{equation}\nonumber
\begin{aligned}
 \int_{0}^{\infty}\vert h_{\epsilon}^{\prime}\vert^{p-2}h_{\epsilon}^{\prime}v^{\prime}\,\mathrm{d}\lambda_{\alpha}
=\int_{0}^{\infty}\tilde{F}_{\epsilon}v\,\mathrm{d}\lambda_{\theta}
\end{aligned}
\end{equation}
where
$$
\tilde{F}_{\epsilon}=\frac{F_{\epsilon}}{\|g_{\epsilon}\|^{p-1}_{L^{p-1}_{\theta}(0,R)}}+(c_{\epsilon}-1)\vert h_{\epsilon}\vert^{p-1}.
$$
From \eqref{1-dirac}, we can see that $(\tilde{F}_{\epsilon})$ is bounded in $L^{1}_{\theta}(0,R)$. Let $\tilde{h}_{\epsilon}=h_{\epsilon}-h_{\epsilon}(R)$ on $(0,R)$. Then each $\tilde{h}_{\epsilon}\in X^{1,p}_{R}(\alpha,\theta)$ and satisfies
\begin{equation}\label{euler-fake}
\begin{aligned}
 \int_{0}^{R}\vert \tilde{h}_{\epsilon}^{\prime}\vert^{p-2}\tilde{h}_{\epsilon}^{\prime}v^{\prime}\,\mathrm{d}\lambda_{\alpha}
=\int_{0}^{R}\vert h_{\epsilon}^{\prime}\vert^{p-2}h_{\epsilon}^{\prime}v^{\prime}\,\mathrm{d}\lambda_{\alpha}=\int_{0}^{R}\tilde{F}_{\epsilon}v\,\mathrm{d}\lambda_{\theta},\;\;\forall v\in X^{1,p}_{R}(\alpha,\theta).
\end{aligned}
\end{equation}
Applying \cite[Lemma~9]{JJ2013}, for any $1<q<p$ we get 
\begin{equation}\label{gradh3}
\|h^{\prime}_{\epsilon}\|_{L^{q}_{\alpha}(0,R)}=\|(\tilde{h}_{\epsilon})^{\prime}\|_{L^{q}_{\alpha}(0,R)}\le C(\alpha,q,R,c_0)
\end{equation}
where where $c_0$ is an upper bound of  $(\tilde{F}_{\epsilon})$ in $L^{1}_{\theta}(0,R)$.  Then, if $\overline{q}=p-1$ (cf. Ramark~\ref{remarkW}), we have 
\begin{equation}\label{qbarraW}
\|h_\epsilon\|_{W^{1,\overline{q}}_R}=(\|h^{\prime}_{\epsilon}\|_{L^{\overline{q}}_{\alpha}(0,R)}+\|h_{\epsilon}\|_{L^{\overline{q}}_{\theta}(0,R)})^{1/\overline{q}}\le (C(\alpha,\overline{q},R,c_0)+1)^{1/\overline{q}}.
\end{equation}
Hence, $(h_\epsilon)$ bounded in $W^{1,\overline{q}}_{R}(\alpha,
\theta)$. Also, since $\theta\ge \alpha$ and $\alpha=p-1$ we have $\alpha-\overline{q}+1=1>0$ and $\theta\ge \alpha-\overline{q}=0$. Further, 
$$
\overline{q}^{*}=\overline{q}^{*}(\alpha,\theta,\overline{q})=\frac{(\theta+1)\overline{q}}{\alpha-\overline{q}+1}=(\theta+1)(p-1)\ge p, \;\; \forall\; p\ge 2.
$$
Hence, from \eqref{Ebeddingswhole}, it follows that  the continuous embedding 
\begin{equation}\label{embed-qbarra}
W^{1,\overline{q}}_{R}(\alpha,
\theta)\hookrightarrow L^{q}_{\theta}(0,R),\;\; 1<q\le \overline{q}^{*}.
\end{equation}
Note that $p\le \overline{q}^{*}$ and, from \eqref{qbarraW}, $(h_\epsilon)$ is bounded in $W^{1,\overline{q}}_{R}(\alpha, 
\theta)$. Then \eqref{embed-qbarra} yields 
\begin{equation}\label{h3}
\sup_{\epsilon>0}\|h_{\epsilon}\|_{L^{q}_{\theta}(0,R)}\le C_0,\;\; \forall \, 1<q<p.
\end{equation}
Combining \eqref{gradh3} and \eqref{h3}, we conclude that  $(h_\epsilon)$ bounded in $W^{1,q}_{R}(\alpha,
\theta)$, for any $1<q<p$. Therefore,
\begin{equation}\nonumber 
h_{\epsilon}\rightharpoonup h\;\;\mbox{weakly in}\;\; W^{1,q}_{R}(\alpha,
\theta),\;\; 1<q<p.
\end{equation}
But, we have $\alpha-q+1>\alpha-p+1=0$ and $\theta\ge \alpha>\alpha-q$, for all $1<q<p$. According with \eqref{Ebeddingswhole}  (cf. Remark~\ref{remarkW}) we have the compact embedding
\begin{equation}\label{embed-fake}
W^{1,q}_{R}(\alpha,
\theta)\hookrightarrow L^{s}_{\theta}(0,R),\;\; 1<s<q^{*}(\alpha,\theta,q)=\frac{(\theta+1)q}{\alpha-q+1}.
\end{equation}
Since $\alpha-p+1=0$, we have $q^{*}(\alpha,\theta,q)>p$ if $q$ is  sufficiently near $p$. In particular, using that $(h_\epsilon)$ is bounded in $W^{1,q}_{R}(\alpha, 
\theta)$ and \eqref{embed-fake} we obtain
\begin{equation}\label{embed-fake2}
h_{\epsilon}\rightarrow h\;\;\mbox{in}\;\; L^{s}_{\theta}(0,R),\;\; 1<s<p.
\end{equation}
Combining \eqref{euler-fake}, \eqref{embed-fake} and \eqref{embed-fake2}, we get
\begin{equation}\label{hequation-fake}
\begin{aligned}
 \int_{0}^{R}\vert h^{\prime}\vert^{p-2}h^{\prime}v^{\prime}\,\mathrm{d}\lambda_{\alpha}=(\eta-1)\int_{0}^{R}\vert h\vert ^{p-1}v\,\mathrm{d}\lambda_{\theta},\;\;\forall v\in X^{1,p}_{R}(\alpha,\theta).
\end{aligned}
\end{equation}
Analogous to \eqref{eq.integral}, by using the  test function $v_{\delta}$ given by  \eqref{test-magic}  and noticing that $h$ is a non-increasing function, we can write 
\begin{equation}\label{eq.integral-fake}
    \vert h^{\prime}(r)\vert^{p-1}= (-h^{\prime}(r))^{p-1}=\frac{\eta-1}{\omega_{\alpha}r^{\alpha}}\int_{0}^{r}\vert h\vert^{p-1}\,\mathrm{d}\lambda_{\theta},\;\; r\in (0,R].
 \end{equation}
 Since $\|h\|_{L^{p-1}_{\theta}}=\lim_{\epsilon\rightarrow0}\|h_{\epsilon}\|_{L^{p-1}_{\theta}}=1$, from \eqref{eq.integral-fake}, it follows that
 \begin{equation}\nonumber
    \vert h^{\prime}(R)\vert^{p-1}=\frac{\eta-1}{\omega_{\alpha}R^{\alpha}}<0
 \end{equation}
 which is a contradiction. Hence, \eqref{gep-1bounded} holds.  Thus, from Lemma~\ref{lemmaFe} and \eqref{gep-1bounded}, by setting $\overline{f}_{\epsilon}=F_{\epsilon}+(c_{\epsilon}-1)\vert g_\epsilon\vert^{p-1}$ we get  that the sequence $(\overline{f}_{\epsilon})$ is  bounded in $L^{1}_{\theta}(0,R)$ for all $R>0$.  From \eqref{EulerG} (setting $v\equiv0$ on $[R,\infty)$) we obtain
 \begin{equation}\label{EulerGG}
\begin{aligned}
 \int_{0}^{R}\vert g_{\epsilon}^{\prime}\vert^{p-2}g_{\epsilon}^{\prime}v^{\prime}\,\mathrm{d}\lambda_{\alpha}
=\int_{0}^{R}\overline{f}_{\epsilon}v\,\mathrm{d}\lambda_{\theta},\;\; \forall\; v\in X^{1,p}_{R}(\alpha,\theta).
\end{aligned}
\end{equation}
Using  \cite[Lemma~9]{JJ2013}, we get 
\begin{equation}\label{gradgg3}
\|g^{\prime}_{\epsilon}\|_{L^{q}_{\alpha}(0,R)}\le C(\alpha,q,R,c_0), \;\; 1<q<p,
\end{equation}
where $c_0$ is an upper bound of  $(\overline{f}_{\epsilon})$ in $L^{1}_{\theta}(0,R)$. Now, from \eqref{gep-1bounded}, we can  argue such as  in \eqref{qbarraW}, \eqref{embed-qbarra} and \eqref{h3} to conclude that   $(g_\epsilon)$ is bounded in $W^{1,q}_{R}(\alpha,
\theta)$, $1<q<p$.  Hence, there exists  $g_{\eta}\in W^{1,q}_{R}(\alpha,
\theta)$ such that 
\begin{equation}\label{W-conv-true}
g_{\epsilon}\rightharpoonup g_{\eta}\;\;\mbox{weakly in}\;\; W^{1,q}_{R}(\alpha,
\theta),\;\; 1<q<p.
\end{equation}
Further, from the compact embedding \eqref{embed-fake}, for any $R>0$
\begin{equation}\label{Lpconv-true}
g_{\epsilon}\rightarrow g_{\eta}\;\;\mbox{in}\;\; L^{p-1}_{\theta}(0,R)\;\;\mbox{and}\;\; g_{\epsilon}(r)\rightarrow g_{\eta}(r)\;\;\mbox{a.e in}\;\;(0,R).
\end{equation}
We proceed to show that $g_{\eta}$ satisfies the equation \eqref{g-trueeq}. From \eqref{eq.integral}, we obtain 
\begin{equation}\label{Geq.integral}
     -g_{\epsilon}^{\prime}(r)=\frac{1}{\omega^{\frac{1}{\alpha}}_{\alpha}r}\Big[\int_{0}^{r}(F_{\epsilon}+(c_{\epsilon}-1)\vert g_{\epsilon}\vert^{p-1})\mathrm{d}\lambda_{\theta}\Big]^{\frac{1}{p-1}}, \quad r>0.
 \end{equation}
 Fixed $\rho>0$ arbitrarily,  by integrating \eqref{Geq.integral} on $(r,\rho)$ we obtain 
\begin{equation}\label{explicity-gepsilon}
     g_{\epsilon}(r)=g_{\epsilon}(\rho)+\frac{1}{\omega^{\frac{1}{\alpha}}_{\alpha}}\int_{r}^{\rho}\frac{1}{t}\Big[\int_{0}^{t}(F_{\epsilon}+(c_{\epsilon}-1)\vert g_{\epsilon}\vert^{p-1})\mathrm{d}\lambda_{\theta}\Big]^{\frac{1}{p-1}}\, \mathrm{d}t.
 \end{equation}
 Using Lemma~\ref{lemmaFe} and \eqref{Lpconv-true},  
  \begin{equation}\label{explicity-g}
 \begin{aligned}
g_{\eta}(r)&=g_{\eta}(\rho)+\frac{1}{\omega^{\frac{1}{\alpha}}_{\alpha}}\int_{r}^{\rho}\frac{1}{t}\Big[1+(\eta-1)\int_{0}^{t}\vert g_{\eta}\vert^{p-1}\mathrm{d}\lambda_{\theta}\Big]^{\frac{1}{p-1}}\, \mathrm{d}t.
\end{aligned}
 \end{equation}
By differentiating  this equation we obtain \eqref{g-trueeq}.  Now, let $[r_1, R]\subset (0,\infty)$ be an arbitrary compact interval.  From \eqref{v-dirac} and \eqref{Geq.integral}  we obtain $\vert g_{\epsilon}^{\prime}\vert\le c/r_1$ on $[r_1,R]$, where $c>0$ does not depend on $\epsilon$. In addition, by \eqref{Lpconv-true} we can pick $\rho_0>R$ such that $g_{\epsilon}(\rho_0)\rightarrow g_{\eta}(\rho_0)$ as $\epsilon\rightarrow 0$. This fact combined with \eqref{explicity-gepsilon} and \eqref{v-dirac}  ensure $\vert g_{\epsilon}(r)\vert\le c_1+c_2\ln(\rho_0/r_1)$ on $[r_1,R]$, with $c_1,c_2>0$ do not depend on $\epsilon$. Hence, the Ascoli-Arzelà theorem and  \eqref{Lpconv-true}  imply that $(g_{\epsilon})$ converges to $g_{\eta}$ in $C^{0}[r_1, R]$.  Also,  from  \eqref{v-dirac} and \eqref{Geq.integral}  we have $\vert g_{\epsilon}^{\prime}(r)\vert\le c/r_1$ on $[r_1,\infty)$. Also, by using \eqref{Geq.integral} and \eqref{explicity-g} we conclude that $g^{\prime}_{\epsilon}(r)\rightarrow g^{\prime}_{\eta}(r)$ a.e in $(r_1,\infty)$. Hence, from the Lebesgue dominated convergence theorem we have $g^{\prime}_{\epsilon}\rightarrow g^{\prime}_{\eta}$ in $L^{p}_{\alpha}(r_1,\infty)$.
\end{proof}
\begin{remark}
 From \eqref{g-trueeq}, we can see that $g_{\eta}$ satisfies the equation
 \begin{equation}\label{g-trueeqDelta}
\int_{0}^{\infty}\vert g_{\eta}^{\prime}\vert^{p-2}g_{\eta}^{\prime}v^{\prime}\mathrm{d}\lambda_{\alpha}+\int_{0}^{\infty}\vert g_{\eta}\vert^{p-1}v\mathrm{d}\lambda_{\theta}=\delta_{0}(v)+
\eta\int_{0}^{\infty}\vert g_{\eta}\vert^{p-1}v\mathrm{d}\lambda_{\theta},
\end{equation}
for any  $v\in X^{1,p}_{\infty}\cap C[0,\infty)$,
where $\delta_0$ is the Dirac measure concentrated at origin $r=0$.
\end{remark}
\begin{lemma}\label{g-form}  Let $g_{\eta}$ be given by Lemma~\ref{lemma-greenduro}. Then there exists 
$$
\mathcal{A}_{\eta}=\lim_{r\rightarrow 0}\big[g_{\eta}(r)+\frac{\theta+1}{\mu_{\alpha,\theta}}\ln r\big].
$$ 
Also, for some $z\in C^{1}(0,\infty)$ and $z(r)=O(r^{\theta+1}\vert \ln r\vert^{p-1})$ as $r\rightarrow 0$, $g_{\eta}$ takes the form
\begin{equation}\label{g-shape}
g_{\eta}(r)=-\frac{\theta+1}{\mu_{\alpha,\theta}}\ln r +\mathcal{A}_{\eta} +z(r), \;\; 0<r\le 1.
\end{equation} 
\end{lemma}
\begin{proof}
For $t>0$, we set 
$
h_{g_{\eta}}(t)=(1-\eta)\int_{0}^{t}\vert g_{\eta}\vert^{p-1}\mathrm{d}\lambda_{\theta}.
$
We claim that 
\begin{equation}\label{hgopequeno}
\lim_{t\rightarrow 0}\frac{h_{g_{\eta}}(t)}{t}=0.
\end{equation}
Indeed,  by \eqref{explicity-g} we have $t^{\sigma}g_{\eta}(t)\rightarrow 0$ as $t\rightarrow 0$, for $\sigma>0$. Then, since $g_{\eta}$ belongs to $L^{p-1}_{\theta}(0,R)$, with $R>0$, the L’Hospital rule yields \eqref{hgopequeno}.  Now, we note that 
\begin{align}\label{hg-expoansion}
\left[1-h_{g_{\eta}}(t)\right]^{\frac{1}{p-1}}&=1-\frac{1}{p-1}h_{g_{\eta}}(t)+E_2(t), \quad t>0
\end{align}
where 
\begin{equation}\nonumber
E_2(t)=-\frac{p-2}{2!(p-1)^2}\left(1-\tau_0 h_{g_{\eta}}(t)\right)^{-\frac{2p-3}{p-1}}h^2_{g_{\eta}}(t),
\end{equation}
for some $\tau_0=\tau_0(t)\in (0,1)$. In view of \eqref{hgopequeno} we also have  
\begin{equation}\label{E2opequeno}
\lim_{t\rightarrow 0}\frac{E_2(t)}{t^2}=0.
\end{equation}
From  \eqref{explicity-g} and \eqref{hg-expoansion}, for any  $0<r\le 1$
 \begin{equation}\label{ghg}
 \begin{aligned}
g_{\eta}(r)&=g_{\eta}(1)-\frac{\theta+1}{\mu_{\alpha,\theta}}\ln r-\frac{\theta+1}{\mu_{\alpha,\theta}}\frac{1}{p-1}\int_{r}^{1}\frac{h_{g_{\eta}}(t)}{t}dt+\frac{\theta+1}{\mu_{\alpha,\theta}}\int_{r}^{1}\frac{E_2(t)}{t} dt.
\end{aligned}
 \end{equation}
Hence, from \eqref{hgopequeno} and \eqref{E2opequeno},  there exists
\begin{align}\nonumber
\mathcal{A}_{\eta}&=\lim_{r\rightarrow 0}\big[g_{\eta}(r)+\frac{\theta+1}{\mu_{\alpha,\theta}}\ln r\big]\\
&=g_{\eta}(1)-\frac{\theta+1}{\mu_{\alpha,\theta}}\frac{1}{p-1}\int_{0}^{1}\frac{h_{g_{\eta}}(t)}{t}dt+\frac{\theta+1}{\mu_{\alpha,\theta}}\int_{0}^{1}\frac{E_2(t)}{t} dt.
\end{align}
From \eqref{ghg} we also can write 
\begin{align}\label{gform-quasi}
g_{\eta}(r)&=-\frac{\theta+1}{\mu_{\alpha,\theta}} \ln r+\mathcal{A}_{\eta}+z(r)
\end{align}
with 
\begin{align}\nonumber
z(r)=\frac{\theta+1}{\mu_{\alpha,\theta}}\frac{1}{p-1}\int_{0}^{r}\frac{h_{g_{\eta}}(t)}{t}dt-\frac{\theta+1}{\mu_{\alpha,\theta}}\int_{0}^{r}\frac{E_2(t)}{t} dt.
\end{align}
We observe that $z(0)=0$.  In addition,  the L’Hospital rule  and \eqref{gform-quasi}  yield
\begin{equation}\label{o-crucial}
\begin{aligned}
\lim_{r\rightarrow 0}\frac{h_{g_{\eta}}(r)}{r^{\theta+1}\vert\ln r\vert^{p-1}}
&=\frac{\omega_{\theta}}{\omega_{\alpha}}\frac{1-\eta}{\theta+1}.\\
\end{aligned}
\end{equation}
Thus, we obtain 
\begin{equation}\nonumber
\begin{aligned}
\lim_{r\rightarrow 0}\frac{\int_{0}^{r}\frac{h_{g_{\eta}}(t)}{t}dt}{r^{\theta+1}\vert\ln r\vert^{p-1}}= \frac{\omega_{\theta}}{\omega_{\alpha}}\frac{1-\eta}{(\theta+1)^2}.
\end{aligned}
\end{equation}
Since $E_2(t)=O(h^2_{g_{\eta}}(t))$ as $t\rightarrow 0$, from the above argument we can also write 
 \begin{equation}\nonumber
\begin{aligned}
\lim_{r\rightarrow 0}\frac{\int_{0}^{r}\frac{E_2(t)}{t}dt}{r^{\theta+1}\vert\ln r\vert^{p-1}}=0.
\end{aligned}
\end{equation}
Then we get 
\begin{equation}\nonumber
z(r)=O(r^{\theta+1}\vert\ln r\vert^{p-1}) +o(r^{\theta+1}\vert\ln r\vert^{p-1}) ,\quad \mbox{as}\quad r\rightarrow 0.
\end{equation}
\end{proof}
\begin{lemma} \label{abaixo} For $\mathcal{A}_{\eta}$ is given by Lemma~\ref{g-form}, we have
\begin{equation}\nonumber
  AD(\eta, \mu_{\alpha,\theta}, \alpha,\theta)\le \frac{ \omega_{\theta}}{\theta+1}e^{\mu_{\alpha,\theta}\mathcal{A}_{\eta}+\gamma+\Psi(p)},
  \end{equation}
  where $\Psi(x)=\frac{d}{dx}(\ln\Gamma(x))$ is the digamma function and $\gamma=-\Psi(1)$ is the Euler-Mascheroni constant.
\end{lemma}
\begin{proof}
Fix $\rho>0$. From  the same argument in \eqref{Fe-later},  for $r\ge \rho$ we can write
\begin{equation}\nonumber
   \begin{aligned}
    \vert u_{\epsilon}(r)\vert^{\frac{p}{p-1}}\varphi_{p}^{\prime}\big(\eta_{\epsilon}\vert u_{\epsilon}(r)\vert^{\frac{p}{p-1}}\big)
   & \le C(\alpha, \eta, \theta)e^{\mu_{\alpha,\theta}C2^{\frac{1}{p-1}}\rho^{-\frac{\theta+1}{p}}}\vert u_{\epsilon}(r)\vert^{\frac{k_0p}{p-1}}.
   \end{aligned} 
\end{equation}
Using Lemma~\ref{a/d}, the convergence in \eqref{eploc-convergence} and  the Fatou's lemma we get
\begin{equation}\nonumber
\int_{\rho}^{\infty}g_{\epsilon}F_{\epsilon}\mathrm{d}\lambda_{\theta}=\frac{a^{\frac{p}{p-1}}_{\epsilon}b_{\epsilon}}{d_{\epsilon}}\int_{\rho}^{\infty}\vert u_{\epsilon}\vert^{\frac{p}{p-1}}\varphi^{\prime}_{p}\big(\eta_{\epsilon}\vert u_{\epsilon}\vert^{\frac{p}{p-1}}\big)\mathrm{d}\lambda_{\theta}=o_{\epsilon}(\rho)
\end{equation}
where $o_{\epsilon}(\rho)\rightarrow 0$ as $\epsilon\rightarrow 0$. In addition, \eqref{Geq.integral} yields
\begin{equation}\nonumber
\begin{aligned}
\int_{\rho}^{\infty}\vert g^{\prime}_{\epsilon}\vert^{p}\mathrm{d}\lambda_{\alpha}&=-\int_{\rho}^{\infty}g^{\prime}_{\epsilon}(s) \Big[\int_{0}^{s}(F_{\epsilon}+(c_{\epsilon}-1)\vert g_{\epsilon}\vert^{p-1})\mathrm{d}\lambda_{\theta}\Big]\mathrm{d}s\\
&=\omega_{\alpha}\rho^{\alpha}\vert g^{\prime}_{\epsilon}(\rho)\vert^{p-1}g_{\epsilon}(\rho)+
\int_{\rho}^{\infty}\left(g_{\epsilon}F_{\epsilon}+(c_{\epsilon}-1)\vert g_{\epsilon}\vert^{p}\right)\mathrm{d}\lambda_{\theta}.\\
\end{aligned}
\end{equation}
Thus, 
\begin{equation}\label{out-gg}
\begin{aligned} \int_{\rho}^{\infty}\vert g^{\prime}_{\epsilon}\vert^{p}\mathrm{d}\lambda_{\alpha}+(1-c_{\epsilon})\int_{\rho}^{\infty}\vert g_{\epsilon}\vert^{p}\mathrm{d}\lambda_{\theta}
=\omega_{\alpha}\rho^{\alpha}\vert g^{\prime}_{\epsilon}(\rho)\vert^{p-1}g_{\epsilon}(\rho)+ o_{\epsilon}(\rho).
\end{aligned}
\end{equation}
Combining \eqref{1-dirac},  \eqref{Lpconv-true} and \eqref{Geq.integral}, we can see that $\omega_{\alpha}\rho^{\alpha}\vert g_{\epsilon}^{\prime}(\rho)\vert^{p-1}\le c$, for some $c>0$ which does not depend on $\epsilon$. Since $c_{\epsilon}\rightarrow \eta$, from \eqref{out-gg}
\begin{equation}\nonumber
\begin{aligned} \int_{\rho}^{\infty}\vert g_{\epsilon} \vert^{p}\mathrm{d}\lambda_{\theta}
\le c_{1}g_{\epsilon}(\rho)+ o_{\epsilon}(\rho).
\end{aligned}
\end{equation}
Lemma~\ref{lemma-greenduro} together with  Fatou's lemma yields
\begin{equation}\label{G-fora}
\begin{aligned}
\int_{\rho}^{\infty}\vert g_{\eta} \vert^{p}\mathrm{d}\lambda_{\theta}\le  \lim_{\epsilon\rightarrow 0}\int_{\rho}^{\infty}\vert g_{\epsilon} \vert^{p}\mathrm{d}\lambda_{\theta}
\le c_{1}g_{\eta}(\rho).
\end{aligned}
\end{equation}
In particular,  $g_{\eta}\in L^{p}_{\theta}(\rho, \infty)$, for $\rho>0$ and we obtain $g_{\eta}(r)\rightarrow 0$ as $r\rightarrow\infty$. Thus, from \eqref{G-fora} we obtain 
\begin{equation}\nonumber
\lim_{\rho\rightarrow\infty}\lim_{\epsilon\rightarrow 0}\int_{\rho}^{\infty}\vert g_{\epsilon} \vert^{p}\mathrm{d}\lambda_{\theta}=\lim_{\rho\rightarrow\infty} \int_{\rho}^{\infty}\vert g_{\eta} \vert^{p}\mathrm{d}\lambda_{\theta}=0.
\end{equation}
Using Lemma~\ref{lemma-greenduro} we conclude that  $g_{\epsilon}\rightarrow g_{\eta}$ in $L^{p}_{\theta}(0,\infty)$. Thus, from \eqref{g-trueeq} we have 
 \begin{equation}\nonumber
\omega_{\alpha}\rho^{\alpha}\vert g_{\eta}^{\prime}(\rho)\vert^{p-1}g_{\eta}(\rho)=g_{\eta}(\rho)+(\eta-1)g_{\eta}(\rho)\int_{0}^{\rho}\vert g_{\eta}\vert^{p-1}\mathrm{d}\lambda_{\theta}=g_{\eta}(\rho)+o_{\rho}(1), \;\mbox{as}\; \rho\rightarrow 0
\end{equation}
where we have used $0\le g_{\eta}(\rho)\int_{0}^{\rho}\vert g_{\eta}\vert^{p-1}\mathrm{d}\lambda_{\theta}\le \int_{0}^{\rho}\vert g_{\eta} \vert^{p}\mathrm{d}\lambda_{\theta}\rightarrow 0$ as $\rho\rightarrow 0$.
Hence, from \eqref{out-gg}  we have
\begin{equation}\nonumber
\begin{aligned}
 \int_{\rho}^{\infty}\vert u^{\prime}_{\epsilon}\vert^{p}\mathrm{d}\lambda_{\alpha}+\int_{\rho}^{\infty}\vert u_{\epsilon}\vert^{p}\mathrm{d}\lambda_{\theta}&=\frac{1}{a^{\frac{p}{p-1}}_{\epsilon}}\big[g_{\eta}(\rho)+\eta\int_{\rho}^{\infty}\vert g_{\eta} \vert^{p}\mathrm{d}\lambda_{\theta}+o_{\rho}(1)+o_{\epsilon}(\rho)\big]\\
 &=\frac{1}{a^{\frac{p}{p-1}}_{\epsilon}}\big[g_{\eta}(\rho)+\eta\|g_{\eta}\|^{p}_{L^{p}_{\theta}}+o_{\epsilon}(\rho)+o_{\rho}(1)\big].
 \end{aligned}
\end{equation}
We also have 
\begin{equation}
\int_{0}^{\rho}\vert u_{\epsilon}\vert^{p}\mathrm{d}\lambda_{\theta}=\frac{1}{a^{\frac{p}{p-1}}_{\epsilon}}\Big(\int_{0}^{\rho}\vert g_{\eta} \vert^{p}\mathrm{d}\lambda_{\theta}+o_{\epsilon}(\rho)\Big)=\frac{1}{a^{\frac{p}{p-1}}_{\epsilon}}\left(o_{\rho}(1)+o_{\epsilon}(\rho)\right).
\end{equation}
Since $\|u_{\epsilon}\|=1$ the two previous estimates yield
\begin{equation}\nonumber
\begin{aligned}
 \int_{0}^{\rho}\vert u^{\prime}_{\epsilon}\vert^{p}\mathrm{d}\lambda_{\alpha}
 &=1-\frac{1}{a^{\frac{p}{p-1}}_{\epsilon}}\big[g_{\eta}(\rho)+\eta\|g_{\eta}\|^{p}_{L^{p}_{\theta}}+o_{\epsilon}(\rho)+o_{\rho}(1)\big].
 \end{aligned}
\end{equation}
Set $\tau_{\epsilon,\rho}=\|u^{\prime}_{\epsilon}\|^{p}_{L^{p}_{\alpha}(0,\rho)}$. Define  $v_{\epsilon,\rho}=u_{\epsilon}-u_{\epsilon}(\rho)$ on $(0,\rho]$. We have $v_{\epsilon,\rho}\in X^{1,p}_{\rho}$. Also, by using Lemma~\ref{we-concentra} we have that $(v_{\epsilon,\rho}/\tau^{1/p}_{\epsilon,\rho})$ is concentrating at the origin. Hence, by \cite[Lemma~A]{JJ2013} (see also  \cite{CC}) we obtain
\begin{equation}\label{CCJJ}
\limsup_{\epsilon\rightarrow 0}\int_{0}^{\rho}\big(e^{\mu_{\alpha,\theta}\vert v_{\epsilon,\rho}/\tau^{1/p}_{\epsilon,\rho}\vert^{\frac{p}{p-1}}}-1\big)\mathrm{d}\lambda_{\theta}\le \frac{\omega_{\theta}\rho^{\theta+1}}{\theta+1}e^{\gamma+\Psi(p)}.
\end{equation}
Fix $R>0$. We have $0<r_{\epsilon}R\le \rho$, for $\epsilon>0$ small enough.  From Lemma~\ref{vto1wtolog} we can write $u_{\epsilon}=a_{\epsilon}(1+o_{\epsilon}(R))$ uniformly on $[0,r_{\epsilon}R)$. Thus, 
\begin{equation}\nonumber
\begin{aligned}
\eta_{\epsilon}\vert u_{\epsilon}\vert^{\frac{p}{p-1}}
&\le \mu_{\alpha,\theta}\vert u_{\epsilon}\vert^{\frac{p}{p-1}}+\frac{\mu_{\alpha,\theta}\eta}{p-1}\|g_{\eta}\|^{p}_{L^{p}_{\theta}}+o_{\epsilon}(R).\\
\end{aligned}
\end{equation}
Analogously
\begin{equation}\nonumber
\begin{aligned}
\vert u_{\epsilon}\vert^{\frac{p}{p-1}}
&=(v_{\epsilon,\rho}+u_{\epsilon}(\rho))^{\frac{p}{p-1}}=\vert v_{\epsilon,\rho}\vert^{\frac{p}{p-1}}+\frac{p}{p-1}\vert v_{\epsilon,\rho}\vert^{\frac{1}{p-1}}u_{\epsilon}(\rho)+o_{\epsilon}(\rho).
\end{aligned}
\end{equation}
So, it follows that 
\begin{equation}\nonumber
\begin{aligned}
\eta_{\epsilon}\vert u_{\epsilon}\vert^{\frac{p}{p-1}}
&\le \mu_{\alpha,\theta}\vert v_{\epsilon,\rho}\vert^{\frac{p}{p-1}}+\frac{p\mu_{\alpha,\theta}}{p-1}\vert v_{\epsilon,\rho}\vert^{\frac{1}{p-1}}u_{\epsilon}(\rho)+\frac{\mu_{\alpha,\theta}\eta}{p-1}\|g_{\eta}\|^{p}_{L^{p}_{\theta}}+o_{\epsilon}(R)+o_{\epsilon}(\rho),
\end{aligned}
\end{equation}
on $[0,r_{\epsilon}R).$ From Lemma~\ref{lemma-greenduro},
$
\vert v_{\epsilon,\rho}\vert^{\frac{1}{p-1}}u_{\epsilon}(\rho)=g_{\eta}(\rho)+o_{\epsilon}(R)
$
and by definition of $\tau_{\epsilon,\rho}$ we can see that
\begin{equation}\nonumber
\tau^{\frac{1}{p-1}}_{\epsilon,\rho}=1-\frac{1}{p-1}\frac{1}{a^{\frac{p}{p-1}}_{\epsilon}}\left(g_{\eta}(\rho)+\eta\|g_{\eta}\|^{p}_{L^{p}_{\theta}}+o_{\epsilon}(\rho)+o_{\rho}(1)\right).
\end{equation}
Hence, by using $v_{\epsilon,\rho}=a_{\epsilon}[1+o_{\epsilon}(R)+o_{\epsilon}(\rho)] $ uniformly on $[0,r_{\epsilon}R)$ and $\tau_{\epsilon,\rho}=1+o_{\epsilon}(\rho)$, we can write
\begin{equation}\nonumber
\begin{aligned}
\eta_{\epsilon}\vert u_{\epsilon}\vert^{\frac{p}{p-1}}
&\le \mu_{\alpha,\theta}\vert v_{\epsilon,\rho}/\tau^{1/p}_{\epsilon,\rho}\vert^{\frac{p}{p-1}}+\mu_{\alpha,\theta}g_{\eta}(\rho)+o_{\epsilon}(\rho)+o_{\epsilon}(R)+o_{\rho}(1).\\
\end{aligned}
\end{equation}
Since $0<r_{\epsilon}R\le \rho$ and $\varphi_{p}(t)\le e^{t}$, for $t\ge 0$ 
\begin{equation}\nonumber
\begin{aligned}
\int_{0}^{r_{\epsilon}R}\varphi_{p}\big(\eta_{\epsilon}\vert u_{\epsilon}\vert^{\frac{p}{p-1}}\big)\mathrm{d}\lambda_{\theta}
&\le  \Big[e^{\mu_{\alpha,\theta}g_{\eta}(\rho)+o_{\epsilon}(\rho)+o_{\epsilon}(R)+o_{\rho}(1)}\Big]\times\\
&\Big[\int_{0}^{\rho}\big(e^{\mu_{\alpha,\theta}\vert v_{\epsilon,\rho}/\tau^{1/p}_{\epsilon,\rho}\vert^{\frac{p}{p-1}}}-1\big)\mathrm{d}\lambda_{\theta}\Big]+o_{\epsilon}(\rho).
\end{aligned}
\end{equation}
Thus, from \eqref{CCJJ} we obtain
\begin{equation}\nonumber
\begin{aligned}
\limsup_{\epsilon\rightarrow 0}\int_{0}^{r_{\epsilon}R}\varphi_{p}\big(\eta_{\epsilon}\vert u_{\epsilon}\vert^{\frac{p}{p-1}}\big)\mathrm{d}\lambda_{\theta}\le\frac{\omega_{\theta}\rho^{\theta+1}}{\theta+1} e^{\mu_{\alpha,\theta}g_{\eta}(\rho)+o_{\rho}(1)}e^{\gamma+\Psi(p)}.
\end{aligned}
\end{equation}
Recalling
\eqref{qovere} we can see that 
\begin{equation}\nonumber
\begin{aligned}
\int_{0}^{r_{\epsilon}R}\varphi_{p}\big(\eta_{\epsilon}\vert u_{\epsilon}\vert^{\frac{p}{p-1}}\big)\mathrm{d}\lambda_{\theta} 
&=(1+o_{\epsilon}(R))\frac{d_{\epsilon}}{a^{\frac{p}{p-1}}_{\epsilon}}\int_{0}^{R}e^{(1+o_{\epsilon}(R))\mu_{\alpha,\theta}\frac{p}{p-1}w}\mathrm{d}\lambda_{\theta}.
 \end{aligned}
\end{equation}
Letting $\epsilon\rightarrow 0$ and $R\rightarrow\infty$ and using Lemma~\ref{a/d} we obtain
\begin{equation}\nonumber
\begin{aligned}
\lim_{R\rightarrow \infty}\lim_{\epsilon\rightarrow 0}\int_{0}^{r_{\epsilon}R}\varphi_{p}\big(\eta_{\epsilon}\vert u_{\epsilon}\vert^{\frac{p}{p-1}}\big)\mathrm{d}\lambda_{\theta} = AD(\eta,\mu_{\alpha,\theta}, \alpha,\theta).
 \end{aligned}
\end{equation}
Then 
\begin{equation}\nonumber
\begin{aligned}
AD(\eta,\mu_{\alpha,\theta}, \alpha,\theta)\le\frac{\omega_{\theta}}{\theta+1} e^{\mu_{\alpha,\theta}g_{\eta}(\rho)+(\theta+1)\ln\rho+o_{\rho}(1)}e^{\gamma+\Psi(p)}.
\end{aligned}
\end{equation}
Letting $\rho\rightarrow 0$ and using Lemma~\ref{g-form}  we conclude the proof.
\end{proof}
\begin{lemma}\label{lema-tecnico} For any $z>0$ and $p\ge2$ we have 
\begin{equation}\label{LT-eq1}
\int_{0}^{z}\frac{s^{p-1}}{(1+s)^{p}}\mathrm{d}s=\ln(1+z)-[\gamma+\Psi(p)]+\int_{\frac{z}{1+z}}^{1}\frac{1-s^{p-1}}{1-s}\mathrm{d}s
\end{equation}
and
\begin{equation}\label{LT-eq2}
\int_{0}^{z}\frac{s^{p-2}}{(1+s)^{p}}\mathrm{d}s=p-1-\int_{z}^{\infty}\frac{1-s^{p-2}}{1-s}\mathrm{d}s.
\end{equation}
In addition, we have
\begin{equation}\label{LT-eq3}
\frac{\Gamma(p)\Gamma(1+x)}{\Gamma(p+x)}=1-[\Psi(p)+\gamma]x+O(x^2),\quad \mbox{as}\quad x\rightarrow 0.
\end{equation}
\end{lemma}
\begin{proof}
We recall (see \cite{special} for instance) the following 
\begin{equation}\label{GPsi}
\Gamma(1)=1, \;\;\Gamma(x+1)=x\Gamma(x)\;\; \mbox{and}\;\;  
\int_{0}^{\infty}\frac{s^{x-1}}{(1+s)^{x+y}}ds=\frac{\Gamma(x)\Gamma(y)}{\Gamma(x+y)},\;\; x, y>0
\end{equation}
and the integral representation of $\Psi(x)$ due to Dirichlet
\begin{equation}\label{DPsi}
\Psi(x)=\int_{0}^{\infty}\frac{1}{z}\big(1-\frac{1}{(1+z)^x}\big)\mathrm{d}z,\quad x>0.
\end{equation}
Since $\Psi(1)=-\gamma$, by using  \eqref{DPsi} and the change of variables $s=1/(1+z)$ it follows that
\begin{equation}\nonumber
\begin{aligned}
\Psi(p)+\gamma&=\int_{0}^{\infty}\frac{1}{z}\Big[\big(1-\frac{1}{(1+z)^p}\big)-\big(1-\frac{1}{1+z}\big)\Big]\mathrm{d}z\\
&=\int_{0}^{1}\frac{1-s^{p-1}}{1-s}\mathrm{d}s.
\end{aligned}
\end{equation}
Thus, setting $1/\tau =1+s$  we can write
\begin{equation}\nonumber
\begin{aligned}
\int_{0}^{z}\frac{s^{p-1}}{(1+s)^{p}}\mathrm{d}s
&= \int_{\frac{1}{1+z}}^{1}\Big[\frac{1}{\tau}+\frac{1}{\tau}((1-\tau)^{p-1}-1)\Big]\mathrm{d}\tau\\
&=\ln(1+z)-[\Psi(p)+\gamma]+\int_{\frac{z}{1+z}}^{1}\frac{1-s^{p-1}}{1-s}\mathrm{d}s
\end{aligned}
\end{equation}
which proves \eqref{LT-eq1}. Analogously, with help of \eqref{GPsi} we obtain \eqref{LT-eq2}. Finally, the Taylor's expansion at $x=0$  of the function 
$x\mapsto {\Gamma(p)\Gamma(1+x)}/{\Gamma(p+x)}$ yields \eqref{LT-eq3}.
\end{proof}
\begin{lemma} \label{acima} There exists $\eta_{0}\in (0, 1)$ such that 
\begin{equation}\nonumber
\begin{aligned}
AD(\eta,\mu_{\alpha,\theta}, \alpha,\theta)>\frac{\omega_{\theta}}{\theta+1} e^{\mu_{\alpha,\theta}\mathcal{A}_{\eta}+\gamma+\Psi(p)},
\end{aligned}
\end{equation}
for any $0<\eta<\eta_0$.
\end{lemma}
\begin{proof}
For any $\epsilon>0$ small enough, let us define
\begin{equation}\label{v-testfunction}
v_{\epsilon}(r)=\left\{
\begin{aligned}
& c+\frac{1}{c^{\frac{1}{p-1}}}\Big[-\frac{p-1}{\mu_{\alpha,\theta}}\ln\Big(1+c_{\alpha,\theta}\Big(\frac{r}{\epsilon}\Big)^{\frac{\theta+1}{p-1}}\Big)+b\Big],& \mathrm{if}\;\; r\le \epsilon L\\
&\frac{1}{c^{\frac{1}{p-1}}}g_{\eta}(r), &  \mathrm{if}\;\; r>\epsilon L
\end{aligned}
\right.
\end{equation}
where $g_{\eta}$ is given by Lemma~\ref{lemma-greenduro}, $c_{\alpha,\theta}=(\omega_{\theta}/(\theta+1))^{1/(p-1)}$,  $L=-\ln\epsilon$ and $b$, $c$ are constants determined later.  To ensure $v_{\epsilon}\in X^{1,p}_{\infty}$, we assume the identity
\begin{equation}\label{lotus primaria}
c+\frac{1}{c^{\frac{1}{p-1}}}\Big[-\frac{p-1}{\mu_{\alpha,\theta}}\ln\Big(1+c_{\alpha,\theta}L^{\frac{\theta+1}{p-1}}\Big)+b\Big]=\frac{1}{c^{\frac{1}{p-1}}}g_{\eta}(\epsilon L). 
\end{equation}
From \eqref{g-shape} and the choice $L=-\ln \epsilon$, we can  write 
\begin{equation}\label{c-choice}
\begin{aligned}
c^{\frac{p}{p-1}} &=\frac{1}{\mu_{\alpha,\theta}}\ln \frac{\omega_{\theta}}{\theta+1}-\frac{\theta+1}{\mu_{\alpha,\theta}}\ln\epsilon+\mathcal{A}_{\eta}-b+O\Big(L^{-\frac{\theta+1}{p-1}}\Big).
\end{aligned}
\end{equation}
By simplicity, let us define 
$
\mathcal{C}^{q}_{g}(r)=\int_{0}^{r}\vert g_{\eta}\vert^{q}\mathrm{d}\lambda_{\theta}$,\; $q\ge 1,  r>0.$ In the same way of \eqref{o-crucial}, we have 
\begin{equation}\label{ordem-capaz}
\lim_{r\rightarrow 0}\frac{\mathcal{C}^{q}_{g}(r)}{r^{\theta+1}\vert \ln r\vert^{q}}=\frac{\omega_{\theta}}{\theta+1}\frac{1}{\omega^{\frac{q}{\alpha}}_{\alpha}}.
\end{equation}
Using \eqref{g-trueeq}, we have
\begin{equation}\nonumber
\begin{aligned}
\int_{\epsilon L}^{\infty}\vert g_{\eta}^{\prime}\vert^{p}\mathrm{d}\lambda_{\alpha}
&=\omega_{\alpha}(\epsilon L)^{\alpha}\vert g_{\eta}^{\prime}(\epsilon L)\vert^{p-1}g_{\eta}(\epsilon L)+(\eta-1)\int_{\epsilon L}^{\infty}\vert g_{\eta}\vert^{p}\mathrm{d}\lambda_{\theta}.
\end{aligned}
\end{equation}
From this and by using \eqref{g-trueeq} again we can write 
\begin{equation}\nonumber
\begin{aligned}
\int_{\epsilon L}^{\infty}\vert v^{\prime}_{\epsilon}\vert^{p}\mathrm{d}\lambda_{\alpha} &+\int_{\epsilon L}^{\infty}\vert v_{\epsilon}\vert^{p}\mathrm{d}\lambda_{\theta}\\
&=\frac{1}{c^{\frac{p}{p-1}}}\big[\eta\|g_{\eta}\|^{p}_{L^{p}_{\theta}}+g_{\eta}(\epsilon L)+(\eta-1)g_{\eta}(\epsilon L)\mathcal{C}^{p-1}_{g}(\epsilon L) -\eta \mathcal{C}^{p}_{g}(\epsilon L) \big].
\end{aligned}
\end{equation}
Therefore, from \eqref{ordem-capaz}  and Lemma~\ref{g-form}
\begin{equation}\label{test-fora} 
\begin{aligned}
&\int_{\epsilon L}^{\infty}\vert v^{\prime}_{\epsilon}\vert^{p}\mathrm{d}\lambda_{\alpha}+\int_{\epsilon L}^{\infty}\vert v_{\epsilon}\vert^{p}\mathrm{d}\lambda_{\theta}\\
&\quad=\frac{1}{c^{\frac{p}{p-1}}}\Big[\eta\|g_{\eta}\|^{p}_{L^{p}_{\theta}}-\frac{\theta+1}{\mu_{\alpha,\theta}}\ln\epsilon-\frac{\theta+1}{\mu_{\alpha,\theta}}\ln L+\mathcal{A}_{\eta}+O((\epsilon L)^{\theta+1}\vert \ln(\epsilon L)\vert^{p})\Big].
\end{aligned}
\end{equation}
By simplicity, we denote $z_{L}=c_{\alpha,\theta}L^{\frac{\theta+1}{p-1}}$. Hence,  from Lemma~\ref{lema-tecnico}
\begin{equation}\nonumber
\begin{aligned}
\int_{0}^{\epsilon L}\vert v^{\prime}_{\epsilon}\vert^{p}\mathrm{d}\lambda_{\alpha}
&=\frac{p-1}{c^{\frac{p}{p-1}}\mu_{\alpha,\theta}}\Big[\ln\left(1+z_{L}\right)-[\gamma+\Psi(p)]+\int_{\frac{z_{L}}{1+z_{L}}}^{1}\frac{1-s^{p-1}}{1-s}\mathrm{d}s\Big].
\end{aligned}
\end{equation}
Noticing that  
\begin{equation}\nonumber
\lim_{z\rightarrow\infty}(1+z)\int_{\frac{z}{1+z}}^{1}\frac{1-s^{p-1}}{1-s}\mathrm{d}s=p-1 \quad \mbox{(by L’Hospital rule)}
\end{equation}
and 
\begin{equation}\nonumber
\frac{p-1}{\mu_{\alpha,\theta}}\ln\left(1+z_{L}\right)=\frac{\theta+1}{\mu_{\alpha,\theta}}\ln L+\frac{1}{\mu_{\alpha,\theta}}\ln \frac{\omega_{\theta}}{\theta+1}+O\Big(L^{-\frac{\theta+1}{p-1}}\Big)
\end{equation}
we can see that
\begin{equation}\label{gradtest-dentro}
\begin{aligned}
& \int_{0}^{\epsilon L}\vert v^{\prime}_{\epsilon}\vert^{p}\mathrm{d}\lambda_{\alpha} =\frac{1}{c^{\frac{p}{p-1}}}\Big[\frac{\theta+1}{\mu_{\alpha,\theta}}\ln L+\frac{1}{\mu_{\alpha,\theta}}\ln \frac{\omega_{\theta}}{\theta+1}-\frac{p-1}{\mu_{\alpha,\theta}}[\gamma+\Psi(p)]+O\Big(L^{-\frac{\theta+1}{p-1}}\Big)\Big].
\end{aligned}
\end{equation}
On the other hand, by using $L=-\ln \epsilon$, \eqref{g-shape} and  \eqref{lotus primaria} we obtain for any $q\ge 1$
 \begin{equation}\label{test-dentro}
\begin{aligned}
\int_{0}^{\epsilon L}\vert v_{\epsilon}\vert^{q}\mathrm{d}\lambda_{\theta}&=\frac{1}{c^{\frac{q}{p-1}}}\int_{0}^{\epsilon L}\Big\vert g_{\eta}(\epsilon L)+\frac{p-1}{\mu_{\alpha,\theta}}\ln\bigg(\frac{1+c_{\alpha,
\theta}L^{\frac{\theta+1}{p-1}}}{1+c_{\alpha,
\theta}\big(\frac{r}{\epsilon}\big)^{\frac{\theta+1}{p-1}}}\bigg)\Big\vert^{q} \mathrm{d}\lambda_{\theta}\\
&=\frac{1}{c^{\frac{q}{p-1}}}\Big[O\big((\epsilon L)^{\theta+1}\vert \ln(\epsilon L)\vert^{q}\big)\Big].
\end{aligned}
\end{equation}
Also, for any $q\ge p$, directly from \eqref{v-testfunction} and \eqref{ordem-capaz} we have
\begin{equation}\label{test-foraLq}
\begin{aligned}
\int_{\epsilon L}^{\infty}\vert v_{\epsilon}\vert^{q}\mathrm{d}\lambda_{\theta}&=\frac{1}{c^{\frac{q}{p-1}}}\Big[\|g_{\eta}\|^{q}_{L^{q}_{\theta}}+O\left((\epsilon L)^{\theta+1}\vert\ln(\epsilon L)\vert^{q}\right)\Big].
\end{aligned}
\end{equation}
Combining  \eqref{test-dentro} and \eqref{test-foraLq} we have 
\begin{equation}\label{lp-test}
\begin{aligned}
\|v_{\epsilon}\|^{q}_{L^{q}_{\theta}}=\frac{1}{c^{\frac{q}{p-1}}}\left[\|g_{\eta}\|^{q}_{L^{q}_{\theta}}+O\left((\epsilon L)^{\theta+1}\vert \ln(\epsilon L)\vert^{q}\right)\right], \quad q\ge p.
\end{aligned}
\end{equation}
In addition,  from \eqref{test-fora}, \eqref{gradtest-dentro} and \eqref{test-dentro}
\begin{equation}\nonumber
\begin{aligned}
 \|v_{\epsilon}\|^{p}& =\frac{1}{c^{\frac{p}{p-1}}}\Big[\eta\|g_{\eta}\|^{p}_{L^{p}_{\theta}}+\frac{1}{\mu_{\alpha,\theta}}\ln \frac{\omega_{\theta}}{\theta+1}-\frac{\theta+1}{\mu_{\alpha,\theta}}\ln\epsilon\\
& +\mathcal{A}_{\eta}-\frac{p-1}{\mu_{\alpha,\theta}}[\gamma+\Psi(p)]+O\left(L^{-\frac{\theta+1}{p-1}}\right)\Big].\\
\end{aligned}
\end{equation}
For $\epsilon>0$ small enough, we can choose $b$ in \eqref{c-choice} of the form 
\begin{equation}\label{b-choice}
\begin{aligned}
b=& -\eta\|g_{\eta}\|^{p}_{L^{p}_{\theta}}+\frac{p-1}{\mu_{\alpha,\theta}}[\gamma+\Psi(p)]+O\Big(L^{-\frac{\theta+1}{p-1}}\Big)
\end{aligned}
\end{equation}
such that $c^{\frac{p}{p-1}}$ satisfies the equation \eqref{c-choice} and  $\|v_{\epsilon}\|=1$. In addition,  $c$ satisfies
\begin{equation}\label{c-choice2}
\begin{aligned}
c^{\frac{p}{p-1}}& =\eta\|g_{\eta}\|^{p}_{L^{p}_{\theta}}+\frac{1}{\mu_{\alpha,\theta}}\ln \frac{\omega_{\theta}}{\theta+1}-\frac{\theta+1}{\mu_{\alpha,\theta}}\ln\epsilon+\mathcal{A}_{\eta}-\frac{p-1}{\mu_{\alpha,\theta}}[\gamma+\Psi(p)]+O\Big(L^{-\frac{\theta+1}{p-1}}\Big).
\end{aligned}
\end{equation}
Now, for any $q>0$
\begin{equation}\nonumber
\begin{aligned}
 \left(1+\eta\|v_{\epsilon}\|^{p}_{L^{p}_{\theta}}\right)^{q}&=1+q\eta\|v_{\epsilon}\|^{p}_{L^{p}_{\theta}}+\frac{q(q-1)}{2!}(\eta\|v_{\epsilon}\|^{p}_{L^{p}_{\theta}})^{2}+R(\|v_{\epsilon}\|^{p}_{L^{p}_{\theta}})
\end{aligned}
\end{equation}
where, for some $0<\tau_{\epsilon}<1$,
\begin{equation}\nonumber
R(\|v_{\epsilon}\|^{p}_{L^{p}_{\theta}})=\frac{q(q-1)(q-2)}{3!}\left(1+\tau_{\epsilon}\eta\|v_{\epsilon}\|^{p}_{L^{p}_{\theta}}\right)^{q-3}(\eta\|v_{\epsilon}\|^{p}_{L^{p}_{\theta}})^{3}.
\end{equation}
Thus, taking into account \eqref{lp-test}, \eqref{c-choice2}  and  $L=-\ln\epsilon$ we can see that 
\begin{equation}\label{term-druet}
\begin{aligned}
 \left(1+\eta\|v_{\epsilon}\|^{p}_{L^{p}_{\theta}}\right)^{\frac{1}{p-1}}&=1+\frac{1}{p-1}\frac{\eta \|g_{\eta}\|^{p}_{L^{p}_{\theta}}}{c^{\frac{p}{p-1}}}-\frac{p-2}{2(p-1)^2}\frac{\eta^{2} \|g_{\eta}\|^{2p}_{L^{p}_{\theta}}}{c^{\frac{2p}{p-1}}} +O\Big(c^{-\frac{3p}{p-1}}\Big).
\end{aligned}
\end{equation}
Since $\varphi_{p}(t)\ge t^{k_0}/k_0!,$ for $t\ge 0$,  by using \eqref{test-foraLq} we can write
\begin{equation}\nonumber
\begin{aligned}
&\int_{\epsilon L}^{\infty} \varphi_{p}\big(\mu_{\alpha,\theta}\big(1+\eta\|v_{\epsilon}\|^{p}_{L^{p}_{\theta}}\big)^{\frac{1}{p-1}}\vert v_{\epsilon}\vert^{\frac{p}{p-1}}\big)\mathrm{d}\lambda_{\theta}\\
&\ge \frac{\mu^{k_0}_{\alpha,\theta}}{k_0
!}\frac{1}{c^{\frac{pk_0}{(p-1)^2}}} \left(1+\eta\|v_{\epsilon}\|^{p}_{L^{p}_{\theta}}\right)\Big[\|g_{\eta}\|^{\frac{k_0p}{p-1}}_{L^{\frac{pk_0}{p-1}}_{\theta}}+O\Big((\epsilon L)^{\theta+1}\vert \ln(\epsilon L)\vert^{{\frac{k_0p}{p-1}}}\Big)\Big].
\end{aligned}
\end{equation}
From  \eqref{lp-test},  $\|v_{\epsilon}\|^{p}_{L^{p}_{\theta}}=O(c^{-\frac{p}{p-1}})$ and consequently
\begin{equation}\label{cota-fora}
\begin{aligned}
&\int_{\epsilon L}^{\infty} \varphi_{p}\big(\mu_{\alpha,\theta}\big(1+\eta\|v_{\epsilon}\|^{p}_{L^{p}_{\theta}}\big)^{\frac{1}{p-1}}\vert v_{\epsilon}\vert^{\frac{p}{p-1}}\big)\mathrm{d}\lambda_{\theta}\\
&\ge \frac{1}{c^{\frac{pk_0}{(p-1)^2}}}\Big[\frac{\mu^{k_0}_{\alpha,\theta}}{k_0!}\|g_{\eta}\|^{\frac{k_0p}{p-1}}_{L^{\frac{pk_0}{p-1}}_{\theta}}+O\Big((\epsilon L)^{\theta+1}\vert \ln(\epsilon L)\vert^{\frac{k_0p}{p-1}}\Big)+O\left(c^{-\frac{p}{p-1}}\right)\Big].
\end{aligned}
\end{equation}
Note that for $\epsilon>0$ small enough, we have 
\begin{equation}\nonumber
c^{-\frac{p}{p-1}}\Big[-\frac{p-1}{\mu_{\alpha,\theta}}\ln\Big(1+c_{\alpha,\theta}\Big(\frac{r}{\epsilon}\Big)^{\frac{\theta+1}{p-1}}\Big)+b\Big]>-1,
\end{equation}
for any  $r\in (0,\epsilon L]$. Hence, the inequality $(1+t)^{d}\ge 1+dt$ for $t>-1$ and $d\in (1, 2]$ yields
\begin{equation}\nonumber
\begin{aligned}
\Big\vert \frac{v_{\epsilon}}{c}\Big\vert^{\frac{p}{p-1}}& \ge 1+\frac{p}{p-1}\frac{1}{c^{\frac{p}{p-1}}}\Big[w\left(\frac{r}{\epsilon}\right)+b\Big],
\end{aligned}
\end{equation}
where $w$ is given by \eqref{w-express}. Using \eqref{b-choice} 
\begin{equation}\label{term-Moser}
\begin{aligned}
 \vert v_{\epsilon}\vert^{\frac{p}{p-1}}& \ge c^{\frac{p}{p-1}}-\frac{p\eta}{p-1}\|g_{\eta}\|^{p}_{L^{p}_{\theta}}+\frac{p}{\mu_{\alpha,\theta}}[\gamma+\Psi(p)]+\frac{p}{p-1}w(r/\epsilon)+O\Big(L^{-\frac{\theta+1}{p-1}}\Big).
\end{aligned}
\end{equation}
From \eqref{term-druet}, we have for any $r\in (0,\epsilon L]$
\begin{equation}\label{two-pieces}
\begin{aligned}
  \left(1+\eta\|v_{\epsilon}\|^{p}_{L^{p}_{\theta}}\right)^{\frac{1}{p-1}}\vert v_{\epsilon}\vert^{\frac{p}{p-1}} & = \Big(1+\frac{\eta \|g_{\eta}\|^{p}_{L^{p}_{\theta}}}{(p-1)c^{\frac{p}{p-1}}}\Big)\vert v_{\epsilon}\vert^{\frac{p}{p-1}}\\
 & + \Big[\frac{2-p}{2(p-1)^2}\frac{\eta^{2} \|g_{\eta}\|^{2p}_{L^{p}_{\theta}}}{c^{\frac{2p}{p-1}}} +O\left(c^{-\frac{3p}{p-1}}\right)\Big]\vert v_{\epsilon}\vert^{\frac{p}{p-1}}.
\end{aligned}
\end{equation}
Next, we shall estimate each term on the right hand side of \eqref{two-pieces}. Firstly,  from  \eqref{term-Moser}
\begin{equation}\nonumber
\begin{aligned}
 \Big(1+\frac{\eta \|g_{\eta}\|^{p}_{L^{p}_{\theta}}}{(p-1)c^{\frac{p}{p-1}}}\Big)\vert v_{\epsilon}\vert ^{\frac{p}{p-1}} &\ge  \Big(1+\frac{\eta \|g_{\eta}\|^{p}_{L^{p}_{\theta}}}{(p-1)c^{\frac{p}{p-1}}}\Big)\frac{p}{p-1}w(r/\epsilon)\\ 
 &+c^{\frac{p}{p-1}}-\eta \|g_{\eta}\|^{p}_{L^{p}_{\theta}}-\frac{p\eta^{2} \|g_{\eta}\|^{2p}_{L^{p}_{\theta}}}{(p-1)^{2}c^{\frac{p}{p-1}}}+ \frac{p}{\mu_{\alpha,\theta}}[\gamma+\Psi(p)]\\ 
 & +\frac{p}{\mu_{\alpha,\theta}}[\gamma+\Psi(p)]\frac{\eta \|g_{\eta}\|^{p}_{L^{p}_{\theta}}}{(p-1)c^{\frac{p}{p-1}}}+O\left(L^{-\frac{\theta+1}{p-1}}\right).\\
\end{aligned}
\end{equation}
Using the identity in  \eqref{c-choice2}, we can also write
\begin{equation}\label{pieceI}
\begin{aligned}
& \Big(1 +\frac{\eta \|g_{\eta}\|^{p}_{L^{p}_{\theta}}}{(p-1)c^{\frac{p}{p-1}}}\Big)\vert v_{\epsilon}\vert^{\frac{p}{p-1}} \ge  \Big(1+\frac{\eta \|g_{\eta}\|^{p}_{L^{p}_{\theta}}}{(p-1)c^{\frac{p}{p-1}}}\Big)\frac{p}{p-1}w(r/\epsilon)\\ 
 &\quad + \frac{1}{\mu_{\alpha,\theta}}\ln \frac{\omega_{\theta}}{\theta+1}-\frac{\theta+1}{\mu_{\alpha,\theta}}\ln\epsilon+\mathcal{A}_{\eta}+\frac{1}{\mu_{\alpha,\theta}}[\gamma+\Psi(p)]-\frac{p\eta^{2} \|g_{\eta}\|^{2p}_{L^{p}_{\theta}}}{(p-1)^{2}c^{\frac{p}{p-1}}}\\ 
 & \quad +\frac{p}{\mu_{\alpha,\theta}}[\gamma+\Psi(p)]\frac{\eta \|g_{\eta}\|^{p}_{L^{p}_{\theta}}}{(p-1)c^{\frac{p}{p-1}}}+O\left(L^{-\frac{\theta+1}{p-1}}\right).\\
\end{aligned}
\end{equation}
On the other hand, from  \eqref{term-Moser} it is easy to see that
\begin{equation}\label{pieceII}
\begin{aligned}
 \Big(\frac{2-p}{2(p-1)^2}\frac{\eta^{2} \|g_{\eta}\|^{2p}_{L^{p}_{\theta}}}{c^{\frac{2p}{p-1}}} +O\left(c^{-\frac{3p}{p-1}}\Big)\right)\vert v_{\epsilon}\vert^{\frac{p}{p-1}} 
&=\frac{1}{c^{\frac{p}{p-1}}}\Phi_{\epsilon}
\end{aligned}
\end{equation}
where 
\begin{equation}\nonumber
\begin{aligned}
\Phi_{\epsilon}&=\Big[\frac{2-p}{2(p-1)^2}\eta^{2} \|g_{\eta}\|^{2p}_{L^{p}_{\theta}}+O\left(c^{-\frac{p}{p-1}}\right)\Big]\\
& \times\Big[1-\frac{p\eta\|g_{\eta}\|^{p}_{L^{p}_{\theta}}}{(p-1)c^{\frac{p}{p-1}}}+\frac{p[\gamma+\Psi(p)]}{\mu_{\alpha,\theta}c^{\frac{p}{p-1}}}+\frac{pw(r/\epsilon)}{(p-1)c^{\frac{p}{p-1}}}+O\Big(L^{-\frac{\theta+1}{p-1}}\Big)\Big].\\
\end{aligned}
\end{equation}
Note that  (recall $\theta\ge \alpha=p-1$)
\begin{equation}\label{pieceII-limit}
\begin{aligned}
\lim_{\epsilon\rightarrow 0}\Big[\Phi_{\epsilon} &+\frac{p^2-p+2}{2(p-1)^{2}}\eta^{2} \|g_{\eta}\|^{2p}_{L^{p}_{\theta}}+O\Big(\frac{c^{\frac{p}{p-1}}}{L^{\frac{\theta+1}{p-1}}}\Big)\Big]\\
&=\left\{\begin{aligned}
& (p-2)^2\frac{\eta^2 \|g_{\eta}\|^{2p}_{L^{p}_{\theta}}}{2(p-1)^2}>0, \;\;&\mbox{if}&\;\; p>2\\
& 2\eta^2\|g_{\eta}\|^{4}_{L^{2}_{\theta}}>0, \;\;&\mbox{if}&\;\; p=2.
\end{aligned}\right.
\end{aligned}
\end{equation}
For $\epsilon>0$ small enough, combining \eqref{two-pieces}, \eqref{pieceI}, \eqref{pieceII} and \eqref{pieceII-limit}, we have 
\begin{equation}\label{term-druet-Moser}
\begin{aligned}
  \left(1+\eta\|v_{\epsilon}\|^{p}_{L^{p}_{\theta}}\right)^{\frac{1}{p-1}} &\vert v_{\epsilon}\vert^{\frac{p}{p-1}} \ge\Big(1+\frac{\eta \|g_{\eta}\|^{p}_{L^{p}_{\theta}}}{(p-1)c^{\frac{p}{p-1}}}\Big)\frac{p}{p-1}w(r/\epsilon) \\
& + \frac{1}{\mu_{\alpha,\theta}}\ln \frac{\omega_{\theta}}{\theta+1}-\frac{\theta+1}{\mu_{\alpha,\theta}}\ln\epsilon+\mathcal{A}_{\eta}+\frac{1}{\mu_{\alpha,\theta}}[\gamma+\Psi(p)]\\
& +\frac{1}{\mu_{\alpha,\theta}}[\gamma+\Psi(p)]\frac{p\eta \|g_{\eta}\|^{p}_{L^{p}_{\theta}}}{(p-1)c^{\frac{p}{p-1}}}-\frac{p^2-p+2}{2(p-1)^{2}}\frac{\eta^{2} \|g_{\eta}\|^{2p}_{L^{p}_{\theta}}}{c^{\frac{p}{p-1}}}.
\end{aligned}
\end{equation}
From \eqref{exp-frac},  \eqref{test-dentro},\eqref{lp-test}, \eqref{c-choice2} and $L=-\ln \epsilon$, we have
\begin{equation}\label{cota-dentro 2part}
\begin{aligned}
&  \int_{0}^{\epsilon L} \varphi_{p}\big(\mu_{\alpha,\theta}\big(1 +\eta\|v_{\epsilon}\|^{p}_{L^{p}_{\theta}}\big)^{\frac{1}{p-1}}\vert v_{\epsilon}\vert^{\frac{p}{p-1}}\big)\mathrm{d}\lambda_{\theta} \\
 &\quad\quad\quad= \int_{0}^{\epsilon L}e^{\mu_{\alpha,\theta}\big(1+\eta\|v_{\epsilon}\|^{p}_{L^{p}_{\theta}}\big)^{\frac{1}{p-1}}\vert v_{\epsilon}\vert^{\frac{p}{p-1}}}\mathrm{d}\lambda_{\theta}+O\Big(L^{-\frac{\theta+1}{p-1}}\Big).
\end{aligned}
\end{equation}
By simplicity, set 
$$
Y_{p,\eta}=\frac{1}{c^{\frac{p}{p-1}}}\frac{p\eta}{p-1} \|g_{\eta}\|^{p}_{L^{p}_{\theta}}.
$$
With this notation, from \eqref{LT-eq3} and \eqref{GPsi} we have 
\begin{equation}\nonumber
\begin{aligned}
&\frac{1}{\epsilon^{\theta+1}} \int_{0}^{\epsilon L}e^{\mu_{\alpha,\theta}\Big(1+\frac{\eta \|g_{\eta}\|^{p}_{L^{p}_{\theta}}}{(p-1)c^{\frac{p}{p-1}}}\Big)\frac{p}{p-1}w(r/\epsilon)}\mathrm{d}\lambda_{\theta} =(p-1) \int_{0}^{z_L}\frac{s^{p-2}}{(1+s)^{p+Y_{p,\eta}}}ds\\
 &=\frac{\Gamma(p)\Gamma\left(1+Y_{p,\eta}\right)}{\Gamma\left(p+Y_{p,\eta}\right)}
 -(p-1) \int_{z_{L}}^{\infty}\frac{s^{p-2}}{(1+s)^{p+Y_{p,\eta}}}ds\\
 &=1-[\Psi(p)+\gamma]Y_{p,\eta}+O\left(c^{-\frac{2p}{p-1}}\right) + O\Big(L^{-\frac{\theta+1}{p-1}}\Big).
\end{aligned}
\end{equation}
This together with \eqref{term-druet-Moser} and the inequality $e^{x}\ge 1+x$ yields
\begin{equation}\nonumber
\begin{aligned}
&  \int_{0}^{\epsilon L}e^{\mu_{\alpha,\theta}\big(1+\eta\|v_{\epsilon}\|^{p}_{L^{p}_{\theta}}\big)^{\frac{1}{p-1}}\vert v_{\epsilon}\vert^{\frac{p}{p-1}}}\mathrm{d}\lambda_{\theta}\\
&\ge  \frac{\omega_{\theta}}{\theta+1}e^{\mu_{\alpha,\theta}\mathcal{A}_{\eta}+\gamma+\Psi(p)}\Big[1-\frac{p^2-p+2}{2(p-1)^{2}}\frac{\mu_{\alpha,\theta}\eta^{2} \|g_{\eta}\|^{2p}_{L^{p}_{\theta}}}{c^{\frac{p}{p-1}}}+O\left(c^{-\frac{2p}{p-1}}\right) + O\left(L^{-\frac{\theta+1}{p-1}}\right)\\
& +[\gamma+\Psi(p)]\frac{p^2-p+2}{2(p-1)^2}\frac{\mu_{\alpha,\theta}\eta^2\|g_{\eta}\|^{2p}_{L^{p}_{\theta}}}{c^{\frac{p}{p-1}}}Y_{p,\eta}\Big].
\end{aligned}
\end{equation}
Hence, 
\begin{equation}\label{cota-dentro part1}
\begin{aligned}
&  \int_{0}^{\epsilon L}e^{\mu_{\alpha,\theta}\big(1+\eta\|v_{\epsilon}\|^{p}_{L^{p}_{\theta}}\big)^{\frac{1}{p-1}}\vert v_{\epsilon}\vert^{\frac{p}{p-1}}}\mathrm{d}\lambda_{\theta}
 \ge \frac{\omega_{\theta}}{\theta+1}e^{\mu_{\alpha,\theta}\mathcal{A}_{\eta}+\gamma+\Psi(p)}\\
&+ \frac{\omega_{\theta}}{\theta+1}e^{\mu_{\alpha,\theta}\mathcal{A}_{\eta}+\gamma+\Psi(p)}\frac{\eta\mu_{\alpha,\theta} \|g_{\eta}\|^{p}_{L^{p}_{\theta}}}{c^{\frac{p}{p-1}}}\Big[\left([\gamma+\Psi(p)]Y_{p,\eta}-1\right)\frac{p^2-p+2}{2(p-1)^{2}}\eta \|g_{\eta}\|^{p}_{L^{p}_{\theta}}\\
& +O\left(c^{-\frac{2p}{p-1}}\right)+ O\left(L^{-\frac{\theta+1}{p-1}}\right)\Big].
\end{aligned}
\end{equation}
Now, by using \eqref{cota-fora}, \eqref{cota-dentro 2part}, \eqref{cota-dentro part1} we obtain
\begin{equation}\label{cota-final}
\begin{aligned}
&  \int_{0}^{\infty}e^{\mu_{\alpha,\theta}\big(1+\eta\|v_{\epsilon}\|^{p}_{L^{p}_{\theta}}\big)^{\frac{1}{p-1}}\vert v_{\epsilon}\vert^{\frac{p}{p-1}}}\mathrm{d}\lambda_{\theta}
 \ge \frac{\omega_{\theta}}{\theta+1}e^{\mu_{\alpha,\theta}\mathcal{A}_{\eta}+\gamma+\Psi(p)}+\frac{\|g_{\eta}\|^{p}_{L^{p}_{\theta}}}{c^{\frac{k_0p}{(p-1)^2}}}H(\epsilon,\eta),
\end{aligned}
\end{equation}
where
\begin{equation}\nonumber
\begin{aligned}
 & H(\epsilon,\eta)=\frac{\mu^{k_0}_{\alpha,\theta}}{k_0!}\frac{\|g_{\eta}\|^{\frac{k_0p}{p-1}}_{L^{\frac{pk_0}{p-1}}_{\theta}}}{\|g_{\eta}\|^{p}_{L^{p}_{\theta}}}+\frac{1}{c^{\frac{p}{p-1}}\|g_{\eta}\|^{p}_{L^{p}_{\theta}}}\Big[O\Big(c^{\frac{p}{p-1}}(\epsilon L)^{\theta+1}\ln^{\frac{k_0p}{p-1}}(\epsilon L)\Big)+O\left(1\right)\Big]\\
& +\frac{\omega_{\theta}\eta\mu_{\alpha,\theta}}{\theta+1}e^{\mu_{\alpha,\theta}\mathcal{A}_{\eta}+\gamma+\Psi(p)}c^{\frac{p}{p-1}(\frac{k_0}{p-1}-1)}\left([\gamma+\Psi(p)]Y_{p,\eta}-1\right)\frac{p^2-p+2}{2(p-1)^{2}}\eta \|g_{\eta}\|^{p}_{L^{p}_{\theta}} \\
& + O\Big(c^{\frac{k_0p}{(p-1)^2}-\frac{2p}{p-1}}\Big)+ O\Big(c^{\frac{k_0p}{(p-1)^2}}L^{-\frac{\theta+1}{p-1}}\Big).
\end{aligned}
\end{equation}
Now, it is sufficient to show that $H(\epsilon,\eta)>0$, for $\epsilon>0, \eta\ge 0$ small.
\paragraph{\textbf{Case 1:}} $2\le p\in\mathbb{R}$,  $\eta =0$ and $\epsilon>0$ small.
In this case we have 
\begin{equation}\nonumber
\begin{aligned}
H(\epsilon,0)&=\frac{\mu^{k_0}_{\alpha,\theta}}{k_0!}\frac{\|g_{0}\|^{\frac{k_0p}{p-1}}_{L^{\frac{pk_0}{p-1}}_{\theta}}}{\|g_{0}\|^{p}_{L^{p}_{\theta}}}+\frac{1}{c^{\frac{p}{p-1}}\|g_{0}\|^{p}_{L^{p}_{\theta}}}\Big[O\Big(c^{\frac{p}{p-1}}(\epsilon L)^{\theta+1}\ln^{\frac{k_0p}{p-1}}(\epsilon L)\Big)+O\left(1\right)\Big]\\
& +O\Big(c^{\frac{k_0p}{(p-1)^2}-\frac{2p}{p-1}}\Big)+ O\Big(c^{\frac{k_0p}{(p-1)^2}}L^{-\frac{\theta+1}{p-1}}\Big).
\end{aligned}
\end{equation}
Noting that $p-1\le k_0<p$ and by using \eqref{c-choice2} with $L=-\ln \epsilon$  we can see that
\begin{equation}\label{O-limits}
c^{\frac{p}{p-1}}\rightarrow\infty, \; c^{\frac{k_0p}{(p-1)^2}-\frac{2p}{p-1}}\rightarrow 0,\; c^{\frac{k_0p}{(p-1)^2}}L^{-\frac{\theta+1}{p-1}}\rightarrow 0,\; c^{\frac{p}{p-1}}(\epsilon L)^{\theta+1}\ln^{\frac{k_0p}{p-1}}(\epsilon L)\rightarrow 0,
\end{equation}
as $\epsilon\rightarrow 0$.
Thus,  we obtain $H(\epsilon,0)>0$ for $\epsilon>0$ small enough.
\paragraph{\textbf{Case 2:}} $2\le p\in\mathbb{N}$,  $\epsilon>0$ and $\eta>0$ small. Let $g_0$ be the solution of the equation \eqref{g-trueeq}, that is, 
\begin{equation}\nonumber
\omega_{\alpha}r^{\alpha}\vert g_{0}^{\prime}(r)\vert^{p-1}+\int_{0}^{r}\vert g_{0}\vert^{p-1}\mathrm{d}\lambda_{\theta}=1,\quad r>0.
\end{equation}
A simple scaling argument shows that $g_{\eta}(r)=g_{0}((1-\eta)^{1/(\theta+1)}r)$ . Hence, we have 
\begin{equation}\label{gg0}
A_{\eta}=A_0-\frac{1}{\mu_{\alpha,\theta}}\ln(1-\eta)\quad \mbox{and}\quad \|g_{\eta}\|^{q}_{L^{q}_{\theta}}=(1-\eta)^{-1}\|g_{0}\|^{q}_{L^{q}_{\theta}},\quad q\ge p.
\end{equation}
Now, since $p\in\mathbb{N}$, we have  $k_0=p-1$.  Thus,  by using \eqref{gg0}
\begin{equation}\nonumber
\begin{aligned}
& H(\epsilon,\eta)=\frac{\mu^{p-1}_{\alpha,\theta}}{(p-1)!}+\frac{1-\eta}{c^{\frac{p}{p-1}}\|g_{0}\|^{p}_{L^{p}_{\theta}}}\Big[O\left(c^{\frac{p}{p-1}}(\epsilon L)^{\theta+1}\vert \ln(\epsilon L)\vert^{p}\right)+O\left(1\right)\Big]\\
& +\Big[\frac{\eta}{1-\eta}\frac{\omega_{\theta}\mu_{\alpha,\theta}}{\theta+1}e^{\mu_{\alpha,\theta}\mathcal{A}_{0}+\gamma+\Psi(p)}\Big(\frac{p}{p-1}[\gamma+\Psi(p)]\frac{\eta}{1-\eta}\frac{\|g_{0}\|^{p}_{L^{p}_{\theta}}}{c^{\frac{p}{p-1}}}-1\Big)\Big]\times\\
&\quad\;\Big[\frac{p^2-p+2}{2(p-1)^{2}}\frac{\eta}{1-\eta} \|g_{0}\|^{p}_{L^{p}_{\theta}}\Big]+O\left(c^{-\frac{p}{p-1}}\right)+ O\left(c^{\frac{p}{p-1}}L^{-\frac{\theta+1}{p-1}}\right).
\end{aligned}
\end{equation}
From this, we conclude that $H(\epsilon,\eta)>0$, if $\epsilon,\eta>0$ is small enough.
\paragraph{\textbf{Case 3:}} $2\le p\not\in\mathbb{N}$,  $\epsilon>0$ and $\eta>0$ small. Here $p-1<k_0<p$ and, from  \eqref{c-choice2} and \eqref{gg0}
\begin{equation}\nonumber
\begin{aligned}
\eta c^{\frac{p}{p-1}(\frac{k_0}{p-1}-1)}& =\Big[\eta^{\frac{1}{\frac{k_0}{p-1}-1}}\frac{\eta}{1-\eta}\|g_{0}\|^{p}_{L^{p}_{\theta}}+\eta^{\frac{1}{\frac{k_0}{p-1}-1}}\frac{1}{\mu_{\alpha,\theta}}\ln \frac{\omega_{\theta}}{\theta+1}\\
&-\frac{\theta+1}{\mu_{\alpha,\theta}}\eta^{\frac{1}{\frac{k_0}{p-1}-1}}\ln\epsilon+\eta^{\frac{1}{\frac{k_0}{p-1}-1}}\mathcal{A}_{0}-\frac{1}{\mu_{\alpha,\theta}}\eta^{\frac{1}{\frac{k_0}{p-1}-1}}\ln(1-\eta)\\
&-\eta^{\frac{1}{\frac{k_0}{p-1}-1}}\frac{p-1}{\mu_{\alpha,\theta}}[\gamma+\Psi(p)]+O\left(L^{-\frac{\theta+1}{p-1}}\right)\Big]^{\frac{k_0}{p-1}-1}.
\end{aligned}
\end{equation}
Hence, by choosing $\epsilon=\eta>0$  we have 
\begin{equation}\label{c-control}
\eta c^{\frac{p}{p-1}(\frac{k_0}{p-1}-1)}\rightarrow 0, \quad \mbox{as}\;\; \epsilon=\eta\rightarrow 0.
\end{equation}
Moreover, using \eqref{gg0}
\begin{equation}\nonumber
\begin{aligned}
& H(\epsilon,\eta)=\frac{\mu^{k_0}_{\alpha,\theta}}{k_0!}\frac{\|g_{0}\|^{\frac{k_0p}{p-1}}_{L^{\frac{pk_0}{p-1}}_{\theta}}}{\|g_{0}\|^{p}_{L^{p}_{\theta}}}+\frac{1-\eta}{c^{\frac{p}{p-1}}\|g_{0}\|^{p}_{L^{p}_{\theta}}}\Big[O\Big(c^{\frac{p}{p-1}}(\epsilon L)^{\theta+1}\ln^{\frac{k_0p}{p-1}}(\epsilon L)\Big)+O\left(1\right)\Big]\\
&+\Big[\frac{\eta^2c^{\frac{p}{p-1}(\frac{k_0}{p-1}-1)}}{(1-\eta)^2}\frac{\omega_{\theta}\mu_{\alpha,\theta}}{\theta+1}e^{\mu_{\alpha,\theta}\mathcal{A}_{0}+\gamma+\Psi(p)}\Big]\times\\
& \quad\;\Big[\Big(\frac{p(\gamma+\Psi(p))}{p-1}\frac{\eta\|g_{0}\|^{p}_{L^{p}_{\theta}}}{c^{\frac{p}{p-1}}(1-\eta)}-1\Big)\frac{p^2-p+2}{2(p-1)^{2}} \|g_{0}\|^{p}_{L^{p}_{\theta}}\Big] \\
& \quad+O\Big(c^{\frac{k_0p}{(p-1)^2}-\frac{2p}{p-1}}\Big)+ O\Big(c^{\frac{k_0p}{(p-1)^2}}L^{-\frac{\theta+1}{p-1}}\Big).\\
\end{aligned}
\end{equation}
Thus, using the convergences in \eqref{O-limits} and \eqref{c-control} we have $H(\epsilon,\eta)>0$,  if $\epsilon=\eta>0$ is small enough.
\end{proof}
\section{Proof of Theorem~\ref{thm non-existence}: Non-existence of maximizers} 
\label{sec-non-existence}
We follow the argument of Ishiwata of \cite{Ishi}. For $p=2$, \cite[Theorem~1.4]{JJ2012} yields
\begin{equation}\nonumber
\mathrm{C}_{\mu,\theta}=\sup_{0\not\equiv u\in X^{1,2}_{\infty}}\frac{\|u^{\prime}\|^{2}_{L^{2}_{1}}}{\|u\|^{2}_{L^{2}_{\theta}}}\int_{0}^{\infty}\Big(e^{\mu \frac{u^{2}}{\|u^{\prime}\|^2_{L^{2}_{1}}}}-1\Big)\mathrm{d}\lambda_{\theta}<\infty,
\end{equation}
for any $0<\mu<\mu_{1,\theta}=2\pi(\theta+1)$. Thus, the series expansion of $x\mapsto e^x-1$ yields
\begin{equation}\label{p=2AT}
\|u\|^{2j}_{L^{2}_{\theta}}\le \mathrm{C}_{\mu,\theta}\frac{j!}{\mu^{j}}\|u^{\prime}\|^{2(j-1)}_{L^{2}_{1}}\|u\|^{2}_{L^{2}_{\theta}},\;\; j\ge 1.
\end{equation}
Let $\mathcal{M}=\{u\in X^{1,2}_{\infty}\;\vert \; \|u\|=1\}$. For any  $u\in\mathcal{M}$, we set
\begin{equation}\nonumber
u_{\tau}(r)=\tau^{\frac{1}{2}}u(\tau^{\frac{1}{\theta+1}}r),\quad v_{\tau}= \frac{u_{\tau}}{\|u_{\tau}\|},\;\; \tau>0
\end{equation}
Hence $v_{\tau}\in\mathcal{M}$. Define
\begin{equation}\nonumber
J(u)=\int_{0}^{\infty}\Big(e^{\mu(1+\eta\|u\|^{2}_{L^{2}_{\theta}})u^{2}}-1\Big)\mathrm{d}\lambda_{\theta}.
\end{equation}
If $u$ is a maximizer for $ AD(\eta, \mu, 1,\theta)$  then $u\in\mathcal{M}.$ Since $v_{\tau}$ is a curve in  $\mathcal{M}$  with $v_{1}=u$ we must have
\begin{equation}\label{cri-max}
\frac{d}{d\tau}J(v_{\tau})\big\vert_{\tau=1}=0. 
\end{equation}
From \eqref{rts-LpLq}, we can write
\begin{equation}\nonumber
J(v_{\tau})=\sum_{j=1}^{\infty} \frac{\mu^{j}}{j!}\Big(1+\eta \frac{\|u\|^{2}_{L^{2}_{\theta}}}{\tau \|u^{\prime}\|^{2}_{L^{2}_{1}}+\|u\|^{2}_{L^{2}_{\theta}}}\Big)^{j}\frac{\tau^{j-1}\|u\|^{2j}_{L^{2j}_{\theta}}}{\left(\tau \|u^{\prime}\|^{2}_{L^{2}_{1}}+\|u\|^{2}_{L^{2}_{\theta}}\right)^{j}}.
\end{equation}
Since $\|u\|=1$, it follows that
\begin{equation}\nonumber
\begin{aligned}
\frac{d}{d\tau}J(v_{\tau})\Big\vert_{\tau=1} 
& = -\mu \|u\|^{2}_{L^{2}_{\theta}}\|u^{\prime}\|^{2}_{L^{2}_{1}}\\
&+ \sum_{j=2}^{\infty}\frac{\mu^{j}}{j!}\left(1+\eta\|u\|^{2}_{L^{2}_{\theta}}\right)^{j-1}\|u\|^{2j}_{L^{2j}_{\theta}} \Big[-j \eta \|u\|^{2}_{L^{2}_{\theta}} \|u^{\prime}\|^{2}_{L^{2}_{1}}\\
& + \big(1+\eta\|u\|^{2}_{L^{2}_{\theta}}\big)\big(j-1 -j\|u^{\prime}\|^{2}_{L^{2}_{1}}\big)\Big].\\
& \le  -\mu \|u\|^{2}_{L^{2}_{\theta}}\|u^{\prime}\|^{2}_{L^{2}_{1}}+ \sum_{j=2}^{\infty}\frac{(2\mu)^{j}}{(j-1)!}\|u\|^{2j}_{L^{2j}_{\theta}}.
\end{aligned}
\end{equation}
For any $0<\gamma<\mu_{1,\theta}$, \eqref{p=2AT} yields
\begin{equation}\nonumber
\begin{aligned}
\frac{d}{d\tau}J(v_{\tau})\Big\vert_{\tau=1} &
 \le  \|u\|^{2}_{L^{2}_{\theta}}\|u^{\prime}\|^{2}_{L^{2}_{1}}\Big(-\mu+ \frac{4\mu^{2}}{\gamma^2}\mathrm{C}_{\gamma,\theta}\sum_{j=2}^{\infty} \frac{j(2\mu)^{j-2}}{\gamma^{j-2}}\|u^{\prime}\|^{2(j-2)}_{L^{2}_{1}}\Big)\\
&\le \mu \|u\|^{2}_{L^{2}_{\theta}}\|u^{\prime}\|^{2}_{L^{2}_{1}}\Big(-1+ \frac{4\mu}{\gamma^2}\mathrm{C}_{\gamma,\theta}\sum_{j=0}^{\infty} \frac{(j+2)(2\mu)^{j}}{\gamma^{j}}\Big).
\end{aligned}
\end{equation}
Thus, for $\mu<\mu_{1,\theta}/4$ and by choosing $\overline{\gamma}=3\mu_{1,\theta}/4$, we get
\begin{equation}\nonumber
\begin{aligned}
 \frac{d}{d\tau}J(v_{\tau})\Big\vert_{\tau=1} 
&\le  \mu \|u\|^{2}_{L^{2}_{\theta}}\|u^{\prime}\|^{2}_{L^{2}_{1}}\left(-1+ \mu C\right), \;\;\mbox{with}\;\; C=\frac{4}{\overline{\gamma}^2}\mathrm{C}_{\overline{\gamma},\theta}\sum_{j=0}^{\infty} (j+2)\frac{2^{j}}{3^j}.
\end{aligned}
\end{equation}
But, for any $\mu<\min\left\{\mu_{1,\theta}/4, 1/C\right\}$, this contradicts \eqref{cri-max}.

\bigskip
\bigskip
\end{document}